\documentclass[oneside,american]{amsart}
\usepackage[T1]{fontenc}
\usepackage[latin9]{inputenc}
\usepackage{babel}
\usepackage{verbatim}
\usepackage{refstyle}
\usepackage{units}
\usepackage{mathrsfs}
\usepackage{enumitem}
\usepackage{amstext}
\usepackage{amsthm}
\usepackage{amssymb}
\usepackage[unicode=true,pdfusetitle,
 bookmarks=true,bookmarksnumbered=false,bookmarksopen=false,
 breaklinks=false,pdfborder={0 0 1},backref=false,colorlinks=false]
 {hyperref}
\usepackage{breakurl}

\makeatletter

\AtBeginDocument{\providecommand\thmref[1]{\ref{thm:#1}}}
\AtBeginDocument{\providecommand\lemref[1]{\ref{lem:#1}}}
\AtBeginDocument{\providecommand\propref[1]{\ref{prop:#1}}}
\AtBeginDocument{\providecommand\enuref[1]{\ref{enu:#1}}}
\AtBeginDocument{\providecommand\corref[1]{\ref{cor:#1}}}
\RS@ifundefined{subsecref}
  {\newref{subsec}{name = \RSsectxt}}
  {}
\RS@ifundefined{thmref}
  {\def\RSthmtxt{theorem~}\newref{thm}{name = \RSthmtxt}}
  {}
\RS@ifundefined{lemref}
  {\def\RSlemtxt{lemma~}\newref{lem}{name = \RSlemtxt}}
  {}

\numberwithin{equation}{section}
\numberwithin{figure}{section}
\theoremstyle{plain}
\newtheorem{thm}{\protect\theoremname}[section]
  \theoremstyle{plain}
  \newtheorem*{conjecture*}{\protect\conjecturename}
 \theoremstyle{definition}
 \newtheorem*{defn*}{\protect\definitionname}
  \theoremstyle{definition}
  \newtheorem{defn}[thm]{\protect\definitionname}
  \theoremstyle{plain}
  \newtheorem{prop}[thm]{\protect\propositionname}
  \theoremstyle{plain}
  \newtheorem{cor}[thm]{\protect\corollaryname}
  \theoremstyle{remark}
  \newtheorem*{rem*}{\protect\remarkname}
  \theoremstyle{plain}
  \newtheorem{lem}[thm]{\protect\lemmaname}

\newref{prop}{name=Proposition~}
\newref{thm}{name=Theorem~}
\newref{lem}{name=Lemma~}
\newref{cor}{name=Corollary~}
\newref{conj}{name=Conjecture~}

\makeatother

  \providecommand{\conjecturename}{Conjecture}
  \providecommand{\corollaryname}{Corollary}
  \providecommand{\definitionname}{Definition}
  \providecommand{\lemmaname}{Lemma}
  \providecommand{\propositionname}{Proposition}
  \providecommand{\remarkname}{Remark}
\providecommand{\theoremname}{Theorem}

\begin{document}

\title[Non-cooperative KPP systems]{Non-cooperative Fisher\textendash KPP systems: traveling waves and
long-time behavior}

\author{Léo Girardin}

\thanks{Laboratoire Jacques-Louis Lions, CNRS UMR 7598, Université Pierre
et Marie Curie, 4 place Jussieu, 75005 Paris, France}

\email{girardin@ljll.math.upmc.fr}
\begin{abstract}
This paper is concerned with non-cooperative parabolic reaction\textendash diffusion
systems which share structural similarities with the scalar Fisher\textendash KPP
equation. These similarities make it possible to prove, among other
results, an extinction and persistence dichotomy and, when persistence
occurs, the existence of a positive steady state, the existence of
traveling waves with a half-line of possible speeds and a positive
minimal speed and the equality between this minimal speed and the
spreading speed for the Cauchy problem. Non-cooperative KPP systems
can model various phenomena where the following three mechanisms occur:
local diffusion in space, linear cooperation and superlinear competition.
\end{abstract}



\maketitle
\tableofcontents{}

\section{Introduction}

In this paper, we study a large class of parabolic reaction\textendash diffusion
systems whose prototype is the so-called Lotka\textendash Volterra
mutation\textendash competition\textendash diffusion system:
\[
\left\{ \begin{matrix}\partial_{t}u_{1}-d_{1}\partial_{xx}u_{1}=r_{1}u_{1}-\left(\sum\limits _{j=1}^{N}c_{1,j}u_{j}\right)u_{1}-\mu u_{1}+\mu u_{2}\\
\partial_{t}u_{2}-d_{2}\partial_{xx}u_{2}=r_{2}u_{2}-\left(\sum\limits _{j=1}^{N}c_{2,j}u_{j}\right)u_{2}-2\mu u_{2}+\mu u_{1}+\mu u_{3}\\
\vdots\\
\partial_{t}u_{N}-d_{N}\partial_{xx}u_{N}=r_{N}u_{N}-\left(\sum\limits _{j=1}^{N}c_{N,j}u_{j}\right)u_{N}-\mu u_{N}+\mu u_{N-1}
\end{matrix}\right.
\]
where $N$ is an integer larger than or equal to $2$ and the coefficients
$d_{i}$, $r_{i}$, $c_{i,j}$ (with $i,j\in\{1,\text{\dots},N\}$)
and $\mu$ are positive real numbers.

This system can be understood as an ecological model, where $\left(u_{1},\text{\dots},u_{N}\right)$
is a metapopulation density phenotypically structured, $\mu u_{i-1}-\mu u_{i}$
and $\mu u_{i+1}-\mu u_{i}$ are the step-wise mutations of the $i$-th
phenotype with a mutation rate $\mu$, $d_{i}$ is its dispersal rate,
$r_{i}$ is its growth rate per capita in absence of mutation, $c_{i,j}$
is the rate of the competition exerted by the $j$-th phenotype on
the $i$-th phenotype, $\frac{r_{i}}{c_{i,i}}$ is the carrying capacity
of the $i$-th phenotype in absence of mutation and interphenotypic
competition.

We are especially interested in spreading properties which describe
the invasion of the population in an uninhabited environment and which
are expected to involve so-called traveling wave solutions. Such solutions
were first studied, independently and both in 1937, by Fisher \cite{Fisher_1937}
on one hand and by Kolmogorov, Petrovsky and Piskunov \cite{KPP_1937}
on the other hand for the equation that is now well-known as the Fisher\textendash KPP
equation, Fisher equation or KPP equation:
\[
\partial_{t}u-\partial_{xx}u=u\left(1-u\right).
\]

While a lot of work has been accomplished about traveling waves and
spreading properties for scalar reaction\textendash diffusion equations,
the picture is much less complete regarding coupled systems of reaction\textendash diffusion
equations. In particular, almost nothing is known about non-cooperative
systems like the system above.

Before going any further, let us introduce more precisely the problem.

\subsection{Notations}

Let $\left(n,n'\right)\in\left(\mathbb{N}\cap[1,+\infty)\right)^{2}$.
The set of the first $n$ positive integers $\left[1,n\right]\cap\mathbb{N}$
is denoted $\left[n\right]$ (and $\left[0\right]=\emptyset$ by convention).

\subsubsection{Typesetting conventions}

In order to ease the reading, we reserve the italic typeface ($x$,
$f$, $X$) for reals, real-valued functions or subsets of $\mathbb{R}$,
the bold typeface ($\mathbf{v}$, $\mathbf{A}$) for euclidean vectors
or vector-valued functions, in lower case for column vectors and in
upper case for other matrices\footnote{This convention being superseded by the previous one when the dimension
is specifically equal to $1$.}, the sans serif typeface in upper case ($\mathsf{B}$, $\mathsf{K}$)
for subsets of euclidean spaces\footnote{Same exception.} and the
calligraphic typeface in upper case ($\mathscr{C}$, $\mathscr{L}$)
for functional spaces and operators.

\subsubsection{Linear algebra notations}
\begin{itemize}
\item The canonical basis of $\mathbb{R}^{n}$ is denoted $\left(\mathbf{e}_{n,i}\right)_{i\in\left[n\right]}$.
The euclidean norm of $\mathbb{R}^{n}$ is denoted $\left|\bullet\right|_{n}$.
The open euclidean ball of center $\mathbf{v}\in\mathbb{R}^{n}$ and
radius $r>0$ and its boundary are denoted $\mathsf{B}_{n}\left(\mathbf{v},r\right)$
and $\mathsf{S}_{n}\left(\mathbf{v},r\right)$ respectively.
\item The space $\mathbb{R}^{n}$ is equipped with one partial order $\geq_{n}$
and two strict partial orders $>_{n}$ and $\gg_{n}$, defined as
\[
\mathbf{v}\geq_{n}\hat{\mathbf{v}}\text{ if }v_{i}\geq\hat{v}_{i}\text{ for all }i\in\left[n\right],
\]
\[
\mathbf{v}>_{n}\hat{\mathbf{v}}\text{ if }\mathbf{v}\geq_{n}\hat{\mathbf{v}}\text{ and }\mathbf{v}\neq\hat{\mathbf{v}},
\]
\[
\mathbf{v}\gg_{n}\hat{\mathbf{v}}\text{ if }v_{i}>\hat{v}_{i}\text{ for all }i\in\left[n\right].
\]
The strict orders $>_{n}$ and $\gg_{n}$ coincide if and only if
$n=1$. \\
A vector $\mathbf{v}\in\mathbb{R}^{n}$ is \textit{nonnegative} if
$\mathbf{v}\geq_{n}\mathbf{0}$, \textit{nonnegative nonzero} if $\mathbf{v}>_{n}\mathbf{0}$,
\textit{positive} if $\mathbf{v}\gg_{n}\mathbf{0}$. The sets of all
nonnegative, nonnegative nonzero and positive vectors are respectively
denoted $\mathsf{K}_{n}$, $\mathsf{K}_{n}^{+}$ and $\mathsf{K}_{n}^{++}$.
\item The sets $\mathsf{K}_{n}^{+}\cap\mathsf{S}_{n}\left(\mathbf{0},1\right)$
and $\mathsf{K}_{n}^{++}\cap\mathsf{S}_{n}\left(\mathbf{0},1\right)$
are respectively denoted $\mathsf{S}_{n}^{+}\left(\mathbf{0},1\right)$
and $\mathsf{S}_{n}^{++}\left(\mathbf{0},1\right)$.
\item For any $X\subset\mathbb{R}$, the sets of $X$-valued matrices of
dimension $n\times n'$ and $n\times n$ are respectively denoted
$\mathsf{M}_{n,n'}\left(X\right)$ and $\mathsf{M}_{n}\left(X\right)$
. If $X=\mathbb{R}$ and if the context is unambiguous, we simply
write $\mathsf{M}_{n,n'}$ and $\mathsf{M}_{n}$. As usual, the entry
at the intersection of the $i$-th row and the $j$-th column of the
matrix $\mathbf{A}\in\mathsf{M}_{n,n'}$ is denoted $a_{i,j}$ and
the $i$-th component of the vector $\mathbf{v}\in\mathbb{R}^{n}$
is denoted $v_{i}$. For any vector $\mathbf{v}\in\mathbb{R}^{n}$,
$\text{diag}\mathbf{v}$ denotes the diagonal matrix whose $i$-th
diagonal entry is $v_{i}$.
\item Matrices are vectors and consistently we may apply the notations $\geq_{nn'}$,
$>_{nn'}$ and $\gg_{nn'}$ as well as the vocabulary nonnegative,
nonnegative nonzero and positive to matrices. We emphasize this convention
because of the possible confusion with the notion of \textquotedblleft positive
definite square matrix\textquotedblright . 
\item A matrix $\mathbf{A}\in\mathsf{M}_{n}$ is \textit{essentially nonnegative},
\textit{essentially nonnegative nonzero}, \textit{essentially positive}
if $\mathbf{A}-\min\limits _{i\in\left[n\right]}\left(a_{i,i}\right)\mathbf{I}_{n}$
is nonnegative, nonnegative nonzero, positive respectively.
\item The identity of $\mathsf{M}_{n}$ and the element of $\mathsf{M}_{n,n'}$
whose every entry is equal to $1$ are respectively denoted $\mathbf{I}_{n}$
and $\mathbf{1}_{n,n'}$ ($\mathbf{1}_{n}$ if $n=n'$) . 
\item We recall the definition of the Hadamard product of a pair of matrices
$\left(\mathbf{A},\mathbf{B}\right)^{2}\in\left(\mathsf{M}_{n,n'}\right)^{2}$:
\[
\mathbf{A}\circ\mathbf{B}=\left(a_{i,j}b_{i,j}\right)_{\left(i,j\right)\in\left[n\right]\times\left[n'\right]}.
\]
 The identity matrix under Hadamard multiplication is $\mathbf{1}_{n,n'}$. 
\item The spectral radius of any $\mathbf{A}\in\mathsf{M}_{n}$ is denoted
$\rho\left(\mathbf{A}\right)$. Recall from the Perron\textendash Frobenius
theorem that if $\mathbf{A}$ is nonnegative and irreducible, $\rho\left(\mathbf{A}\right)$
is the dominant eigenvalue of $\mathbf{A}$, called the \textit{Perron\textendash Frobenius
eigenvalue} $\lambda_{PF}\left(\mathbf{A}\right)$, and is the unique
eigenvalue associated with a positive eigenvector. Recall also that
if $\mathbf{A}\in\mathsf{M}_{n}$ is essentially nonnegative and irreducible,
the Perron\textendash Frobenius theorem can still be applied. In such
a case, the unique eigenvalue of $\mathbf{A}$ associated with a positive
eigenvector is $\lambda_{PF}\left(\mathbf{A}\right)=\rho\left(\mathbf{A}-\min\limits _{i\in\left[n\right]}\left(a_{i,i}\right)\mathbf{I}_{n}\right)+\min\limits _{i\in\left[n\right]}\left(a_{i,i}\right)$.
Any eigenvector associated with $\lambda_{PF}\left(\mathbf{A}\right)$
is referred to as a \textit{Perron\textendash Frobenius eigenvector}
and the unit one is denoted $\mathbf{n}_{PF}\left(\mathbf{A}\right)$.
\end{itemize}

\subsubsection{Functional analysis notations}
\begin{itemize}
\item We will consider a parabolic problem of two real variables, the \textquotedblleft time\textquotedblright{}
$t$ and the \textquotedblleft space\textquotedblright{} $x$. A (straight)
\textit{parabolic cylinder} in $\mathbb{R}^{2}$ is a subset of the
form $\left(t_{0},t_{f}\right)\times\left(a,b\right)$ with $\left(t_{0},t_{f},a,b\right)\in\overline{\mathbb{R}}^{4}$,
$t_{0}<t_{f}$ and $a<b$. The parabolic boundary $\partial_{P}\mathsf{Q}$
of a bounded parabolic cylinder $\mathsf{Q}$ is defined classically.
A \textit{classical solution} of some second-order parabolic problem
of dimension $n$ set in a parabolic cylinder $\mathsf{Q}=\left(t_{0},t_{f}\right)\times\left(a,b\right)$
is a solution in
\[
\mathscr{C}^{1}\left(\left(t_{0},t_{f}\right),\mathscr{C}^{2}\left(\left(a,b\right),\mathbb{R}^{n}\right)\right)\cap\mathscr{C}\left(\mathsf{Q}\cup\partial\mathsf{Q},\mathbb{R}^{n}\right).
\]
Similarly, a classical solution of some second-order elliptic problem
of dimension $n$ set in an interval $\left(a,b\right)\subset\overline{\mathbb{R}}$
is a solution in
\[
\mathscr{C}^{2}\left(\left(a,b\right),\mathbb{R}^{n}\right)\cap\mathscr{C}\left(\left(a,b\right)\cup\partial\left(a,b\right),\mathbb{R}^{n}\right).
\]
\item Consistently with $\mathbb{R}^{n}$, the set of functions $\left(\mathbb{R}^{n}\right)^{\left(\mathbb{R}^{n'}\right)}$
is equipped with
\[
\mathbf{f}\geq_{\mathbb{R}^{n'},\mathbb{R}^{n}}\hat{\mathbf{f}}\text{ if }\mathbf{f}\left(\mathbf{v}\right)-\hat{\mathbf{f}}\left(\mathbf{v}\right)\in\mathsf{K}_{n}\text{ for all }\mathbf{v}\in\mathbb{R}^{n'},
\]
\[
\mathbf{f}>_{\mathbb{R}^{n'},\mathbb{R}^{n}}\hat{\mathbf{f}}\text{ if }\mathbf{f}\geq_{\mathbb{R}^{n'},\mathbb{R}^{n}}\hat{\mathbf{f}}\text{ and }\mathbf{f}\neq\hat{\mathbf{f}},
\]
\[
\mathbf{f}\gg_{\mathbb{R}^{n'},\mathbb{R}^{n}}\hat{\mathbf{f}}\text{ if }\mathbf{f}\left(\mathbf{v}\right)-\hat{\mathbf{f}}\left(\mathbf{v}\right)\in\mathsf{K}_{n}^{++}\text{ for all }\mathbf{v}\in\mathbb{R}^{n'}.
\]
We define consistently nonnegative, nonnegative nonzero and positive
functions\footnote{Regarding functions, some authors use $>$ to denote what is here
denoted $\gg$. Thus the use of these two functional notations will
be as sparse as possible and we will prefer the less ambiguous expressions
\textquotedblleft nonnegative nonzero\textquotedblright{} and \textquotedblleft positive\textquotedblright .}. 
\item The composition of two compatible functions $\mathbf{f}$ and $\hat{\mathbf{f}}$
is denoted $\mathbf{f}\left[\hat{\mathbf{f}}\right]$, the usual $\circ$
being reserved for the Hadamard product.
\item If the context is unambiguous, a functional space $\mathscr{F}\left(\mathsf{X},\mathbb{R}\right)$
is denoted $\mathscr{F}\left(\mathsf{X}\right)$. 
\item For any smooth open bounded connected set $\Omega\subset\mathbb{R}^{n'}$
and any second order linear elliptic operator $\mathscr{L}:\mathscr{C}^{2}\left(\Omega,\mathbb{R}^{n}\right)\to\mathscr{C}\left(\Omega,\mathbb{R}^{n}\right)$
with coefficients in $\mathscr{C}_{b}\left(\Omega,\mathbb{R}^{n}\right)$,
the \textit{Dirichlet principal eigenvalue} of $\mathscr{L}$ in $\Omega$,
denoted $\lambda_{1,Dir}\left(-\mathscr{L},\Omega\right)$, is well-defined
if $\mathscr{L}$ is order-preserving in $\Omega$. Recall from the
Krein\textendash Rutman theorem that $\lambda_{1,Dir}\left(-\mathscr{L},\Omega\right)$
is the unique eigenvalue associated with a principal eigenfunction
positive in $\Omega$ and null on $\partial\Omega$. Sufficient conditions
for the order-preserving property are:
\begin{itemize}
\item $n=1$;
\item $n\geq2$ and the system is \textit{weakly coupled} (the coupling
occurs only in the zeroth order term) and \textit{fully coupled} (the
zeroth order coefficient is an essentially nonnegative irreducible
matrix). When $n\geq2$, order-preserving operators are also referred
to as \textit{cooperative} operators.
\end{itemize}
\end{itemize}

\subsection{Setting of the problem}

From now on, an integer $N\in\mathbb{N}\cap[2,+\infty)$ is fixed.
For the sake of brevity, the subscripts depending only on $1$ or
$N$ in the various preceding notations will be omitted when the context
is unambiguous. 

We also fix $\mathbf{d}\in\mathsf{K}^{++}$, $\mathbf{D}=\text{diag}\mathbf{d}$,
$\mathbf{L}\in\mathsf{M}$ and $\mathbf{c}\in\mathscr{C}^{1}\left(\mathbb{R}^{N},\mathbb{R}^{N}\right)$. 

The semilinear parabolic evolution system under scrutiny is
\[
\partial_{t}\mathbf{u}-\mathbf{D}\partial_{xx}\mathbf{u}=\mathbf{L}\mathbf{u}-\mathbf{c}\left[\mathbf{u}\right]\circ\mathbf{u},\quad\left(E_{KPP}\right)
\]
 the unknown being $\mathbf{u}:\mathbb{R}^{2}\to\mathbb{R}^{N}$ (although
$\left(E_{KPP}\right)$ might occasionally be restricted to a parabolic
cylinder).

The associated semilinear elliptic stationary system is
\[
-\mathbf{D}\mathbf{u}''=\mathbf{L}\mathbf{u}-\mathbf{c}\left[\mathbf{u}\right]\circ\mathbf{u},\quad\left(S_{KPP}\right)
\]
 the unknown being $\mathbf{u}:\mathbb{R}\to\mathbb{R}^{N}$ (although
$\left(S_{KPP}\right)$ might occasionally be restricted to an interval).

\subsubsection{Restrictive assumptions}

The main restrictive assumptions are the following ones.

{\renewcommand\labelenumi{(${H}_\theenumi$)}
\begin{enumerate}
\item $\mathbf{L}$ is essentially nonnegative and irreducible. 
\item $\mathbf{c}\left(\mathsf{K}\right)\subset\mathsf{K}$.
\item $\mathbf{c}\left(\mathbf{0}\right)=\mathbf{0}$.
\item There exist
\[
\left(\underline{\alpha},\delta,\underline{\mathbf{c}}\right)\in[1,+\infty)^{2}\times\mathsf{K}^{++}
\]
such that
\[
\sum_{j=1}^{N}l_{i,j}n_{j}\geq0\implies\alpha^{\delta}\underline{c}_{i}\leq c_{i}\left(\alpha\mathbf{n}\right)
\]
for all
\[
\left(\mathbf{n},\alpha,i\right)\in\mathsf{S}^{+}\left(\mathbf{0},1\right)\times[\underline{\alpha},+\infty)\times\left[N\right].
\]
\end{enumerate}
}

A few immediate consequences of these assumptions deserve to be pointed
out.
\begin{itemize}
\item $\left(E_{KPP}\right)$ and $\left(S_{KPP}\right)$ are not cooperative
and do not satisfy a comparison principle.
\item The Perron\textendash Frobenius eigenvalue $\lambda_{PF}\left(\mathbf{L}\right)$
is well-defined and the system $\mathbf{u}'=\mathbf{L}\mathbf{u}$
is cooperative.
\item For all $\mathbf{v}\in\mathbb{R}^{N}$, the Jacobian matrix of $\mathbf{w}\mapsto\mathbf{c}\left(\mathbf{w}\right)\circ\mathbf{w}$
at $\mathbf{v}$ is
\[
\text{diag}\mathbf{c}\left(\mathbf{v}\right)+\left(\mathbf{v}\mathbf{1}_{1,N}\right)\circ D\mathbf{c}\left(\mathbf{v}\right).
\]
 In particular, at $\mathbf{v}=\mathbf{0}$, this Jacobian is null
if and only if $\left(H_{3}\right)$ is satisfied. Also, if $D\mathbf{c}\left(\mathbf{v}\right)\geq\mathbf{0}$
for all $\mathbf{v}\in\mathsf{K}$, then the system $\mathbf{u}'=-\mathbf{c}\left[\mathbf{u}\right]\circ\mathbf{u}$
is competitive.
\item This framework contains both the Lotka\textendash Volterra linear
competition $\mathbf{c}\left(\mathbf{u}\right)=\mathbf{C}\mathbf{u}$
and the Gross\textendash Pitaevskii quadratic competition $\mathbf{c}\left(\mathbf{u}\right)=\mathbf{C}\left(\mathbf{u}\circ\mathbf{u}\right)$
(with, in both cases, $\mathbf{C}\gg\mathbf{0}$).
\end{itemize}

\subsubsection{KPP property}

The system $\left(E_{KPP}\right)$ is, in some sense, a \textquotedblleft multidimensional
KPP equation\textquotedblright . Let us recall the main features of
scalar KPP nonlinearities:
\begin{enumerate}
\item $f'\left(0\right)>0$ (instability of the null state),
\item $f'\left(0\right)v\geq f\left(v\right)$ for all $v\geq0$ (no Allee
effect),
\item there exists $K>0$ such that $f\left(v\right)<0$ if and only if
$v>K$ (saturation).
\end{enumerate}
Of course, our assumptions $\left(H_{1}\right)$\textendash $\left(H_{4}\right)$
aim to put forward a possible generalization of these features. A
few comments are in order.

Regarding the saturation property, the growth at least linear of $\mathbf{c}$
$\left(H_{4}\right)$ will imply an analogous statement. Ensuring
uniform $\mathscr{L}^{\infty}$ estimates is really the main mathematical
role of the competitive term. 

Regarding the presence of an Allee effect, $\mathbf{c}\left(\mathsf{K}\right)\subset\mathsf{K}$
$\left(H_{2}\right)$ and $\mathbf{c}\left(\mathbf{0}\right)=\mathbf{0}$
$\left(H_{3}\right)$ clearly yield that $\partial_{t}\mathbf{u}-\mathbf{D}\partial_{xx}\mathbf{u}=\mathbf{L}\mathbf{u}$
is the linearization at $\mathbf{0}$ of $\left(E_{KPP}\right)$ and
moreover that $\mathbf{f}:\mathbf{v}\mapsto\mathbf{L}\mathbf{v}-\mathbf{c}\left(\mathbf{v}\right)\circ\mathbf{v}$
satisfies
\[
D\mathbf{f}\left(\mathbf{0}\right)\mathbf{v}\geq\mathbf{f}\left(\mathbf{v}\right)\text{ for all }\mathbf{v}\in\mathsf{K}.
\]

Regarding the instability of the null state, we stress here that the
notion of positivity of matrices is somewhat ambiguous and, consequently,
finding a natural generalization of $f'\left(0\right)>0$ is not completely
straightforward. 

In order to decide which positivity sense is the right one, we offer
the following criterion. On one hand, a suitable multidimensional
generalization of the KPP equation should enable generalizations of
the striking results concerning its scalar counterpart. On the other
hand, the most remarkable result about the KPP equation is that the
answer to many natural questions (value of the spreading speed, persistence
in bounded domains, etc.) only depends on $f'\left(0\right)$ (the
importance of $f'\left(0\right)$ can already be seen in the features
above). Thus, in our opinion, a KPP system should also be linearly
determinate regarding these questions. 

With this criterion in mind, let us explain for instance why positivity
understood as positive definite matrices (i.e. positive spectrum)
is not satisfying. In such a case, Lotka\textendash Volterra competition\textendash diffusion
nonlinearities, whose linearization at $\mathbf{0}$ has the form
$\text{diag}\mathbf{r}$ with $\mathbf{r}\in\mathsf{K}^{++}$, would
be KPP nonlinearities. Nevertheless, it is known that the spreading
speed of a competition\textendash diffusion system is not necessarily
linearly determinate (for instance, see Lewis\textendash Li\textendash Weinberger
\cite{Lewis_Weinberg}). 

On the contrary, the main theorems of the present paper will show
unambiguously that irreducibility and essential nonnegativity $\left(H_{1}\right)$
supplemented with $\lambda_{PF}\left(\mathbf{L}\right)>0$ is the
right notion. This confirmation of the relevance of $\left(H_{1}\right)$\textendash $\left(H_{4}\right)$
will then lead us to a general definition of multidimensional KPP
nonlinearity. 

\subsection{Main results}

\subsubsection{KPP-type theorems established under $\left(H_{1}\right)$\textendash $\left(H_{4}\right)$}
\begin{thm}
\label{thm:Strong_positivity} {[}Strong positivity{]} For all nonnegative
classical solutions $\mathbf{u}$ of $\left(E_{KPP}\right)$ set in
$\left(0,+\infty\right)\times\mathbb{R}$, if $x\mapsto\mathbf{u}\left(0,x\right)$
is nonnegative nonzero, then $\mathbf{u}$ is positive in $\left(0,+\infty\right)\times\mathbb{R}$.

Consequently, all nonnegative nonzero classical solutions of $\left(S_{KPP}\right)$
are positive. 
\end{thm}

\begin{thm}
\label{thm:Absorbing_set} {[}Absorbing set and upper estimates{]}
There exists a positive function $\mathbf{g}\in\mathscr{C}\left([0,+\infty),\mathsf{K}^{++}\right)$,
component-wise non-decreasing, such that all nonnegative classical
solutions $\mathbf{u}$ of $\left(E_{KPP}\right)$ set in $\left(0,+\infty\right)\times\mathbb{R}$
satisfy
\[
\mathbf{u}\left(t,x\right)\leq\left(g_{i}\left(\sup_{x\in\mathbb{R}}u_{i}\left(0,x\right)\right)\right)_{i\in\left[N\right]}\text{ for all }\left(t,x\right)\in[0,+\infty)\times\mathbb{R}
\]
and furthermore, if $x\mapsto\mathbf{u}\left(0,x\right)$ is bounded,
then
\[
\left(\limsup_{t\to+\infty}\sup_{x\in\mathbb{R}}u_{i}\left(t,x\right)\right)_{i\in\left[N\right]}\leq\mathbf{g}\left(0\right).
\]

Consequently, all bounded nonnegative classical solutions $\mathbf{u}$
of $\left(S_{KPP}\right)$ satisfy
\[
\mathbf{u}\leq\mathbf{g}\left(0\right).
\]
\end{thm}

\begin{thm}
\label{thm:Extinction_or_persistence} {[}Extinction or persistence
dichotomy{]}
\begin{enumerate}[label=\roman*)]
\item  Assume $\lambda_{PF}\left(\mathbf{L}\right)<0$. Then all bounded
nonnegative classical solutions of $\left(E_{KPP}\right)$ set in
$\left(0,+\infty\right)\times\mathbb{R}$ converge asymptotically
in time, exponentially fast, and uniformly in space to $\mathbf{0}$.
\item Conversely, assume $\lambda_{PF}\left(\mathbf{L}\right)>0$. Then
there exists $\nu>0$ such that all bounded positive classical solutions
$\mathbf{u}$ of $\left(E_{KPP}\right)$ set in $\left(0,+\infty\right)\times\mathbb{R}$
satisfy, for all bounded intervals $I\subset\mathbb{R}$,
\[
\left(\liminf_{t\to+\infty}\inf_{x\in I}u_{i}\left(t,x\right)\right)_{i\in\left[N\right]}\geq\nu\mathbf{1}_{N,1}.
\]
Consequently, all bounded nonnegative classical solutions of $\left(S_{KPP}\right)$
are valued in
\[
\prod_{i=1}^{N}\left[\nu,g_{i}\left(0\right)\right].
\]
\end{enumerate}
\end{thm}

As will be explained later on, the critical case $\lambda_{PF}\left(\mathbf{L}\right)=0$
is more challenging than expected and is not solved here, in spite
of the following extinction conjecture.
\begin{conjecture*}
Assume $\lambda_{PF}\left(\mathbf{L}\right)=0$ and
\[
\text{span}\left(\mathbf{n}_{PF}\left(\mathbf{L}\right)\right)\cap\mathsf{K}\cap\mathbf{c}^{-1}\left(\left\{ \mathbf{0}\right\} \right)=\left\{ \mathbf{0}\right\} .
\]

Then all bounded nonnegative classical solutions of $\left(E_{KPP}\right)$
set in $\left(0,+\infty\right)\times\mathbb{R}$ converge asymptotically
in time and locally uniformly in space to $\mathbf{0}$.
\end{conjecture*}
Although \thmref{Extinction_or_persistence} proves that the attractor
of the induced semiflow is reduced to $\left\{ \mathbf{0}\right\} $
in the extinction case, in the persistence case the long-time behavior
is unclear and might not be reduced to a locally uniform convergence
toward a unique stable steady state. This direct consequence of the
multidimensional structure of $\left(E_{KPP}\right)$ is a major difference
with the scalar KPP equation. Still, the following theorem provides
some additional information about the steady states of $\left(E_{KPP}\right)$
and confirms in some sense the preceding conjecture.
\begin{thm}
\label{thm:Existence_of_steady_states} {[}Existence of steady states{]}
\begin{enumerate}[label=\roman*)]
\item  If $\lambda_{PF}\left(\mathbf{L}\right)<0$, there exists no positive
classical solution of $\left(S_{KPP}\right)$.
\item If $\lambda_{PF}\left(\mathbf{L}\right)=0$ and 
\[
\text{span}\left(\mathbf{n}_{PF}\left(\mathbf{L}\right)\right)\cap\mathsf{K}\cap\mathbf{c}^{-1}\left(\left\{ \mathbf{0}\right\} \right)=\left\{ \mathbf{0}\right\} ,
\]
 there exists no bounded positive classical solution of $\left(S_{KPP}\right)$.
\item If $\lambda_{PF}\left(\mathbf{L}\right)>0$, there exists a constant
positive classical solution of $\left(S_{KPP}\right)$.
\end{enumerate}
\end{thm}

Due to the unclear long-time behavior of $\left(E_{KPP}\right)$ when
$\lambda_{PF}\left(\mathbf{L}\right)>0$, it seems inappropriate to
consider only traveling wave solutions connecting $\mathbf{0}$ to
some stable positive steady state (as is usually done in the monostable
scalar setting). Hence we resort to the following more flexible definition. 
\begin{defn*}
A \textit{traveling wave solution} of $\left(E_{KPP}\right)$ is a
pair
\[
\left(\mathbf{p},c\right)\in\mathscr{C}^{2}\left(\mathbb{R},\mathbb{R}^{N}\right)\times[0,+\infty)
\]
 which satisfies:
\begin{enumerate}
\item $\mathbf{u}:\left(t,x\right)\mapsto\mathbf{p}\left(x-ct\right)$ is
a bounded positive classical solution of $\left(E_{KPP}\right)$;
\item $\left(\liminf\limits _{\xi\to-\infty}p_{i}\left(\xi\right)\right)_{i\in\left[N\right]}\in\mathsf{K}^{+}$;
\item $\lim\limits _{\xi\to+\infty}\mathbf{p}\left(\xi\right)=\mathbf{0}$.
\end{enumerate}
We refer to $\mathbf{p}$ as the \textit{profile} of the traveling
wave and to $c$ as its \textit{speed}. \footnote{Let us emphasize once and for all that the vector field $\mathbf{c}$
is not to be confused with the real number $c$. The former is named
after \textquotedblleft competition\textquotedblright{} whereas the
latter is traditionally named after \textquotedblleft celerity\textquotedblright . }
\end{defn*}
\begin{thm}
\label{thm:Traveling_waves} {[}Traveling waves{]} Assume $\lambda_{PF}\left(\mathbf{L}\right)>0$.
\begin{enumerate}[label=\roman*)]
\item  \label{enu:Existence_minimal_wave_speed} There exists $c^{\star}>0$
such that:
\begin{enumerate}
\item there exists no traveling wave solution of $\left(E_{KPP}\right)$
with speed $c$ for all $c\in[0,c^{\star})$;
\item if, furthermore,
\[
D\mathbf{c}\left(\mathbf{v}\right)\geq\mathbf{0}\text{ for all }\mathbf{v}\in\mathsf{K},
\]
then there exists a traveling wave solution of $\left(E_{KPP}\right)$
with speed $c$ for all $c\geq c^{\star}$.
\end{enumerate}
\item \label{enu:Upper_bound_for_the_profiles} All profiles $\mathbf{p}$
satisfy
\[
\mathbf{p}\leq\mathbf{g}\left(0\right).
\]
\item \label{enu:Persistence_at_the_back_of_the_front} All profiles $\mathbf{p}$
satisfy
\[
\left(\liminf\limits _{\xi\to-\infty}p_{i}\left(\xi\right)\right)_{i\in\left[N\right]}\geq\nu\mathbf{1}_{N,1}.
\]
\item \label{enu:Monotonicity_at_the_edge_of_the_front} All profiles are
component-wise decreasing in a neighborhood of $+\infty$.
\end{enumerate}
\end{thm}

When traveling waves exist for all speeds $c\geq c^{\star}$, $c^{\star}$
is called the minimal wave speed.
\begin{thm}
\label{thm:Spreading_speed} {[}Spreading speed{]} Assume $\lambda_{PF}\left(\mathbf{L}\right)>0$.
For all $x_{0}\in\mathbb{R}$ and all bounded nonnegative nonzero
$\mathbf{v}\in\mathscr{C}\left(\mathbb{R},\mathbb{R}^{N}\right)$,
the classical solution $\mathbf{u}$ of $\left(E_{KPP}\right)$ set
in $\left(0,+\infty\right)\times\mathbb{R}$ with initial data $\mathbf{v}\mathbf{1}_{\left(-\infty,x_{0}\right)}$
satisfies
\[
\left(\lim_{t\to+\infty}\sup_{x\in\left(y,+\infty\right)}u_{i}\left(t,x+ct\right)\right)_{i\in\left[N\right]}=\mathbf{0}\text{ for all }c\in\left(c^{\star},+\infty\right)\text{ and all }y\in\mathbb{R},
\]
\[
\left(\liminf_{t\to+\infty}\inf_{x\in\left[-R,R\right]}u_{i}\left(t,x+ct\right)\right)_{i\in\left[N\right]}\in\mathsf{K}^{++}\text{ for all }c\in[0,c^{\star})\text{ and all }R>0.
\]
\end{thm}

Of course, by well-posedness of $\left(E_{KPP}\right)$, the solution
with initial data $x\mapsto\mathbf{v}\left(-x\right)\mathbf{1}_{\left(-x_{0},+\infty\right)}$
is precisely $\left(t,x\right)\mapsto\mathbf{u}\left(t,-x\right)$
($\mathbf{u}$ being the solution with initial data $\mathbf{v}\mathbf{1}_{\left(-\infty,x_{0}\right)}$).
This gives the expected symmetrical spreading result (the solution
with initial data $x\mapsto\mathbf{v}\left(-x\right)\mathbf{1}_{\left(-x_{0},+\infty\right)}$
spreads on the left at speed $-c^{\star}$). Moreover, since these
two spreading results with front-like initial data actually cover
compactly supported $\mathbf{v}$, we also get straightforwardly the
spreading result for compactly supported initial data (the solution
spreads on the right at speed $c^{\star}$ and on the left at speed
$-c^{\star}$). 

Consequently, $c^{\star}$ is also called the spreading speed associated
with front-like or compactly supported initial data. We recall that
for generic KPP problems these two spreading speeds are different
as soon as the spatial domain is multidimensional. In such a case,
the spreading speed associated with front-like initial data generically
coincides with the minimal wave speed whereas the spreading speed
associated with compactly supported initial data is smaller.
\begin{thm}
\label{thm:Characterization_minimal_speed} {[}Characterization and
estimates for $c^{\star}${]} Assume $\lambda_{PF}\left(\mathbf{L}\right)>0$.
We have
\[
c^{\star}=\min_{\mu>0}\frac{\lambda_{PF}\left(\mu^{2}\mathbf{D}+\mathbf{L}\right)}{\mu}
\]
 and this minimum is attained at a unique $\mu_{c^{\star}}>0$.

Consequently, if we assume (without loss of generality)
\[
d_{1}\leq d_{2}\leq\text{\dots}\leq d_{N},
\]
the following estimates hold.
\begin{enumerate}[label=\roman*)]
\item  We have
\[
2\sqrt{d_{1}\lambda_{PF}\left(\mathbf{L}\right)}\leq c^{\star}\leq2\sqrt{d_{N}\lambda_{PF}\left(\mathbf{L}\right)}.
\]
If $d_{1}<d_{N}$, both inequalities are strict. If $d_{1}=d_{N}$,
both inequalities are equalities.
\item For all $i\in\left[N\right]$ such that $l_{i,i}>0$, we have
\[
c^{\star}>2\sqrt{d_{i}l_{i,i}}.
\]
\item Let $\mathbf{r}\in\mathbb{R}^{N}$ and $\mathbf{M}\in\mathsf{M}$
be given by the unique decomposition of $\mathbf{L}$ of the form
\[
\mathbf{L}=\text{diag}\mathbf{r}+\mathbf{M}\text{ with }\mathbf{M}^{T}\mathbf{1}_{N,1}=\mathbf{0}.
\]
Let $\left(\left\langle d\right\rangle ,\left\langle r\right\rangle \right)\in\left(0,+\infty\right)\times\mathbb{R}$
be defined as
\[
\left\{ \begin{matrix}\left\langle d\right\rangle =\frac{\mathbf{d}^{T}\mathbf{n}_{PF}\left(\mu_{c^{\star}}^{2}\mathbf{D}+\mathbf{L}\right)}{\mathbf{1}_{1,N}\mathbf{n}_{PF}\left(\mu_{c^{\star}}^{2}\mathbf{D}+\mathbf{L}\right)},\\
\left\langle r\right\rangle =\frac{\mathbf{r}^{T}\mathbf{n}_{PF}\left(\mu_{c^{\star}}^{2}\mathbf{D}+\mathbf{L}\right)}{\mathbf{1}_{1,N}\mathbf{n}_{PF}\left(\mu_{c^{\star}}^{2}\mathbf{D}+\mathbf{L}\right)}.
\end{matrix}\right.
\]
If $\left\langle r\right\rangle \geq0$, then
\[
c^{\star}\geq2\sqrt{\left\langle d\right\rangle \left\langle r\right\rangle }.
\]
\end{enumerate}
\end{thm}

\subsubsection{General definition of multidimensional KPP nonlinearity}

The set of assumptions $\left(H_{1}\right)$\textendash $\left(H_{4}\right)$
supplemented with $\lambda_{PF}\left(\mathbf{L}\right)>0$ can be
seen as a particular case of the following definition, which we expect
to be optimal with respect to the preceding collection of theorems.
\begin{defn}
A nonlinear function $\mathbf{f}\in\mathscr{C}^{1}\left(\mathbb{R}^{N},\mathbb{R}^{N}\right)$
is a \textit{KPP nonlinearity} if:
\begin{enumerate}
\item $\mathbf{f}\left(\mathbf{0}\right)=\mathbf{0}$;
\item $D\mathbf{f}\left(\mathbf{0}\right)$ is essentially nonnegative,
irreducible and $\lambda_{PF}\left(D\mathbf{f}\left(\mathbf{0}\right)\right)>0$;
\item $D\mathbf{f}\left(\mathbf{0}\right)\mathbf{v}\geq\mathbf{f}\left(\mathbf{v}\right)$
for all $\mathbf{v}\in\mathsf{K}$;
\item the semiflow induced by $\partial_{t}\mathbf{u}=\mathbf{D}\partial_{xx}\mathbf{u}+\mathbf{f}\left[\mathbf{u}\right]$
with globally bounded, sufficiently regular initial data admits an
absorbing set bounded in $\mathscr{L}^{\infty}\left(\mathbb{R}\right)$.
\end{enumerate}
\end{defn}

Let us explain more precisely how this definition differs from $\left(H_{1}\right)$\textendash $\left(H_{4}\right)$
supplemented with $\lambda_{PF}\left(\mathbf{L}\right)>0$. Defining
\[
\mathbf{L}=D\mathbf{f}\left(\mathbf{0}\right),
\]
\[
\mathbf{c}:\mathbf{v}\mapsto\left\{ \begin{matrix}\left(\frac{1}{v_{i}}\left(\left(\mathbf{L}\mathbf{v}\right)_{i}-f_{i}\left(\mathbf{v}\right)\right)\right)_{i\in\left[N\right]} & \text{if }\mathbf{v}\neq\mathbf{0}\\
\mathbf{0} & \text{if }\mathbf{v}=\mathbf{0}
\end{matrix}\right.,
\]
we find 
\[
\mathbf{f}\left(\mathbf{v}\right)=\mathbf{L}\mathbf{v}-\mathbf{c}\left(\mathbf{v}\right)\circ\mathbf{v}\text{ for all }\mathbf{v}\in\mathbb{R}^{N}.
\]
 On one hand, the irreducibility and essential nonnegativity of $\mathbf{L}$
$\left(H_{1}\right)$, the positivity of its Perron\textendash Frobenius
eigenvalue, as well as the nonnegativity of $\mathbf{c}$ on $\mathsf{K}$
$\left(H_{2}\right)$ with $\mathbf{c}\left(\mathbf{0}\right)=\mathbf{0}$
$\left(H_{3}\right)$ follow directly. On the other hand, the $\mathscr{C}^{1}$
regularity of $\mathbf{c}$ at $\mathbf{0}$ and its specific growth
at infinity $\left(H_{4}\right)$ are not satisfied in general. 

These two properties are satisfied indeed for the applications we
have in mind (which will be exposed in a moment). However it might
be mathematically interesting to consider the case where at least
one of them fails. For instance, let us discuss briefly $\left(H_{4}\right)$.

The only forthcoming result whose proof depends directly on $\left(H_{4}\right)$
is \lemref{Saturation_constant} (which is remarkably one of the main
assumptions of a related paper by Barles, Evans and Souganidis \cite[(F3)]{Barles_Evans_S}).
It is easily seen that if $\mathbf{c}$ grows sublinearly, we cannot
hope in general to recover \lemref{Saturation_constant} (in other
words, under some reasonable assumptions, Barles\textendash Evans\textendash Souganidis\textquoteright s
$\left(\text{F}3\right)$ is satisfied if and only if $\left(H_{4}\right)$;
of course this makes $\left(H_{4}\right)$ even more interesting). 

Nevertheless, this lemma is not a result in itself but a tool used
for the proofs of \thmref{Absorbing_set} as well as the existence
results of \thmref{Existence_of_steady_states} and \thmref{Traveling_waves}.
Hence relaxing $\left(H_{4}\right)$ mainly means finding new proofs
of these results.

Now, without entering into too much details, we point out that if
there exists $\eta>0$ such that the following dissipative assumption:
\[
\left(H_{diss,\eta}\right)\quad\left\{ \begin{matrix}\exists C_{1}\geq0\quad\forall\mathbf{v}\in\mathbb{R}^{N}\quad\left(\mathbf{f}\left(\mathbf{v}\right)+\eta\mathbf{v}\right)^{T}\mathbf{v}\leq C_{1}\\
\exists C_{2}\geq0\quad\forall\mathbf{v}\in\mathbb{R}^{N}\quad D\mathbf{f}\left(\mathbf{v}\right)+\eta\mathbf{I}\leq C_{2}\mathbf{1}\\
\exists\left(C_{3},p\right)\in[0,+\infty)^{2}\quad\forall\mathbf{v}\in\mathbb{R}^{N}\quad\left|\mathbf{f}\left(\mathbf{v}\right)+\eta\mathbf{v}\right|\leq C_{3}\left(1+\left|\mathbf{v}\right|^{p}\right),
\end{matrix}\right.
\]
holds, then the semiflow induced by $\partial_{t}\mathbf{u}=\mathbf{D}\partial_{xx}\mathbf{u}+\mathbf{f}\left[\mathbf{u}\right]$
admits an attractor in some locally uniform topology which is bounded
in $\mathscr{C}_{b}\left(\mathbb{R},\mathbb{R}^{N}\right)$ (see Zelik
\cite{Zelik_2003}). If the semiflow leaves $\mathsf{K}$ invariant
and if we only consider nonnegative initial data, then the quantifiers
$\forall\mathbf{v}\in\mathbb{R}^{N}$ above can all be replaced by
$\forall\mathbf{v}\in\mathsf{K}$. 

In particular, $\mathbf{v}\mapsto\mathbf{L}\mathbf{v}-\mathbf{c}\left(\mathbf{v}\right)\circ\mathbf{v}$
supplemented with $\left(H_{1}\right)$\textendash $\left(H_{3}\right)$
and
\[
\left(H_{4}'\right)\quad\lim_{\left|\mathbf{v}\right|\to+\infty,\mathbf{v}\in\mathsf{K}}\left|\mathbf{c}\left(\mathbf{v}\right)\right|=+\infty\text{ with at most algebraic growth}
\]
satisfies $\left(H_{diss,\eta}\right)$ for any $\eta>0$. (Clearly,
$\left(H_{4}\right)\cup\left(H_{4}'\right)$ contains every choice
of $\mathbf{c}$ such that $\lim\limits _{\left|\mathbf{v}\right|\to+\infty,\mathbf{v}\in\mathsf{K}}\left|\mathbf{c}\left(\mathbf{v}\right)\right|=+\infty$.)

Consequently, dissipative theory provides for some slowly decaying
KPP nonlinearities a proof of \thmref{Absorbing_set}. It should also
provide a proof of \propref{Global_bounds_for_the_elliptic_problem_with_drift},
which is the key estimate to derive the existence of traveling waves,
as well as a proof of the existence result of \thmref{Existence_of_steady_states}
. With these proofs at hand, all our results would be recovered.

\subsection{Related results in the literature}

\subsubsection{Cooperative or almost cooperative systems}

The bibliography about weakly and fully coupled elliptic and parabolic
linear systems is of course extensive. It is possible, for instance,
to define principal eigenvalues and eigenfunctions (Sweers \textit{et
al.} \cite{Birindelli_Mitidieiri_Sweers,Sweers_1992}), to prove the
weak maximum principle (the classical theorems of Protter\textendash Weinberger
\cite{Protter_Weinberger} were refined in the more involved elliptic
case by Figueiredo \textit{et al.} \cite{Figueiredo_1994,Figueiredo_Mit}
and Sweers \cite{Sweers_1992}) or Harnack inequalities (Chen\textendash Zhao
\cite{Chen_Zhao_97} or Arapostathis\textendash Gosh\textendash Marcus
\cite{Araposthathis_} for the elliptic case\footnote{They both prove the same type of results but we will refer hereafter
only to the latter because the former does not cover, as stated, the
one-dimensional space case.}, Földes\textendash Polá\v{c}ik \cite{Foldes_Polacik} for the parabolic
case) and to use the super- and sub-solution method to deduce existence
of solutions (Pao \cite{Pao_1992} among others). In some sense, weakly
and fully coupled systems form the \textquotedblleft right\textquotedblright ,
or at least the most straightforward, generalization of scalar equations.

For (possibly nonlinear) cooperative systems, results analogous to
\thmref{Traveling_waves} \enuref{Existence_minimal_wave_speed},
\enuref{Persistence_at_the_back_of_the_front}, \thmref{Spreading_speed}
and \thmref{Characterization_minimal_speed} were established by Lewis,
Li and Weinberger \cite{Li_Weinberger_,Weinberger_Lew}. Recently,
Al-Kiffai and Crooks \cite{Al_Kiffai_Croo} introduced a convective
term into a two-species cooperative system to study its influence
on linear determinacy.

For non-cooperative systems that can still be controlled from above
and from below by weakly and fully coupled systems whose linearizations
at $\mathbf{0}$ coincide with that of the non-cooperative system,
Wang \cite{Wang_2011} recovered the results of Lewis\textendash Li\textendash Weinberger
by comparison arguments. Before going any further, let us point out
that we will use extensively comparison arguments as well, nevertheless
we will not need equality of the linearizations at $\mathbf{0}$.
This is a crucial difference between the two sets of assumptions.
To illustrate this claim, let us present an explicit example of system
covered by our assumptions and not by Wang\textquoteright s ones:
take any $N\geq3$, $r>0$, $\mu\in\left(0,\frac{r}{2}\right)$ and
define $\mathbf{L}$ and $\mathbf{c}$ as follows:
\[
\mathbf{L}=r\mathbf{I}+\mu\left(\begin{matrix}-1 & 1 & 0 & \dots & 0\\
1 & -2 & \ddots & \ddots & \vdots\\
0 & \ddots & \ddots & \ddots & 0\\
\vdots & \ddots & \ddots & -2 & 1\\
0 & \dots & 0 & 1 & -1
\end{matrix}\right),
\]
\[
\mathbf{c}:\mathbf{v}\mapsto\mathbf{1}\mathbf{v}.
\]
On one hand, $\left(H_{1}\right)$\textendash $\left(H_{4}\right)$
are easily verified, but on the other hand, the function $\mathbf{f}:\mathbf{v}\mapsto\mathbf{L}\mathbf{v}-\mathbf{c}\left[\mathbf{v}\right]\circ\mathbf{v}$
is such that, for all $i\in\left[N\right]\backslash\left\{ 1,N\right\} $
and all $\mathbf{v}\in\mathsf{K}^{++}$,
\[
\frac{\partial\mathbf{f}_{i}}{\partial v_{j}}\left(\mathbf{v}\right)=-v_{i}<0\text{ for all }j\in\left[N\right]\backslash\left\{ i-1,i,i+1\right\} .
\]
 Consequently, the application $v\mapsto\mathbf{f}_{i}\left(v\mathbf{e}_{j}\right)$
is decreasing in $[0,+\infty)$. This clearly violates Wang\textquoteright s
assumptions: this instance of $\left(E_{KPP}\right)$ cannot be controlled
from below by a cooperative system whose linearization at $\mathbf{0}$
is $\partial_{t}\mathbf{u}-\mathbf{D}\partial_{xx}\mathbf{u}=\mathbf{L}\mathbf{u}$. 

Even if $\mathbf{L}$ is essentially positive and the cooperative
functions $\mathbf{f}^{-},\mathbf{f}^{+}$ satisfying
\[
\left\{ \begin{matrix}\mathbf{f}^{-}\left(\mathbf{v}\right)\leq\mathbf{L}\mathbf{v}-\mathbf{c}\left(\mathbf{v}\right)\circ\mathbf{v}\leq\mathbf{f}^{+}\left(\mathbf{v}\right)\\
\mathbf{f}^{-}\left(\mathbf{0}\right)=\mathbf{f}^{+}\left(\mathbf{0}\right)=\mathbf{0}\\
D\mathbf{f}^{-}\left(\mathbf{0}\right)=D\mathbf{f}^{+}\left(\mathbf{0}\right)=\mathbf{L}
\end{matrix}\right.
\]
are constructible, in general it is difficult to verify that $\mathbf{f}^{-}$
and $\mathbf{f}^{+}$ have each a minimal positive zero (another requirement
of Wang). Our setting needs not such a verification. 

Furthermore, even if these minimal zeros exist, several results presented
here are still new.
\begin{enumerate}
\item \thmref{Traveling_waves} \ref{enu:Existence_minimal_wave_speed}
adds to \cite[Theorem 2.1 iii)--v)]{Wang_2011} the existence of a
critical traveling wave (Wang obtained the existence of a bounded
non-constant nonnegative solution traveling at speed $c^{\star}$
but the limit at $+\infty$ of its profile was not addressed). 
\item \thmref{Strong_positivity}, \thmref{Absorbing_set}, \thmref{Extinction_or_persistence}
and \thmref{Existence_of_steady_states} as well as \thmref{Traveling_waves}
\enuref{Upper_bound_for_the_profiles}, \ref{enu:Monotonicity_at_the_edge_of_the_front}
rely more deeply on the KPP structure and are completely new to the
best of our knowledge. 
\end{enumerate}

\subsubsection{KPP systems}

Regarding weakly coupled systems equipped with KPP nonlinearities,
as far as we know most related works assume the essential positivity
of $\mathbf{L}$, some even requiring its positivity. Our results
tend to show that this collection of results should be generalizable
to the whole class of irreducible and essentially nonnegative $\mathbf{L}$
$\left(H_{1}\right)$ provided $\lambda_{PF}\left(\mathbf{L}\right)>0$. 

Dockery, Hutson, Mischaikow and Pernarowski \cite{Dockery_1998} studied
in a celebrated paper the solutions of $\left(S_{KPP}\right)$ in
a bounded and smooth domain with Neumann boundary conditions. Their
matrix $\mathbf{L}$ had the specific form $a\left(x\right)\mathbf{I}+\mu\mathbf{M}$
where $a$ is a non-constant function of the space variable and with
minimal assumptions on the constant matrix $\mathbf{M}$. They also
assumed strict ordering of the components of $\mathbf{d}$, explicit
and symmetric Lotka\textendash Volterra competition, vanishingly small
$\mu$. They proved the existence of a unique positive steady state,
globally attractive for the Cauchy problem with positive initial data,
and which converges as $\mu\to0$ to a steady state where only $u_{1}$
persists. 

More recently, the solutions of $\left(S_{KPP}\right)$, still in
a bounded and smooth domain with Neumann boundary conditions, were
studied under the assumptions of essential positivity of $\mathbf{L}$
and small Lipschitz constant of $\mathbf{v}\mapsto\mathbf{c}\left(\mathbf{v}\right)\circ\mathbf{v}$
by Hei and Wu \cite{Hei_Wu_2005}. They established by means of super-
and sub-solutions the equivalence between the negativity of the principal
eigenvalue of $-\mathbf{D}\frac{\text{d}^{2}}{\text{d}x^{2}}-\mathbf{L}$
and the existence of a positive steady state.

Provided the positivity of $\mathbf{L}$, the vanishing viscosity
limit of $\left(E_{KPP}\right)$ is the object of a work by Barles,
Evans and Souganidis \cite{Barles_Evans_S}. Although their paper
and the present one differ both in results and in techniques, they
share the same ambition: describing the spreading phenomenon for KPP
systems. Therefore our feeling is that together they give a more complete
answer to the problem.

For two-component systems with explicit Lotka\textendash Volterra
competition, $\mathbf{D}=\mathbf{I}_{2}$ and symmetric and positive
$\mathbf{L}$, \thmref{Existence_of_steady_states} and \thmref{Traveling_waves}
\ref{enu:Existence_minimal_wave_speed}, \ref{enu:Persistence_at_the_back_of_the_front},
\ref{enu:Monotonicity_at_the_edge_of_the_front} reduce to the results
of Griette and Raoul \cite{Griette_Raoul} (see Alfaro\textendash Griette
\cite{Alfaro_Griette} for a partial extension to space-periodic media).
Their paper uses very different arguments (topological degree, explicit
computations involving in particular the sum of the equations, weak
mutation limit, phase plane analysis) but was our initial motivation
to work on this question: our intent is really to extend their result
to a larger setting by changing the underlying mathematical techniques.
Let us emphasize that they obtained an algebraic formula for the minimal
wave speed, $c^{\star}=2\sqrt{\lambda_{PF}\left(\mathbf{L}\right)}$,
that we are able to generalize (\thmref{Characterization_minimal_speed}).
The case $\mathbf{D}\neq\mathbf{I}_{2}$ has been investigated heuristically
and numerically by Elliott and Cornell \cite{Elliott_Cornel}, who
considered the weak mutation limit as well and obtained further results. 

Let us point out that the problem of the spreading speed for the Cauchy
problem for the two-component system with explicit Lotka\textendash Volterra
competition was formulated but left open by Elliott and Cornell \cite{Elliott_Cornel}
as well as by Cosner \cite{Cosner_14} and not considered by Griette
and Raoul \cite{Griette_Raoul}. This problem is completely solved
here (see \thmref{Spreading_speed}). 

Just after the submission of this paper, a paper by Moris, Börger
and Crooks \cite{Moris_Borger_Crooks} submitted concurrently and
devoted to the analytical confirmation of Elliott and Cornell\textquoteright s
numerical observations was brought to our attention. By applying\textquoteright s
successfully Wang\textquoteright s framework, they obtained the existence
of traveling waves as well as the spreading speed for the Cauchy problem.
However, in order to apply Wang\textquoteright s framework, they had
to make additional assumptions (roughly speaking, small interphenotypic
competition and small mutations) and which are in fact, in view of
our results, unnecessary. They also obtained very interesting results
regarding the dependency on the mutation rate $\mu$ of the spreading
speed 
\[
\lambda_{PF}\left(\mu_{c^{\star}}\mathbf{D}+\mu_{c^{\star}}^{-1}\left(\text{diag}\mathbf{r}+\mu\mathbf{M}\right)\right)
\]
 and the associated distribution 
\[
\mathbf{n}_{PF}\left(\mu_{c^{\star}}\mathbf{D}+\mu_{c^{\star}}^{-1}\left(\text{diag}\mathbf{r}+\mu\mathbf{M}\right)\right).
\]

\subsection{From systems to non-local equations, from mathematics to applications}

It is well-known that systems can be seen as discretizations of continuous
models. In this subsection, we present briefly some equations structured
not only in time and space but also with a third variable and whose
natural discretizations are particular instances of our system $\left(E_{KPP}\right)$
satisfying the criterion $\lambda_{PF}\left(\mathbf{L}\right)>0$.
Our results bring therefore indirect insight into the spreading properties
of these equations.

Since these examples provide also examples of biomathematical applications
of our results, this subsection gives us the opportunity to present
more precisely these applications, to explain how non-cooperative
KPP systems arise in modeling situations and finally to comment our
results from this application point of view. Several fields of biology
are concerned: evolutionary invasion analysis (also known as adaptive
dynamics), population dynamics, epidemiology. Applications in other
sciences might also exist.

\subsubsection{The cane toads equation with non-local competition}

Recall the definition of the discrete laplacian in a finite domain
of cardinal $N$, 
\[
\mathbf{M}_{Lap,N}=\left(\begin{matrix}-1 & 1 & 0 & \dots & 0\\
1 & -2 & \ddots & \ddots & \vdots\\
0 & \ddots & \ddots & \ddots & 0\\
\vdots & \ddots & \ddots & -2 & 1\\
0 & \dots & 0 & 1 & -1
\end{matrix}\right)\text{ if }N\geq3,
\]
\[
\mathbf{M}_{Lap,2}=\left(\begin{matrix}-1 & 1\\
1 & -1
\end{matrix}\right)\text{ if }N=2.
\]

With this notation, the Lotka\textendash Volterra mutation\textendash competition\textendash diffusion
system exhibited earlier reads

\[
\partial_{t}\mathbf{u}-\mathbf{D}\partial_{xx}\mathbf{u}=\text{diag}\left(\text{\ensuremath{\mathbf{r}}}\right)\mathbf{u}+\mu\mathbf{M}_{Lap}\mathbf{u}-\left(\mathbf{C}\mathbf{u}\right)\circ\mathbf{u}.
\]

An especially interesting instance of it is the system where:
\begin{itemize}
\item for all $i\in\left[N\right]$, $d_{N,i}=\underline{\theta}+\left(i-1\right)\theta_{N}$
with $\theta_{N}=\frac{\overline{\theta}-\underline{\theta}}{N-1}$
and with some fixed $\overline{\theta}>\underline{\theta}>0$;
\item $\mathbf{r}_{N}=r\mathbf{1}_{N,1}$ with some fixed $r>0$;
\item $\mu_{N}=\frac{\alpha}{\theta_{N}^{2}}$ with some fixed $\alpha>0$;
\item $\mathbf{C}_{N}=\theta_{N}\mathbf{1}_{N}$.
\end{itemize}
Since $\lambda_{PF}\left(\mathbf{M}_{Lap,N}\right)=0$ (because $\mathbf{M}_{Lap,N}\mathbf{1}_{N,1}=\mathbf{0}$),
the Perron\textendash Frobenius eigenvalue of $\mathbf{L}$ is positive
indeed:
\[
\lambda_{PF}\left(r\mathbf{I}_{N}+\frac{\alpha}{\theta_{N}^{2}}\mathbf{M}_{Lap,N}\right)=r+\lambda_{PF}\left(\frac{\alpha}{\theta_{N}^{2}}\mathbf{M}_{Lap,N}\right)=r>0.
\]

As $N\to+\infty$, this system converges (at least formally) to the
cane toads equation with non-local competition and bounded phenotypes:
\[
\left\{ \begin{matrix}\partial_{t}n-\theta\partial_{xx}n-\alpha\partial_{\theta\theta}n=n\left(t,x,\theta\right)\left(r-\int_{\underline{\theta}}^{\overline{\theta}}n\left(t,x,\theta'\right)\text{d}\theta'\right)\\
\partial_{\theta}n\left(t,x,\underline{\theta}\right)=\partial_{\theta}n\left(t,x,\overline{\theta}\right)=0\text{ for all }\left(t,x\right)\in\mathbb{R}^{2}
\end{matrix}\right.
\]
 where $n$ is a function of $\left(t,x,\theta\right)$, $\theta\in\left[\underline{\theta},\overline{\theta}\right]$
is the motility trait, $\alpha$ is the mutation rate and $\int_{\underline{\theta}}^{\overline{\theta}}n\left(t,x,\theta'\right)\text{d}\theta'$
is the total population present at $\left(t,x\right)$. 

This equation is named after an invasive species currently invading
Australia. A startling ecologic fact is that this invasion is accelerating
whereas biological invasions usually occur at a constant speed (as
predicted by the KPP equation). However this issue is solved when
the phenotypical structure is taken into account and the following
spatial sorting phenomenon is understood: the fastest toads lead the
invasion, reproduce at the edge of the front, give birth to a new
generation of toads among which faster and slower toads can be found
(as a result of mutations), and the new fastest toads take the lead
of the invasion. 

The introduction of a motility trait $\theta$ with a local mutation
term $\alpha\partial_{\theta\theta}n$ into the scalar KPP equation
is then a way of verifying this theory: does it lead to accelerating
invasions? The answer is positive (transitory acceleration up to a
constant asymptotic speed if $\overline{\theta}<+\infty$, constant
acceleration if $\overline{\theta}=+\infty$) and this is why the
cane toads equation achieved some fame (we refer for instance to \cite{Benichou_Calvez,Bouin_Calvez_2014,Bouin_Calvez_2,Bouin_Henderso},
where more detailed modeling explanations can also be found).

The overcrowding effect, which is nowadays standardly taken into account
in population biology modeling, is modeled by the term $-n\left(t,x,\theta\right)\int_{\underline{\theta}}^{\overline{\theta}}n\left(t,x,\theta'\right)\text{d}\theta'$
which basically considers that one given toad competes with all other
toads surrounding it, independently of their phenotype, and does not
compete with distant toads. Mathematically, this term is the only
responsible for the nonlinearity, non-locality and non-cooperativity
of the model: it could be tempting to neglect it. However, linear
growth models (which go back to Malthus) generically lead to exponential
blow-up. The basic idea of the literature about the cane toads equation
is then exactly the same as the one we are going to use in the forthcoming
proofs: point out and use the KPP nature of the problem. 

The results of the present paper are consistent with the ones for
the cane toads equation with bounded phenotypes. Therefore it might
be possible, in a future sequel providing new estimates uniform with
respect to $N$, to rigorously derive the cane toads equation as the
continuous limit of a family of KPP systems. Since the discrete version
is easier to study, new results might be unfolded by this approach.
However, let us stress that the problem of finding these new uniform
estimates is not to be underestimated and is expected to be a very
difficult one. At least regarding biologists, whose field measurements
somehow always produce discrete classes of phenotypes instead of a
continuum of phenotypes, our results bring forth an interesting new
lead to address the general problem of adaptive dynamics. 

Let us point out that if, instead of phenotypes of cane toads, the
components of $\mathbf{u}$ model different strains of virus, then
we obtain an epidemiological model representing the invasion of a
population of sane individuals by a structured population of infected
individuals (Griette\textendash Raoul \cite{Griette_Raoul}). 

Notice that this cane toads equation is only the first step of a larger
research program: a more realistic model should replace clonal reproduction
by sexual reproduction and should take into account the possibility
of non-constant coefficients $\alpha$ and $r$ as well as that of
a more general competition term (logistic with a non-constant weight
or even non-logistic). It is also interesting to consider non-local
spatial or phenotypical dispersion. 

\subsubsection{The cane toads equation with non-local mutations and competition}

Actually, historically, the cane toads equation comes from a doubly
non-local model due to Prévost \textit{et al.} \cite{Arnold_Desvill,PrevostPhD}
(see also the earlier individual-based model by Champagnat and Méléard
\cite{Champagnat_Mel}). Since the non-local mutation operator is
too difficult to handle mathematically, the cane toads equation with
local mutations was favored as a simplified first approach. However
it remains unsatisfying from the modeling point of view and non-local
kernels, which could take into account large mutations, are the real
aim. 

Defining as above $\theta_{N}=\frac{\overline{\theta}-\underline{\theta}}{N-1}$
and $\left(\theta_{i}\right)_{i\in\left[N\right]}=\left(\underline{\theta}+\left(i-1\right)\theta_{N}\right)_{\ddot{\imath}\in\left[N\right]}$,
the natural discretization of the doubly non-local cane toads equation,
\[
\partial_{t}n-d\left(\theta\right)\partial_{xx}n=rn+\alpha\left(K\star_{\theta}n-n\right)-n\int_{\underline{\theta}}^{\overline{\theta}}n\left(t,x,\theta'\right)\text{d}\theta'
\]
with $d\in\mathscr{C}\left(\left[\underline{\theta},\overline{\theta}\right],\left(0,+\infty\right)\right)$
and $K\in\mathscr{C}\left(\mathbb{R},[0,+\infty)\right)$, is
\[
\partial_{t}\mathbf{u}-\mathbf{D}_{N}\partial_{xx}\mathbf{u}=\mathbf{L}_{N}\mathbf{u}-\left(\theta_{N}\mathbf{1}_{N}\mathbf{u}\right)\circ\mathbf{u},
\]
 with
\[
\mathbf{d}_{N}=\left(d\left(\theta_{i}\right)\right)_{i\in\left[N\right]},
\]
\begin{align*}
\mathbf{L}_{N} & =r\mathbf{I}_{N}+\alpha\left(\theta_{N}\left(K\left(\theta_{i}-\theta_{j}\right)\right)_{\left(i,j\right)\in\left[N\right]^{2}}-\mathbf{I}_{N}\right)\\
 & =\left(r-\alpha\right)\mathbf{I}_{N}+\alpha\theta_{N}\left(K\left(\left(i-j\right)\theta_{N}\right)\right)_{\left(i,j\right)\in\left[N\right]^{2}}.
\end{align*}

The assumptions on $\mathbf{c}$ $\left(H_{2}\right)$\textendash $\left(H_{4}\right)$
are obviously satisfied and, as soon as, say, $K$ is positive, the
assumption on $\mathbf{L}$$\left(H_{1}\right)$ is satisfied as well.
Subsequently, $\lambda_{PF}\left(\mathbf{L}_{N}\right)\geq r-\alpha$,
whence $r>\alpha$ is a sufficient condition to ensure $\lambda_{PF}\left(\mathbf{L}_{N}\right)>0$
for all $N\in\mathbb{N}$.

More generally, the system corresponding to the following equation
(see Prévost \textit{et al.} \cite{Arnold_Desvill,PrevostPhD}):
\begin{align*}
\partial_{t}n-d\left(\theta\right)\partial_{xx}n & =r\left(\theta\right)n\left(t,x,\theta\right)+\int_{\underline{\theta}}^{\overline{\theta}}n\left(t,x,\theta'\right)K\left(\theta,\theta'\right)\text{d}\theta'\\
 & -n\left(t,x,\theta\right)\int_{\underline{\theta}}^{\overline{\theta}}n\left(t,x,\theta'\right)C\left(\theta,\theta'\right)\text{d}\theta'
\end{align*}
with $d\in\mathscr{C}\left(\left[\underline{\theta},\overline{\theta}\right],\left(0,+\infty\right)\right)$,
$r\in\mathscr{C}\left(\left[\underline{\theta},\overline{\theta}\right],[0,+\infty)\right)$,
$K,C\in\mathscr{C}\left(\left[\underline{\theta},\overline{\theta}\right]^{2},[0,+\infty)\right)$
is
\[
\partial_{t}\mathbf{u}-\mathbf{D}_{N}\partial_{xx}\mathbf{u}=\mathbf{L}_{N}\mathbf{u}-\left(\mathbf{C}_{N}\mathbf{u}\right)\circ\mathbf{u},
\]
 with
\[
\mathbf{d}_{N}=\left(d\left(\theta_{i}\right)\right)_{i\in\left[N\right]},
\]
\[
\mathbf{L}_{N}=\text{diag}\left(r\left(\theta_{i}\right)\right)_{i\in\left[N\right]}+\theta_{N}\left(K\left(\theta_{i},\theta_{j}\right)\right)_{\left(i,j\right)\in\left[N\right]^{2}},
\]
\[
\mathbf{C}_{N}=\theta_{N}\left(C\left(\theta_{i},\theta_{j}\right)\right)_{\left(i,j\right)\in\left[N\right]^{2}}.
\]

Again, $\left(H_{3}\right)$ and $\left(H_{4}\right)$ are clearly
satisfied, $\left(H_{2}\right)$ is satisfied if $C$ is nonnegative
and both $\left(H_{1}\right)$ and $\lambda_{PF}\left(\mathbf{L}_{N}\right)>0$
are satisfied if, say, $K$ is positive.

In both cases, of course, the positivity of $K$ is a far from necessary
condition and might be relaxed.

To the best of our knowledge, these doubly non-local equations have
been the object of no study apart from \cite{Arnold_Desvill,PrevostPhD}
and are therefore still very poorly understood. In particular, the
traveling wave problem as well as the spreading problem are completely
open. Consequently, our results are highly valuable when applied to
this system. For mathematicians, they motivate the future work on
the limit $N\to+\infty$. For biologists, they provide new insight
into these modeling problems and show for instance how two different
mutation strategies can be compared and how the spreading speed can
be evaluated. 

\subsubsection{The Gurtin\textendash MacCamy equation with diffusion and overcrowding
effect}

In view of the preceding two examples, it is natural to investigate
the existence of completely different applications, that is applications
not concerned at all with evolutionary biology. Such applications
exist indeed, as shown by this third example. 

Consider the following age-structured equation with diffusion:
\[
\left\{ \begin{matrix}\partial_{t}n+\partial_{a}n-d\left(a\right)\partial_{xx}n=-n\left(t,x,a\right)\left(r\left(a\right)+\int_{0}^{A}n\left(t,x,a'\right)C\left(a,a'\right)\text{d}a'\right)\\
n\left(t,x,0\right)=\int_{a_{m}}^{A}n\left(t,x,a'\right)K\left(a'\right)\text{d}a'\text{ for all }\left(t,x\right)\in\mathbb{R}^{2}\\
n\left(t,x,A\right)=0\text{ for all }\left(t,x\right)\in\mathbb{R}^{2}
\end{matrix}\right.
\]
where $n$ is a function of $\left(t,x,a\right)$, $a\in\left[0,A\right]$
is the age variable, $a_{m}\geq0$ is the maturation age, $A>a_{m}$
is the maximal age, $d\in\mathscr{C}\left(\left[0,A\right],\left(0,+\infty\right)\right)$
is the diffusion rate, $r\in\mathscr{C}\left(\left[0,A\right],\left(0,+\infty\right)\right)$
is the mortality rate, $C\in\mathscr{C}\left(\left[0,A\right]^{2},[0,+\infty)\right)$
is the competition kernel and $K\in\mathscr{C}\left(\left[0,A\right],[0,+\infty)\right)$
is the birth rate. This equation is well-known, at least if $C=0$,
and detailed modeling explanations can be found in the classical Gurtin\textendash MacCamy
references \cite{Gurtin_MacCamy-2,Gurtin_MacCamy-1}. 

Defining
\[
a_{N+1}=\frac{A}{N},
\]
\[
\left(a_{i}\right)_{i\in\left[N\right]}=\left(\left(i-1\right)a_{N+1}\right)_{i\in\left[N\right]},
\]
\[
j_{m,N}=\min\left\{ j\in\left[N\right]\ |\ a_{j}\geq a_{m}\right\} ,
\]
\[
\mathbf{u}\left(t,x\right)=\left(n\left(t,x,a_{i}\right)\right)_{i\in\left[N\right]},
\]
\[
\mathbf{d}_{N}=\left(d\left(a_{i}\right)\right)_{i\in\left[N\right]},
\]
\[
\mathbf{L}_{mortality,N}=-\text{diag}\left(r\left(a_{i}\right)_{i\in\left[N\right]}\right),
\]
\[
\mathbf{L}_{birth,N}=a_{N+1}\left(\begin{matrix}0 & \dots & 0 & K\left(a_{j_{m,N}}\right) & \dots & K\left(a_{N}\right)\\
0 &  &  & \dots &  & 0\\
\vdots &  &  &  &  & \vdots\\
0 &  &  & \dots &  & 0
\end{matrix}\right),
\]
\[
\mathbf{L}_{aging,N}=\frac{1}{a_{N+1}}\left(\begin{matrix}0 & 0 & \dots & \dots & 0\\
1 & -1 & \ddots &  & \vdots\\
0 & \ddots & \ddots & \ddots & \vdots\\
\vdots & \ddots & \ddots & \ddots & 0\\
0 & \dots & 0 & 1 & -1
\end{matrix}\right),
\]
\[
\mathbf{L}_{N}=\mathbf{L}_{mortality,N}+\mathbf{L}_{birth,N}+\mathbf{L}_{aging,N},
\]
\[
\mathbf{C}_{N}=a_{N+1}\left(C\left(a_{i},a_{j}\right)\right)_{\left(i,j\right)\in\left[N\right]^{2}},
\]
 it follows again that
\[
\partial_{t}\mathbf{u}-\mathbf{D}_{N}\partial_{xx}\mathbf{u}=\mathbf{L}_{N}\mathbf{u}-\left(\mathbf{C}_{N}\mathbf{u}\right)\circ\mathbf{u}
\]
is the natural discretization with $\left(H_{3}\right)$ and $\left(H_{4}\right)$
automatically satisfied. $K$ nonnegative nonzero and $C$ nonnegative
are sufficient conditions to enforce $\left(H_{1}\right)$ and $\left(H_{2}\right)$.

Since we have
\[
\lambda_{PF}\left(\mathbf{L}_{N}\right)\geq\lambda_{PF}\left(\mathbf{L}_{birth,N}+\mathbf{L}_{aging,N}\right)-\max_{\left[0,A\right]}r
\]
 and since $\lambda_{PF}\left(\mathbf{L}_{birth,N}+\mathbf{L}_{aging,N}\right)$
is bounded from below by a positive constant independent of $N$ (the
proof of this claim being deliberately not detailed here for the sake
of brevity), if $\max\limits _{\left[0,A\right]}r$ is small enough,
then $\lambda_{PF}\left(\mathbf{L}_{N}\right)>0$ for all $N\in\mathbb{N}$. 

We point out that this KPP system differs noticeably from the Lotka\textendash Volterra
mutation\textendash competition\textendash diffusion system presented
up to now as the main instance of KPP system: here, the matrix $\mathbf{L}$
is highly non-symmetric. This should have important qualitative consequences,
numerically observable. It might even be unexpected that these two
systems share important properties and this makes our theorems even
more interesting. 

As far as we know, the traveling wave problem and the spreading problem
for the continuous age-structured problem are completely open. Therefore
the earlier remarks concerning the impact of our results on the doubly
non-local cane toads equation apply here as well.

\section{Strong positivity}

\thmref{Strong_positivity} is mainly straightforward and follows
from the following local result. 
\begin{prop}
\label{prop:Positivity_for_the_parabolic_problem} Let $\mathsf{Q}\subset\mathbb{R}^{2}$
be a bounded parabolic cylinder and $\mathbf{u}$ be a classical solution
of $\left(E_{KPP}\right)$ set in $\mathsf{Q}$. 

If $\mathbf{u}$ is nonnegative on $\partial_{P}\mathsf{Q}$, then
it is either null or positive in $\mathsf{Q}$. 
\end{prop}

\begin{proof}
Let $K=\max\limits _{\overline{\mathsf{Q}}}\left|\mathbf{u}\right|$
and observe that, for all $i\in\left[N\right]$ and all $\left(t,x\right)\in\mathsf{Q}$,
\[
\left|l_{i,i}-c_{i}\left(\mathbf{u}\left(t,x\right)\right)\right|\leq\left|l_{i,i}\right|+\max_{\mathbf{v}\in\overline{\mathsf{B}\left(\mathbf{0},K\right)}}\left|c_{i}\left(\mathbf{v}\right)\right|.
\]

Then, define
\[
\mathbf{A}:\left(t,x\right)\mapsto\mathbf{L}-\text{diag}\left(\mathbf{c}\left(\mathbf{u}\left(t,x\right)\right)\right).
\]

By the irreducibility and the essential nonnegativity of $\mathbf{L}$
$\left(H_{1}\right)$, $\mathbf{A}\left(t,x\right)$ has these two
properties as well for all $\left(t,x\right)\in\overline{\mathsf{Q}}$.
By the boundedness of $\mathbf{u}$ in $\overline{\mathsf{Q}}$, $\mathbf{A}$
is bounded in $\overline{\mathsf{Q}}$ as well. 

Therefore $\mathbf{u}$ is a solution of the following linear weakly
and fully coupled system with bounded coefficients:
\[
\partial_{t}\mathbf{u}-\mathbf{D}\partial_{xx}\mathbf{u}-\mathbf{A}\mathbf{u}=\mathbf{0}.
\]

By virtue of Protter\textendash Weinberger\textquoteright s strong
maximum principle \cite[Chapter 3, Theorem 13]{Protter_Weinberger},
$\mathbf{u}$ is indeed either null or positive in $\mathsf{Q}$.
\end{proof}
Actually, noticing that the previous proof remains true without any
modification if we add to $\left(E_{KPP}\right)$ a diagonal drift
term $\mathbf{b}\circ\partial_{x}\mathbf{u}$ with $\mathbf{b}\in\mathbb{R}^{N}$,
we state right now a corollary that will be quite useful later on. 
\begin{cor}
\label{cor:Positivity_for_the_elliptic_problem_with_drift} Let $\left(a,b,c\right)\in\mathbb{R}^{3}$
such that $a<b$. Let $\mathbf{u}$ be a nonnegative classical solution
of
\[
-\mathbf{D}\mathbf{u}''-c\mathbf{u}'=\mathbf{L}\mathbf{u}-\mathbf{c}\left[\mathbf{u}\right]\circ\mathbf{u}\text{ in }\left(a,b\right).
\]

Then $\mathbf{u}$ is either null or positive in $\left(a,b\right)$. 
\end{cor}

\begin{rem*}
This statement does not establish the non-negativity of all solutions
of $-\mathbf{D}\mathbf{u}''-c\mathbf{u}'=\mathbf{L}\mathbf{u}-\mathbf{c}\left[\mathbf{u}\right]\circ\mathbf{u}$;
it only enforces the interior positivity of the nonnegative nonzero
solutions. Regarding the weak maximum principle, we refer among others
to Figueiredo \cite{Figueiredo_1994}, Figueiredo\textendash Mitidieri
\cite{Figueiredo_Mit}, Sweers \cite{Sweers_1992}. In view of what
is known in the simpler scalar case, it is to be expected that, for
small $\left|c\right|$ and large enough intervals $\left(a,b\right)$,
sign-changing solutions exist.
\end{rem*}

\section{Absorbing set and upper estimates}

On the contrary, \thmref{Absorbing_set} requires some work.

\subsection{Saturation of the reaction term}

For all $i\in\left[N\right]$, let $\mathsf{H}_{i}\subset\mathbb{R}^{N}$
be the closed half-space defined as
\[
\mathsf{H}_{i}=\left\{ \mathbf{v}\in\mathbb{R}^{N}\ |\ \left(\mathbf{L}\mathbf{v}\right)_{i}\geq0\right\} .
\]
\begin{lem}
\label{lem:Saturation_constant} There exists $\mathbf{k}\in\mathsf{K}^{++}$
such that, for all $i\in\left[N\right]$ and for all $\mathbf{v}\in\mathsf{K}\backslash\mathbf{e}_{i}^{\perp}$,
\[
\left(\mathbf{L}\left(\mathbf{v}+k_{i}\mathbf{e}_{i}\right)-\mathbf{c}\left(\mathbf{v}+k_{i}\mathbf{e}_{i}\right)\circ\left(\mathbf{v}+k_{i}\mathbf{e}_{i}\right)\right)_{i}<0.
\]
\end{lem}

\begin{proof}
Let $i\in\left[N\right]$ and let
\[
\mathsf{F}_{i}=\left(\mathsf{S}^{+}\left(\mathbf{0},1\right)\cap\mathsf{H}_{i}\right)\backslash\mathbf{e}_{i}^{\perp}.
\]

Let
\[
\begin{matrix}f_{i}: & \left(0,+\infty\right)\times\mathsf{S}\left(\mathbf{0},1\right) & \to & \mathbb{R}\\
 & \left(\alpha,\mathbf{n}\right) & \mapsto & \sum\limits _{j=1}^{N}l_{i,j}n_{j}-c_{i}\left(\alpha\mathbf{n}\right)n_{i}.
\end{matrix}
\]

Notice that for all $\mathbf{n}\in\mathsf{S}^{+}\left(\mathbf{0},1\right)\backslash\mathsf{F}_{i}$,
either $\sum\limits _{j=1}^{N}l_{i,j}n_{j}<0$ and then $f_{i}\left(\alpha,\mathbf{n}\right)<0$
for all $\alpha>0$ or $n_{i}=0$ and then $f_{i}\left(\alpha,\mathbf{n}\right)=\sum\limits _{j=1}^{N}l_{i,j}n_{j}\geq0$
does not depend on $\alpha$.

Let $\mathbf{n}\in\mathsf{F}_{i}$. By virtue of the behavior of $\mathbf{c}$
as $\alpha\to+\infty$ $\left(H_{4}\right)$ and since $\mathbf{n}\notin\mathbf{e}_{i}^{\perp}$,
\[
\lim_{\alpha\to+\infty}f_{i}\left(\alpha,\mathbf{n}\right)=-\infty.
\]

Therefore the following quantity is finite and nonnegative:
\[
\alpha_{i,\mathbf{n}}=\inf\left\{ \alpha\geq0\ |\ \forall\alpha'\in\left(\alpha,+\infty\right)\quad f_{i}\left(\alpha',\mathbf{n}\right)<0\right\} .
\]

Now, the set
\[
\left\{ \alpha_{i,\mathbf{n}}n_{i}\ |\ \mathbf{n}\in\mathsf{F}_{i}\right\} =\left\{ \alpha_{i,\mathbf{n}}n_{i}\ |\ \mathbf{n}\in\mathsf{F}_{i},\ \alpha_{i,\mathbf{n}}>\underline{\alpha}\right\} \cup\left\{ \alpha_{i,\mathbf{n}}n_{i}\ |\ \mathbf{n}\in\mathsf{F}_{i},\ \alpha_{i,\mathbf{n}}\leq\underline{\alpha}\right\} 
\]
 is bounded if and only if the set $\left\{ \alpha_{i,\mathbf{n}}n_{i}\ |\ \mathbf{n}\in\mathsf{F}_{i},\ \alpha_{i,\mathbf{n}}>\underline{\alpha}\right\} $
is bounded. Recall the definition of $\underline{\alpha}\geq1$ and
$\delta\geq1$ $\left(H_{4}\right)$. For all $\mathbf{n}\in\mathsf{F}_{i}$
such that $\alpha_{i,\mathbf{n}}>\underline{\alpha}$, thanks to $\left(H_{4}\right)$,
we have by virtue of the discrete Cauchy\textendash Schwarz inequality
\begin{align*}
\left|\alpha_{i,\mathbf{n}}n_{i}\right| & =\alpha_{i,\mathbf{n}}n_{i}\\
 & \leq\alpha_{i,\mathbf{n}}^{\delta}n_{i}\\
 & \leq\frac{\sum_{j=1}^{N}l_{i,j}n_{j}}{\underline{c}_{i}}\\
 & \leq\frac{\left|\left(l_{i,j}\right)_{j\in\left[N\right]}\right|}{\underline{c}_{i}},
\end{align*}
whence the finiteness of
\[
k_{i}=\sup\left\{ \alpha_{i,\mathbf{n}}n_{i}\ |\ \mathbf{n}\in\mathsf{F}_{i}\right\} 
\]
is established. Its positivity follows from the fact that $\mathbf{c}$
vanishes at $\mathbf{0}$ $\left(H_{3}\right)$ which implies that
for all $\mathbf{n}\in\text{int}\mathsf{F}_{i}$, $\alpha_{i,\mathbf{n}}>0$. 

The result about $\mathbf{v}+k_{i}\mathbf{e}_{i}$ with $\mathbf{v}\in\mathsf{K}\backslash\mathbf{e}_{i}^{\perp}$
is a direct consequence. 
\end{proof}
Assuming in addition strict monotonicity of $\alpha\mapsto c_{i}\left(\alpha\mathbf{n}\right)$
(which is for instance satisfied if $\mathbf{c}\left(\mathbf{v}\right)=\mathbf{C}\mathbf{v}$
with $\mathbf{C}\gg\mathbf{0}$, that is in the Lotka\textendash Volterra
competition case), we can obtain the following more precise geometric
description of the reaction term. The proof is quite straightforward
and is not detailed here.
\begin{lem}
Assume in addition that $\alpha\mapsto c_{i}\left(\alpha\mathbf{n}\right)$
is increasing for all $\mathbf{n}\in\mathsf{H}_{i}$. 

Then there exists a collection of connected $\mathscr{C}^{1}$-hypersurfaces
\[
\left(\mathsf{Z}_{i}\right)_{i\in\left[N\right]}\subset\prod\limits _{i=1}^{N}\left(\left(\mathsf{K}^{+}\cap\mathsf{H}_{i}\right)\backslash\mathbf{e}_{i}^{\perp}\right)
\]
 such that, for any $i\in\left[N\right]$ and any $\mathbf{v}\in\left(\mathsf{K}^{+}\cap\mathsf{H}_{i}\right)\backslash\mathbf{e}_{i}^{\perp}$,
\[
\left(\mathbf{L}\mathbf{v}-\mathbf{c}\left(\mathbf{v}\right)\circ\mathbf{v}\right)_{i}=0\text{ if and only if }\mathbf{v}\in\mathsf{Z}_{i}.
\]

For all $i\in\left[N\right]$, $\mathsf{Z}_{i}$ satisfies the following
properties.
\begin{enumerate}
\item For all $\mathbf{n}\in\left(\mathsf{S}^{+}\left(\mathbf{0},1\right)\cap\mathsf{H}_{i}\right)\backslash\mathbf{e}_{i}^{\perp}$,
$\mathsf{Z}_{i}\cap\mathbb{R}\mathbf{n}$ is a singleton. 
\item The function $\mathbf{z}_{i}$ which associates with any $\mathbf{n}\in\left(\mathsf{S}^{+}\left(\mathbf{0},1\right)\cap\mathsf{H}_{i}\right)\backslash\mathbf{e}_{i}^{\perp}$
the unique element of $\mathsf{Z}_{i}\cap\mathbb{R}\mathbf{n}$ is
continuous and is a $\mathscr{C}^{1}$-diffeomorphism of $\left(\mathsf{S}^{++}\left(\mathbf{0},1\right)\cap\text{int}\mathsf{H}_{i}\right)\backslash\mathbf{e}_{i}^{\perp}$
onto $\text{int}\mathsf{Z}_{i}$. 
\item For any $\mathbf{v}\in\mathsf{K}^{+}\backslash\mathbf{e}_{i}^{\perp}$,
$\left(\mathbf{L}\mathbf{v}-\mathbf{c}\left(\mathbf{v}\right)\circ\mathbf{v}\right)_{i}>0$
if and only if
\[
\mathbf{v}\in\mathsf{H}_{i}\text{ and }\left|\mathbf{v}\right|<\left|\mathbf{z}_{i}\left(\frac{\mathbf{v}}{\left|\mathbf{v}\right|}\right)\right|.
\]
\end{enumerate}
\end{lem}

\subsection{Absorbing set and upper estimates}

Define for all $i\in\left[N\right]$
\[
\begin{matrix}g_{i}: & [0,+\infty) & \to & \left(0,+\infty\right)\\
 & \mu & \mapsto & \max\left(\mu,k_{i}\right).
\end{matrix}
\]

The function $g_{i}$ is non-decreasing and piecewise affine (whence
Lipschitz-continuous). 

The following local in space $\mathscr{L}^{\infty}$ estimate for
the parabolic problem is due to Barles\textendash Evans\textendash Souganidis
\cite{Barles_Evans_S}. We repeat its proof for the sake of completeness.
\begin{lem}
\label{lem:Global_bounds_for_the_parabolic_problem} Let $\mathsf{Q}\subset\mathbb{R}^{2}$
be a parabolic cylinder bounded in space and bounded from below in
time.

Let $\mathbf{u}$ be a nonnegative classical solution of $\left(E_{KPP}\right)$
set in $\mathsf{Q}$ such that
\[
\mathbf{u}_{|\partial_{P}\mathsf{Q}}\in\mathscr{L}^{\infty}\left(\partial_{P}\mathsf{Q},\mathbb{R}^{N}\right).
\]

Then we have
\[
\left(\sup_{\mathsf{Q}}u_{i}\right)_{i\in\left[N\right]}\leq\left(g_{i}\left(\sup_{\partial_{P}\mathsf{Q}}u_{i}\right)\right)_{i\in\left[N\right]}.
\]
\end{lem}

\begin{proof}
Let $t_{0}\in\mathbb{R}$, $T\in(0,+\infty]$ and $\left(a,b\right)\in\mathbb{R}^{2}$
such that $\mathsf{Q}=\left(t_{0},t_{0}+T\right)\times\left(a,b\right)$.
Let $i\in\left[N\right]$. 

Define a smooth convex function $\eta:\mathbb{R}\to\mathbb{R}$ which
satisfies
\[
\left\{ \begin{matrix}\eta\left(u\right)=0 & \text{if }u\in(-\infty,g_{i}\left(\sup\limits _{\partial_{P}\mathsf{Q}}u_{i}\right)]\\
\eta\left(u\right)>0 & \text{otherwise}.
\end{matrix}\right.
\]

For all $t\in\left(t_{0},t_{0}+T\right)$, let
\[
\Xi_{i}\left(t\right)=\left\{ x\in\left(a,b\right)\ |\ u_{i}\left(t,x\right)>g_{i}\left(\sup\limits _{\partial_{P}\mathsf{Q}}u_{i}\right)\right\} .
\]

This set is measurable and, by integration by parts, for all $t\in\left(t_{0},t_{0}+T\right)$,
\begin{align*}
\partial_{t}\left(\int_{a}^{b}\eta\left(u_{i}\left(t,x\right)\right)\text{d}x\right) & =\int_{a}^{b}\eta'\left(u_{i}\left(t,x\right)\right)\partial_{t}u_{i}\left(t,x\right)\text{d}x\\
 & =-d_{i}\int_{a}^{b}\eta''\left(u_{i}\left(t,x\right)\right)\left(\partial_{x}u_{i}\left(t,x\right)\right)^{2}\text{d}x\\
 & +\int_{a}^{b}\eta'\left(u_{i}\left(t,x\right)\right)\left(\sum_{j=1}^{N}l_{i,j}u_{j}\left(t,x\right)-c_{i}\left(\mathbf{u}\left(t,x\right)\right)u_{i}\left(t,x\right)\right)\text{d}x\\
 & =-d_{i}\int_{\Xi_{i}\left(t\right)}\eta''\left(u_{i}\left(t,x\right)\right)\left(\partial_{x}u_{i}\left(t,x\right)\right)^{2}\text{d}x\\
 & +\int_{\Xi_{i}\left(t\right)}\eta'\left(u_{i}\left(t,x\right)\right)\left(\sum_{j=1}^{N}l_{i,j}u_{j}\left(t,x\right)-c_{i}\left(\mathbf{u}\left(t,x\right)\right)u_{i}\left(t,x\right)\right)\text{d}x\\
 & \leq0
\end{align*}

Since $\int_{a}^{b}\eta\left(u_{i}\left(t_{0},x\right)\right)\text{d}x=0$,
we deduce
\[
u_{i}\leq g_{i}\left(\sup_{\partial_{P}\mathsf{Q}}u_{i}\right)\text{ in }\mathsf{Q},
\]
whence
\[
\sup_{\mathsf{Q}}u_{i}\leq g_{i}\left(\sup_{\partial_{P}\mathsf{Q}}u_{i}\right).
\]
\end{proof}
As a corollary of this local estimate, we get \thmref{Absorbing_set}.
\begin{prop}
\label{prop:Absorbing_set} Let $\mathbf{u}_{0}\in\mathscr{C}_{b}\left(\mathbb{R},\mathsf{K}\right)$.
Then the unique classical solution $\mathbf{u}$ of $\left(E_{KPP}\right)$
set in $\left(0,+\infty\right)\times\mathbb{R}$ with initial data
$\mathbf{u}_{0}$ satisfies
\[
\left(\sup_{\left(0,+\infty\right)\times\mathbb{R}}u_{i}\right)_{i\in\left[N\right]}\leq\left(g_{i}\left(\sup_{\mathbb{R}}u_{0,i}\right)\right)_{i\in\left[N\right]}
\]
 and furthermore
\[
\left(\limsup_{t\to+\infty}\sup_{x\in\mathbb{R}}u_{i}\left(t,x\right)\right)_{i\in\left[N\right]}\leq\mathbf{g}\left(0\right).
\]

Consequently, all bounded nonnegative classical solutions of $\left(S_{KPP}\right)$
are valued in
\[
\prod_{i=1}^{N}\left[0,g_{i}\left(0\right)\right].
\]
\end{prop}

\begin{proof}
To get the global in space $\mathscr{L}^{\infty}$ estimate, apply
the local one to the family $\left(\mathbf{u}_{R}\right)_{R>0}$,
where $\mathbf{u}_{R}$ is the solution of $\left(E_{KPP}\right)$
set in $\left(0,+\infty\right)\times\left(-R,R\right)$ with
\[
\left\{ \begin{matrix}\mathbf{u}_{R}\left(0,x\right)=\mathbf{u}_{0}\left(x\right) & \text{for all }x\in\left[-R,R\right],\\
\mathbf{u}_{R}\left(t,\pm R\right)=\mathbf{u}_{0}\left(\pm R\right) & \text{for all }t\geq0,
\end{matrix}\right.
\]
and recall that, by classical parabolic estimates (Lieberman \cite{Lieberman_2005})
and a diagonal extraction process, $\left(\mathbf{u}_{R}\right)_{R>0}$
converges up to extraction in $\mathscr{C}_{loc}^{1}\left(\left(0,+\infty\right),\mathscr{C}_{loc}^{2}\left(\mathbb{R},\mathbb{R}^{N}\right)\right)$
to the solution of $\left(E_{KPP}\right)$ set in $\left(0,+\infty\right)\times\mathbb{R}$
with initial data $\mathbf{u}_{0}$.

Next, let us prove that the invariant set
\[
\prod_{i=1}^{N}\left[0,g_{i}\left(0\right)\right]=\prod_{i=1}^{N}\left[0,k_{i}\right]
\]
 is in fact an absorbing set. 

Assume by contradiction that there exists a bounded nonnegative classical
solution $\mathbf{u}$ of $\left(E_{KPP}\right)$ set in $\left(0,+\infty\right)\times\mathbb{R}$
such that there exists $i\in\left[N\right]$ such that
\[
\limsup_{t\to+\infty}\sup_{x\in\mathbb{R}}u_{i}\left(t,x\right)>g_{i}\left(0\right).
\]

Since $\left[0,g_{i}\left(0\right)\right]$ is invariant, it implies
directly
\[
\sup_{x\in\mathbb{R}}u_{i}\left(t,x\right)>g_{i}\left(0\right)\text{ for all }t\geq0.
\]

Using the classical second order condition at any local maximum, it
is easily seen that at any local maximum in space of $u_{i}$, the
time derivative is negative. At any $t>0$ such that there is no local
maximum in space, by $\mathscr{C}^{1}$ regularity of $u_{i}$, $x\mapsto u_{i}\left(t,x\right)$
is either strictly monotonic or piecewise strictly monotonic with
one unique local minimum and consequently it converges to some constant
as $x\to\pm\infty$. At least one of these constants is $\sup\limits _{x\in\mathbb{R}}u_{i}\left(t,x\right)$.
For instance, assume it is the limit at $+\infty$. By classical parabolic
estimates and a diagonal extraction process, there exists $\left(x_{n}\right)_{n\in\mathbb{N}}\in\mathbb{R}^{\mathbb{N}}$
such that $x_{n}\to+\infty$ and such that the following sequence
converges in $\mathscr{C}_{loc}^{1}\left(\left(0,+\infty\right),\mathscr{C}_{loc}^{2}\left(\mathbb{R}\right)\right)$:
\[
\left(\left(t',x\right)\mapsto u_{i}\left(t+t',x+x_{n}\right)\right)_{n\in\mathbb{N}}.
\]
Let $v$ be its limit; by construction, 
\[
v\left(0,x\right)=\sup\limits _{x\in\mathbb{R}}u_{i}\left(t,x\right)\text{ for all }x\in\mathbb{R},
\]
so that
\[
\partial_{xx}v\left(0,x\right)=0\text{ for all }x\in\mathbb{R}.
\]
Using the equation satisfied by $u_{i}$, we obtain
\[
\partial_{t}v\left(0,x\right)<0\text{ for all }x\in\mathbb{R}.
\]
Since this argument does not depend on the choice of the sequence
$\left(x_{n}\right)_{n\in\mathbb{N}}$, we deduce
\[
\limsup_{x\to+\infty}\partial_{t}u_{i}\left(t,x\right)<0.
\]

In all cases, 
\[
t\mapsto\|x\mapsto u_{i}\left(t,x\right)\|_{\mathscr{L}^{\infty}\left(\mathbb{R}\right)}
\]
is a decreasing function, and using the global $\mathscr{L}^{\infty}$
estimate derived earlier, we deduce that
\[
t\mapsto\|u_{i}\|_{\mathscr{L}^{\infty}\left(\left(t,+\infty\right)\times\mathbb{R}\right)}
\]
is a decreasing function as well. Therefore
\[
\limsup_{t\to+\infty}\sup_{x\in\mathbb{R}}u_{i}\left(t,x\right)=\liminf_{t\to+\infty}\sup_{x\in\mathbb{R}}u_{i}\left(t,x\right)=\lim_{t\to+\infty}\sup_{x\in\mathbb{R}}u_{i}\left(t,x\right)>g_{i}\left(0\right).
\]

Now, the sequence
\[
\left(\left(t,x\right)\mapsto u_{i}\left(t+n,x\right)\right)_{n\in\mathbb{N}}
\]
being uniformly bounded in $\mathscr{L}^{\infty}\left(\left(0,+\infty\right)\times\mathbb{R}\right)$,
by classical parabolic estimates and a diagonal extraction process,
it converges up to extraction in $\mathscr{C}_{loc}^{1}\left(\left(0,+\infty\right),\mathscr{C}_{loc}^{2}\left(\mathbb{R}\right)\right)$
to some limit $u_{\infty,i}\in\mathscr{C}^{1}\left(\left(0,+\infty\right),\mathscr{C}^{2}\left(\mathbb{R}\right)\right)$. 

On one hand, by construction, the function
\[
t\mapsto\|x\mapsto u_{\infty,i}\left(t,x\right)\|_{\mathscr{L}^{\infty}\left(\mathbb{R}\right)}
\]
 is constant and larger than $g_{i}\left(0\right)$. But on the other
hand, passing also to the limit the other components of $\left(t,x\right)\mapsto\mathbf{u}\left(t+n,x\right)$
and then repeating the argument used earlier to prove the strict monotonicity
of
\[
t\mapsto\|x\mapsto u_{i}\left(t,x\right)\|_{\mathscr{L}^{\infty}\left(\mathbb{R}\right)},
\]
we deduce the strict monotonicity of
\[
t\mapsto\|x\mapsto u_{\infty,i}\left(t,x\right)\|_{\mathscr{L}^{\infty}\left(\mathbb{R}\right)},
\]
which is an obvious contradiction.

\end{proof}
Quite similarly, we can establish an $\mathscr{L}^{\infty}$ estimate
for $\left(S_{KPP}\right)$, set in a strip, and with an additional
drift. 
\begin{prop}
\label{prop:Global_bounds_for_the_elliptic_problem_with_drift} Let
$\left(a,b,c\right)\in\mathbb{R}^{3}$ such that $a<b$ and $\mathbf{u}$
be a nonnegative classical solution of
\[
-\mathbf{D}\mathbf{u}''-c\mathbf{u}'=\mathbf{L}\mathbf{u}-\mathbf{c}\left[\mathbf{u}\right]\circ\mathbf{u}\text{ in }\left(a,b\right).
\]

Then
\[
\left(\max_{\left[a,b\right]}u_{i}\right)_{i\in\left[N\right]}\leq\left(g_{i}\left(\max_{\left\{ a,b\right\} }u_{i}\right)\right)_{i\in\left[N\right]}.
\]
\end{prop}

\begin{proof}
Assume by contradiction that there exists $i\in\left[N\right]$ such
that
\[
\max_{\left[a,b\right]}u_{i}>g_{i}\left(\max_{\left\{ a,b\right\} }u_{i}\right).
\]

Then there exists $x_{0}\in\left(a,b\right)$ such that
\[
\max_{\left[a,b\right]}u_{i}=u_{i}\left(x_{0}\right)>k_{i}.
\]

There exists $\left(x_{1},x_{2}\right)\in\left(a,b\right)^{2}$ such
that $x_{1}<x_{0}<x_{2}$ and
\[
\left\{ \begin{matrix}u_{i}\left(x\right)>k_{i} & \text{for all }x\in\left(x_{1},x_{2}\right)\\
u_{i}\left(x\right)=\frac{1}{2}\left(k_{i}+u_{i}\left(x_{0}\right)\right) & \text{for all }x\in\left\{ x_{1},x_{2}\right\} .
\end{matrix}\right.
\]

But then we find the inequality
\[
-d_{i}u_{i}''-cu_{i}'\ll0\text{ in }\left(x_{1},x_{2}\right)
\]
 which contradicts the existence of an interior maximum at $x_{0}\in\left(x_{1},x_{2}\right)$.
\end{proof}

\section{Extinction and persistence}

This section is devoted to the proof of \thmref{Extinction_or_persistence}.
The extinction case is mainly straightforward but, because of the
lack of comparison principle, the persistence case is more involved.

\subsection{Extinction}
\begin{prop}
Assume $\lambda_{PF}\left(\mathbf{L}\right)<0$. 

Then all bounded nonnegative classical solutions of $\left(E_{KPP}\right)$
set in $\left(0,+\infty\right)\times\mathbb{R}$ converge asymptotically
in time, exponentially fast, and uniformly in space to $\mathbf{0}$.
\end{prop}

\begin{proof}
It suffices to notice that if $\mathbf{u}$ is a nonnegative bounded
solution of $\left(E_{KPP}\right)$, then $\mathbf{v}:\left(t,x\right)\mapsto\text{e}^{\lambda_{PF}\left(\mathbf{L}\right)t}\mathbf{n}_{PF}\left(\mathbf{L}\right)$
satisfies by virtue of the nonnegativity of $\mathbf{c}$ on $\mathsf{K}$
$\left(H_{2}\right)$
\[
\partial_{t}\left(\mathbf{v}-\mathbf{u}\right)-\mathbf{D}\partial_{xx}\left(\mathbf{v}-\mathbf{u}\right)-\mathbf{L}\left(\mathbf{v}-\mathbf{u}\right)=\mathbf{c}\left[\mathbf{u}\right]\circ\mathbf{u}\geq\mathbf{0}.
\]

Hence, up to a multiplication of $\mathbf{v}$ by a large constant,
the comparison principle (Protter\textendash Weinberger \cite[Chapter 3, Theorem 13]{Protter_Weinberger})
applied to the linear weakly and fully coupled operator $\partial_{t}-\mathbf{D}\partial_{xx}-\mathbf{L}$
in $\left(0,+\infty\right)\times\mathbb{R}$ implies that $\mathbf{0}\leq\mathbf{u}\leq\mathbf{v}$.
The limit easily follows.
\end{proof}

\subsubsection{Regarding the critical case}

The proof for the case $\lambda_{PF}\left(\mathbf{L}\right)<0$ clearly
cannot be adapted if $\lambda_{PF}\left(\mathbf{L}\right)=0$. In
this subsubsection, we briefly explain why the present paper only
conjectures the result in this case.

Let us recall that for the scalar equation $\partial_{t}u-\partial_{xx}u=-u^{2}$,
the comparison principle ensures extinction (by comparison with a
solution of $u'\left(t\right)=-u\left(t\right)^{2}$ with large enough
initial data). Since the comparison principle is not satisfied by
$\left(E_{KPP}\right)$, we cannot hope to generalize this proof and
need to find another method. Still, in view of this scalar result,
it is natural to aim for a proof of extinction.

As a preliminary observation, if some Perron\textendash Frobenius
eigenvectors of $\mathbf{L}$ are zeros of $\mathbf{c}$, then extinction
will not occur in general. Therefore, in order to solve the critical
case, it is necessary to rule out this somehow degenerated case. This
is of course consistent with the critical case for \thmref{Existence_of_steady_states}.

For the non-degenerated non-diffusive system $\mathbf{u}'=\mathbf{L}\mathbf{u}-\mathbf{c}\left[\mathbf{u}\right]\circ\mathbf{u}$,
we know how to handle two particular cases:
\begin{itemize}
\item if $\mathbf{L}$ is symmetric, then the classical Lyapunov function
$V:\mathbf{u}\mapsto\frac{1}{2}\left|\mathbf{u}\right|^{2}$ ensures
extinction;
\item if there exists $\mathbf{a}\in\mathsf{K}^{++}$ such that $\mathbf{c}\left(\mathbf{v}\right)=\left(\mathbf{a}^{T}\mathbf{v}\right)\mathbf{1}_{N,1}$,
the change of unknown
\[
\mathbf{z}:t\mapsto\exp\left(\int_{0}^{t}\left(\mathbf{a}^{T}\mathbf{u}\left(\tau\right)\right)\text{d}\tau\right)\mathbf{u}\left(t\right)
\]
 (exploited for instance by Leman\textendash Méléard\textendash Mirrahimi
\cite[Theorem 1.4]{Leman_Meleard_}) ensures extinction.
\end{itemize}
But even in these special cases, the diffusive system cannot be handled
(as far as we know). 

The first idea of proof (which would be in the parabolic setting an
entropy proof) would involve an integration by parts of $\mathbf{u}^{T}\mathbf{D}\partial_{xx}\mathbf{u}$
and therefore would have to deal with the unboundedness of the space
domain $\mathbb{R}$. In such a situation, the classical trick (multiplication
of $\left(E_{KPP}\right)$ by $\text{e}^{-\varepsilon\left|x\right|}\mathbf{u}\left(t,x\right)^{T}$
instead of $\mathbf{u}\left(t,x\right)^{T}$ so that sufficient integrability
is recovered) brings forth a new problematic term (see for instance
Zelik \cite{Zelik_2003} where this computation is carried on and
only leads to the existence of an absorbing set). Hence, apart from
some particular cases (space-periodic solutions or solutions vanishing
as $x\to\pm\infty$) where we do not have to resort to this trick,
the entropy method does not establish the extinction.

As for the second idea of proof, it is completely ruined by the space
variable: the exponential term now depends also on $x$ and, again,
new problematic terms arise in the equation satisfied by $\mathbf{z}$.

In view of these facts, extinction in the critical case is both a
very natural conjecture and a surprisingly challenging problem (which
would be way beyond the scope of this article). 

\subsection{Persistence}

The first step toward the persistence result is giving some rigorous
meaning to the statement \textquotedblleft if $\lambda_{PF}\left(\mathbf{L}\right)>0$,
then $\mathbf{0}$ is unstable\textquotedblright .

\subsubsection{Slight digression: generalized principal eigenvalues and eigenfunctions
for weakly and fully coupled elliptic systems}
\begin{thm}
\label{thm:Generalized_principal_eigenvalue} Let $\left(n,n'\right)\in\mathbb{N}\cap[1,+\infty)\times\mathbb{N}\cap[2,+\infty)$
and $\mathscr{L}:\mathscr{C}^{2}\left(\mathbb{R}^{n},\mathbb{R}^{n'}\right)\to\mathscr{C}\left(\mathbb{R}^{n},\mathbb{R}^{n'}\right)$
be a second-order elliptic operator, weakly and fully coupled, with
continuous and bounded coefficients. 

Let
\[
\lambda_{1}\left(-\mathscr{L}\right)=\sup\left\{ \lambda\in\mathbb{R}\ |\ \exists\mathbf{v}\in\mathscr{C}^{2}\left(\mathbb{R}^{n},\mathsf{K}_{n'}^{++}\right)\quad-\mathscr{L}\mathbf{v}\geq\lambda\mathbf{v}\right\} \in\overline{\mathbb{R}}.
\]

Then
\[
\lim_{R\to+\infty}\lambda_{1,Dir}\left(-\mathscr{L},\mathsf{B}_{n}\left(\mathbf{0},R\right)\right)=\lambda_{1}\left(-\mathscr{L}\right).
\]

Furthermore, $\lambda_{1}\left(-\mathscr{L}\right)$ is in fact a
finite maximum and there exists a generalized principal eigenfunction,
that is a positive solution of
\[
-\mathscr{L}\mathbf{v}=\lambda_{1}\left(-\mathscr{L}\right)\mathbf{v}.
\]
\end{thm}

\begin{rem*}
The convergence of the Dirichlet principal eigenvalue to the aforementioned
generalized principal eigenvalue as $R\to+\infty$ as well as the
existence of a generalized principal eigenfunction are well-known
for scalar elliptic equations (see Berestycki\textendash Rossi \cite{Berestycki_Ros_1}),
but as far as we know these results do not explicitly appear in the
literature regarding elliptic systems. Still, the proof of Berestycki\textendash Rossi
\cite{Berestycki_Ros_1} uses arguments developed in the celebrated
article by Berestycki\textendash Nirenberg\textendash Varadhan \cite{Berestycki_Nir}
and which have been generalized to weakly and fully coupled elliptic
systems already in order to prove the existence of a Dirichlet principal
eigenvalue in non-necessarily smooth but bounded domains by Birindelli\textendash Mitidieri\textendash Sweers
\cite{Birindelli_Mitidieiri_Sweers}. Hence we only briefly outline
here the proof so that it can be checked that the generalization to
unbounded domains is straightforward. 

It begins with the standard verification of the equality between the
generalized principal eigenvalue as defined above and the Dirichlet
principal eigenvalue for bounded smooth domains (whose existence was
proved for instance by Sweers \cite{Sweers_1992}). Then, since the
generalized principal eigenvalue is, by definition, non-increasing
with respect to the inclusion of the domains, we get that the limit
of the Dirichlet principal eigenvalues as $R\to+\infty$ exists and
is larger than or equal to the generalized principal eigenvalue. It
remains to prove that it is also smaller than or equal to it. This
is done thanks to the family of Dirichlet eigenfunctions $\left(\mathbf{v}_{R}\right)_{R>0}$
associated with the family of Dirichlet principal eigenvalues normalized
by
\[
\min_{i\in\left[n'\right]}v_{i,R}\left(\mathbf{0}\right)=1.
\]
Thanks to Arapostathis\textendash Gosh\textendash Marcus\textquoteright s
Harnack inequality \cite{Araposthathis_} applied to the operator
$\mathscr{L}$, we obtain a locally uniform $\mathscr{L}^{\infty}$
estimate, whence, by virtue of classical elliptic estimates (Gilbarg\textendash Trudinger
\cite{Gilbarg_Trudin}) and a diagonal extraction process, the existence
of a limit, up to extraction, for the family $\left(\mathbf{v}_{R}\right)_{R>0}$
as $R\to+\infty$. This limit $\mathbf{v}_{\infty}$ is nonnegative
nonzero and satisfies
\[
-\mathscr{L}\mathbf{v}_{\infty}=\left[\lim_{R\to+\infty}\lambda_{1,Dir}\left(-\mathscr{L},\mathsf{B}_{n}\left(\mathbf{0},R\right)\right)\right]\mathbf{v}_{\infty}.
\]
Thanks again to Arapostathis\textendash Gosh\textendash Marcus\textquoteright s
Harnack inequality, $\mathbf{v}_{\infty}$ is in fact positive in
$\mathbb{R}^{n}$. Thus, by definition of the generalized principal
eigenvalue, the limit as $R\to+\infty$ is indeed smaller than or
equal to it, and in the end the equality is proved as well as the
existence of a generalized principal eigenfunction $\mathbf{v}_{\infty}$. 
\end{rem*}

\subsubsection{Local instability and persistence}

Let $\gamma\in\left[0,1\right]$. On one hand, as a direct result
of Dancer \cite{Dancer_2009} or Lam\textendash Lou \cite{Lam_Lou_2015},
\[
\lim_{\varepsilon\to0}\lambda_{1,Dir}\left(-\varepsilon^{2}\mathbf{D}\frac{\text{d}^{2}}{\text{d}x^{2}}-\left(\mathbf{L}-\gamma\lambda_{PF}\left(\mathbf{L}\right)\mathbf{I}\right),\mathsf{B}\left(\mathbf{0},1\right)\right)=-\left(1-\gamma\right)\lambda_{PF}\left(\mathbf{L}\right).
\]
On the other hand, by a standard change of variable, 
\[
\lim_{\varepsilon\to0}\lambda_{1,Dir}\left(-\varepsilon^{2}\mathbf{D}\frac{\text{d}^{2}}{\text{d}x^{2}}-\left(\mathbf{L}-\gamma\lambda_{PF}\left(\mathbf{L}\right)\mathbf{I}\right),\mathsf{B}\left(\mathbf{0},1\right)\right)=\lim_{R\to+\infty}\lambda_{1,Dir}\left(-\mathbf{D}\frac{\text{d}^{2}}{\text{d}x^{2}}-\left(\mathbf{L}-\gamma\lambda_{PF}\left(\mathbf{L}\right)\mathbf{I}\right),\mathsf{B}\left(\mathbf{0},R\right)\right).
\]
Therefore, in view of \thmref{Generalized_principal_eigenvalue},
\[
\lambda_{1}\left(-\mathbf{D}\frac{\text{d}^{2}}{\text{d}x^{2}}-\left(\mathbf{L}-\gamma\lambda_{PF}\left(\mathbf{L}\right)\mathbf{I}\right)\right)=-\left(1-\gamma\right)\lambda_{PF}\left(\mathbf{L}\right).
\]

This equality deserves some attention: the generalized principal eigenvalue
of $\mathbf{D}\frac{\text{d}^{2}}{\text{d}x^{2}}+\left(\mathbf{L}-\gamma\lambda_{PF}\left(\mathbf{L}\right)\mathbf{I}\right)$
does not depend on $\mathbf{D}$. Of course, this is reminiscent of
the scalar case, where the equality 
\[
\lambda_{1}\left(-d\frac{\text{d}^{2}}{\text{d}x^{2}}-r\right)=-r
\]
is well-known (and follows for instance from a direct computation
of $\lambda_{1,Dir}\left(-d\frac{\text{d}^{2}}{\text{d}x^{2}}-r,\left(-R,R\right)\right)$
or from the equality with the periodic principal eigenvalue $\lambda_{1,per}\left(-d\frac{\text{d}^{2}}{\text{d}x^{2}}-r\right)$). 

As a corollary, we get the following lemma.
\begin{lem}
\label{lem:Negativity_lambda_Dir_for_large_R} Assume $\lambda_{PF}\left(\mathbf{L}\right)>0$.
Then there exists $\left(R_{0},R_{\nicefrac{1}{2}}\right)\in\left(0,+\infty\right)^{2}$
such that
\[
\lambda_{1,Dir}\left(-\mathbf{D}\frac{\text{d}^{2}}{\text{d}x^{2}}-\mathbf{L},\left(-R_{0},R_{0}\right)\right)<0,
\]
\[
\lambda_{1,Dir}\left(-\mathbf{D}\frac{\text{d}^{2}}{\text{d}x^{2}}-\left(\mathbf{L}-\frac{\lambda_{PF}\left(\mathbf{L}\right)}{2}\mathbf{I}\right),\left(-R_{\nicefrac{1}{2}},R_{\nicefrac{1}{2}}\right)\right)<0.
\]
\end{lem}

\begin{rem*}
In fact, much more precisely, it can be shown that, for all $\gamma\in\left[0,1\right]$,
\[
R\mapsto\lambda_{1,Dir}\left(-\mathbf{D}\frac{\text{d}^{2}}{\text{d}x^{2}}-\left(\mathbf{L}-\gamma\lambda_{PF}\left(\mathbf{L}\right)\mathbf{I}\right),\left(-R,R\right)\right)
\]
is a decreasing homeomorphism from $\left(0,+\infty\right)$ onto
$\left(-\left(1-\gamma\right)\lambda_{PF}\left(\mathbf{L}\right),+\infty\right)$. 
\end{rem*}
By continuity of $\mathbf{c}$ and the fact that it vanishes at $\mathbf{0}$
$\left(H_{3}\right)$, as soon as $\lambda_{PF}\left(\mathbf{L}\right)>0$,
the quantity
\[
\alpha_{\nicefrac{1}{2}}=\max\left\{ \alpha>0\ |\ \forall\mathbf{v}\in\left[0,\alpha\right]^{N}\quad\mathbf{c}\left(\mathbf{v}\right)\leq\frac{\lambda_{PF}\left(\mathbf{L}\right)}{2}\mathbf{1}_{N,1}\right\} 
\]
 is well-defined in $\mathbb{R}$ and is positive. The pair $\left(R_{\nicefrac{1}{2}},\alpha_{\nicefrac{1}{2}}\right)$
will be used repeatedly up to the end of this section.
\begin{lem}
\label{lem:Instability_of_0} Assume $\lambda_{PF}\left(\mathbf{L}\right)>0$.
For all $\mu\in\left(0,\alpha_{\nicefrac{1}{2}}\right)$, let
\[
T_{\mu}=\frac{\ln\alpha_{\nicefrac{1}{2}}-\ln\mu}{-\lambda_{1,Dir}\left(-\mathbf{D}\frac{\text{d}^{2}}{\text{d}x^{2}}-\left(\mathbf{L}-\frac{\lambda_{PF}\left(\mathbf{L}\right)}{2}\mathbf{I}\right),\left(-R_{\nicefrac{1}{2}},R_{\nicefrac{1}{2}}\right)\right)}>0.
\]

For all $\left(t_{0},T,a,b\right)\in\mathbb{R}\times\left(0,+\infty\right)\times\mathbb{R}^{2}$
such that $\frac{b-a}{2}=R_{\nicefrac{1}{2}}$ and for all nonnegative
classical solutions $\mathbf{u}$ of $\left(E_{KPP}\right)$ set in
the bounded parabolic cylinder $\left(t_{0},t_{0}+T\right)\times\left(a,b\right)$,
if
\[
\min_{i\in\left[N\right]}\min_{x\in\left[a,b\right]}u_{i}\left(t_{0},x\right)=\mu,
\]
\[
\max_{i\in\left[N\right]}\max_{\left[t_{0},t_{0}+T\right]\times\left[a,b\right]}u_{i}\leq\alpha_{\nicefrac{1}{2}},
\]
 then $T<T_{\mu}$.
\end{lem}

\begin{proof}
Let
\[
\Lambda=\lambda_{1,Dir}\left(-\mathbf{D}\frac{\text{d}^{2}}{\text{d}x^{2}}-\left(\mathbf{L}-\frac{\lambda_{PF}\left(\mathbf{L}\right)}{2}\mathbf{I}\right),\left(-R_{\nicefrac{1}{2}},R_{\nicefrac{1}{2}}\right)\right)<0.
\]

Let $\mathbf{n}$ be the principal eigenfunction associated with the
preceding Dirichlet principal eigenvalue normalized so that
\[
\max_{i\in\left[N\right]}\max\limits _{\left[-R_{\nicefrac{1}{2}},R_{\nicefrac{1}{2}}\right]}n_{i}=1.
\]
By definition, we have in $\left(-R_{\nicefrac{1}{2}},R_{\nicefrac{1}{2}}\right)$
\[
-\left(-\mathbf{D}\mathbf{n}''-\left(\mathbf{L}-\frac{\lambda_{PF}\left(\mathbf{L}\right)}{2}\mathbf{I}\right)\mathbf{n}\right)=-\Lambda\mathbf{n}\gg\mathbf{0}.
\]

By definition of $\alpha_{\nicefrac{1}{2}}$ and by the nonnegativity
of $\mathbf{c}$ on $\mathsf{K}$ $\left(H_{2}\right)$, for all $\mathbf{v}\in\left[0,\alpha_{\nicefrac{1}{2}}\right]^{N}$,
\[
\mathbf{c}\left(\mathbf{v}\right)\circ\mathbf{v}\leq\frac{\lambda_{PF}\left(\mathbf{L}\right)}{2}\mathbf{v},
\]
 whence
\[
-\left(\mathbf{L}\mathbf{v}-\mathbf{c}\left(\mathbf{v}\right)\circ\mathbf{v}\right)\leq-\left(\mathbf{L}-\frac{\lambda_{PF}\left(\mathbf{L}\right)}{2}\mathbf{I}\right)\mathbf{v}.
\]

Now, fix $\left(t_{0},T,a,b\right)\in\mathbb{R}\times\left(0,+\infty\right)\times\mathbb{R}^{2}$
such that $\frac{b-a}{2}=R_{\nicefrac{1}{2}}$ and $T\geq T_{\mu}$.
Assume by contradiction that there exists a nonnegative solution $\mathbf{u}:\left(t_{0},t_{0}+T\right)\times\left(a,b\right)\to\mathsf{K}$
of $\left(E_{KPP}\right)$ such that the following properties hold
\[
\mu=\min_{i\in\left[N\right]}\min_{x\in\left[a,b\right]}u_{i}\left(t_{0},x\right)>0,
\]
\[
\max_{i\in\left[N\right]}\max_{\left[t_{0},t_{0}+T\right]\times\left[a,b\right]}u_{i}\leq\alpha_{\nicefrac{1}{2}}.
\]

In particular, since $\mu>0$, $\mathbf{u}$ is nonnegative nonzero.

To simplify the notations, hereafter we assume that $t_{0}=0$ and
$\frac{a+b}{2}=0$. The general case is only a matter of straightforward
translations.

Define the function
\[
\mathbf{v}:\left(t,x\right)\mapsto\mu\text{e}^{-\Lambda t}\mathbf{n}\left(x\right).
\]

Clearly
\[
\mathbf{v}\left(0,x\right)\leq\mathbf{u}\left(0,x\right)\text{ for all }x\in\left[a,b\right].
\]
It is easily verified as well that $\mathbf{v}$ satisfies in $\left(0,T_{\mu}\right)\times\left(-R_{\nicefrac{1}{2}},R_{\nicefrac{1}{2}}\right)$
\[
-\left(\partial_{t}\mathbf{v}-\mathbf{D}\partial_{xx}\mathbf{v}-\left(\mathbf{L}-\frac{\lambda_{PF}\left(\mathbf{L}\right)}{2}\mathbf{I}\right)\mathbf{v}\right)\geq\mathbf{0},
\]
whence, by construction of $\alpha_{\nicefrac{1}{2}}$, $\mathbf{w}=\mathbf{u}-\mathbf{v}$
satisfies
\begin{align*}
\partial_{t}\mathbf{w}-\mathbf{D}\partial_{xx}\mathbf{w}-\left(\mathbf{L}-\frac{\lambda_{PF}\left(\mathbf{L}\right)}{2}\mathbf{I}\right)\mathbf{w} & \geq\partial_{t}\mathbf{u}-\mathbf{D}\partial_{xx}\mathbf{u}-\mathbf{L}\mathbf{u}+\mathbf{c}\left[\mathbf{u}\right]\circ\mathbf{u}=\mathbf{0}.
\end{align*}

Most importantly, since by construction
\[
T_{\mu}=\max\left\{ t>0\ |\ \max_{i\in\left[N\right]}\max_{x\in\left[-R_{\nicefrac{1}{2}},R_{\nicefrac{1}{2}}\right]}v_{i}\left(t,x\right)\leq\alpha_{\nicefrac{1}{2}}\right\} ,
\]
there exists $t^{\star}\leq T_{\mu}\leq T$ and $x^{\star}\in\left(-R_{\nicefrac{1}{2}},R_{\nicefrac{1}{2}}\right)$
such that $\mathbf{w}\gg\mathbf{0}$ in $[0,t^{\star})\times\left(-R_{\nicefrac{1}{2}},R_{\nicefrac{1}{2}}\right)$
and $\mathbf{w}\left(t^{\star},x^{\star}\right)\in\partial\mathsf{K}$. 

The strong maximum principle applied to the weakly and fully coupled
linear operator $\partial_{t}-\mathbf{D}\partial_{xx}-\left(\mathbf{L}-\frac{\lambda_{PF}\left(\mathbf{L}\right)}{2}\mathbf{I}\right)$
proves then that $\mathbf{w}=\mathbf{0}$ in $[0,t^{\star})\times\left(-R_{\nicefrac{1}{2}},R_{\nicefrac{1}{2}}\right)$,
which contradicts $\mathbf{w}\left(0,\pm R_{\nicefrac{1}{2}}\right)\gg\mathbf{0}$.
\end{proof}
The persistence result follows.
\begin{prop}
\label{prop:Persistence} Assume $\lambda_{PF}\left(\mathbf{L}\right)>0$. 

There exists $\nu>0$ such that all bounded nonnegative nonzero classical
solutions $\mathbf{u}$ of $\left(E_{KPP}\right)$ set in $\left(0,+\infty\right)\times\mathbb{R}$
satisfy, for all bounded intervals $I\subset\mathbb{R}$,
\[
\left(\liminf_{t\to+\infty}\inf_{x\in I}u_{i}\left(t,x\right)\right)_{i\in\left[N\right]}\geq\nu\mathbf{1}_{N,1}.
\]

Consequently, all bounded nonnegative classical solutions of $\left(S_{KPP}\right)$
are valued in
\[
\prod_{i=1}^{N}\left[\nu,g_{i}\left(0\right)\right].
\]
\end{prop}

\begin{proof}
Let $\mathbf{u}$ be a bounded nonnegative nonzero classical solution
of $\left(E_{KPP}\right)$ set in $\left(0,+\infty\right)\times\mathbb{R}$.
In view of \propref{Absorbing_set}, for all $\varepsilon>0$ there
exists $t_{\varepsilon}\in\left(0,+\infty\right)$ such that
\[
\mathbf{u}\leq\left(\max_{i\in\left[N\right]}\left(g_{i}\left(0\right)\right)+\varepsilon\right)\mathbf{1}_{N,1}\text{ in }\left(t_{\varepsilon},+\infty\right)\times\mathbb{R}.
\]
Let $I\subset\mathbb{R}$ be a bounded interval. Fix temporarily $\varepsilon>0$
and $x\in I$ and define $I_{x}=\left(x-R_{\nicefrac{1}{2}},x+R_{\nicefrac{1}{2}}\right)$. 

A first application of \lemref{Instability_of_0} establishes that
there exists $\hat{t}_{x}\in[t_{\varepsilon},+\infty)$ such that
\[
\max_{i\in\left[N\right]}\max_{y\in\overline{I_{x}}}u_{i}\left(\hat{t}_{x},y\right)=\alpha_{\nicefrac{1}{2}}
\]
 and that there exists $\tau>0$ such that
\[
\max_{i\in\left[N\right]}\max_{y\in\overline{I_{x}}}u_{i}\left(t,y\right)>\alpha_{\nicefrac{1}{2}}\text{ for all }t\in\left(\hat{t}_{x},\hat{t}_{x}+\tau\right).
\]

Hence the following quantity is well-defined in $\left[\hat{t}_{x}+\tau,+\infty\right]$:
\[
t_{1}=\inf\left\{ t\geq\hat{t}_{x}+\tau\ |\ \max_{i\in\left[N\right]}\max_{y\in\overline{I_{x}}}u_{i}\left(t,y\right)<\alpha_{\nicefrac{1}{2}}\right\} .
\]

Assume first $t_{1}<+\infty$. Then by continuity, 
\[
\max_{i\in\left[N\right]}\max_{y\in\overline{I_{x}}}u_{i}\left(t_{1},y\right)=\alpha_{\nicefrac{1}{2}}.
\]

Let 
\[
A_{\mathbf{L},\mathbf{c},\varepsilon}=\max\limits _{\left(i,j\right)\in\left[N\right]^{2}}\left|l_{i,j}\right|+\max\limits _{i\in\left[N\right]}\max\limits _{\mathbf{w}\in\left[0,\max\limits _{i\in\left[N\right]}\left(g_{i}\left(0\right)\right)+\varepsilon\right]^{N}}c_{i}\left(\mathbf{w}\right).
\]
 By virtue of Földes\textendash Polá\v{c}ik\textquoteright s Harnack
inequality \cite{Foldes_Polacik}, there exists $\overline{\kappa}>0$,
dependent only on $N$, $R_{\nicefrac{1}{2}}$, $\min\limits _{i\in\left[N\right]}d_{i}$,
$\max\limits _{i\in\left[N\right]}d_{i}$ and $A_{\mathbf{L},\mathbf{c},\varepsilon}$
such that, for all 
\[
\mathbf{w}\in\mathscr{C}_{b}\left(\left(0,+\infty\right)\times\mathbb{R},\left[0,\max\limits _{i\in\left[N\right]}\left(g_{i}\left(0\right)\right)+\varepsilon\right]^{N}\right),
\]
all nonnegative classical solutions $\mathbf{v}$ of the linear weakly
and fully coupled system with bounded coefficients 
\[
\partial_{t}\mathbf{v}-\mathbf{D}\partial_{xx}\mathbf{v}-\left(\mathbf{L}-\text{diag}\left(\mathbf{c}\left[\mathbf{w}\right]\right)\right)\mathbf{v}=\mathbf{0}\text{ in }I_{x}
\]
 satisfy
\[
\min_{i\in\left[N\right]}\min_{y\in\overline{I_{x}}}v_{i}\left(t_{1}+1,y\right)\geq\overline{\kappa}\max_{i\in\left[N\right]}\max_{y\in\overline{I_{x}}}v_{i}\left(t_{1},y\right).
\]

We stress that $\overline{\kappa}$ does not depend on $\mathbf{w}$.
In particular, taking $\mathbf{w}=\mathbf{v}=\mathbf{u}$, we deduce
\[
\min_{i\in\left[N\right]}\min_{y\in\overline{I_{x}}}u_{i}\left(t_{1}+1,y\right)\geq\overline{\kappa}\alpha_{\nicefrac{1}{2}}.
\]

Of course, up to a shrink of $\overline{\kappa}$, we can assume without
loss of generality $\overline{\kappa}\in\left(0,1\right)$. Then let
\[
T=\frac{-\ln\overline{\kappa}}{-\lambda_{1,Dir}\left(-\mathbf{D}\frac{\text{d}^{2}}{\text{d}x^{2}}-\left(\mathbf{L}-\frac{\lambda_{PF}\left(\mathbf{L}\right)}{2}\mathbf{I}\right),I_{x}\right)}>0.
\]

$T$ does not depend on the choice of $\mathbf{u}$.

A second application of \lemref{Instability_of_0} establishes
\[
\max_{i\in\left[N\right]}\max_{y\in\overline{I_{x}}}u_{i}\left(t_{1}+1+T,y\right)>\alpha_{\nicefrac{1}{2}}.
\]
Hence, defining the sequence $\left(t_{n}\right)_{n\in\mathbb{N}}$
by the recurrence relation
\[
t_{n+1}=\inf\left\{ t\geq t_{n}+1+T\ |\ \max_{i\in\left[N\right]}\max_{y\in\overline{I_{x}}}u_{i}\left(t,y\right)<\alpha_{\nicefrac{1}{2}}\right\} 
\]
and repeating by induction the process, we deduce that any connected
component of
\[
\left\{ t\in\left(\hat{t}_{x},+\infty\right)\ |\ \max_{i\in\left[N\right]}\max_{y\in\overline{I_{x}}}u_{i}\left(t,y\right)<\alpha_{\nicefrac{1}{2}}\right\} 
\]
 is an interval of length smaller than $1+T$. 

A second application of Földes\textendash Polá\v{c}ik\textquoteright s
Harnack inequality shows that there exists $\overline{\sigma_{\varepsilon}}>0$,
dependent only on $N$, $R_{\nicefrac{1}{2}}$, $T$, $\min\limits _{i\in\left[N\right]}d_{i}$,
$\max\limits _{i\in\left[N\right]}d_{i}$ and $A_{\mathbf{L},\mathbf{c},\varepsilon}$
such that, for all $t\in\left(\hat{t}_{x},+\infty\right)$,
\[
\min_{i\in\left[N\right]}\min_{y\in\overline{I_{x}}}u_{i}\left(t+T+2,y\right)\geq\overline{\sigma_{\varepsilon}}\max_{i\in\left[N\right]}\max_{\left(t',y\right)\in\left[t,t+T+1\right]\times\overline{I_{x}}}u_{i}\left(t',y\right),
\]
 whence
\[
\min_{i\in\left[N\right]}\min_{y\in\overline{I_{x}}}u_{i}\left(t,y\right)\geq\overline{\sigma_{\varepsilon}}\alpha_{\nicefrac{1}{2}}\text{ for all }t\in\left(\hat{t}_{_{x}}+T+2,+\infty\right).
\]

Assume next $t_{1}=+\infty$. Then
\[
\max_{i\in\left[N\right]}\max_{y\in\overline{I_{x}}}u_{i}\left(t,y\right)\geq\alpha_{\nicefrac{1}{2}}\text{ for all }t\in\left(\hat{t}_{_{x}},+\infty\right),
\]
and consequently
\[
\min_{i\in\left[N\right]}\min_{y\in\overline{I_{x}}}u_{i}\left(t,y\right)\geq\overline{\sigma_{\varepsilon}}\alpha_{\nicefrac{1}{2}}\text{ for all }t\in\left(\hat{t}_{_{x}}+T+2,+\infty\right).
\]

Since $I$ is bounded and $x\mapsto\hat{t}_{x}$ can be assumed continuous
in $\mathbb{R}$ without loss of generality, it follows
\[
\min_{i\in\left[N\right]}\inf_{y\in I}u_{i}\left(t,y\right)\geq\overline{\sigma_{\varepsilon}}\alpha_{\nicefrac{1}{2}}\text{ for all }t\in\left(\max_{x\in\overline{I}}\left(\hat{t}_{_{x}}\right)+T+2,+\infty\right),
\]
whence
\[
\liminf_{t\to+\infty}\min_{i\in\left[N\right]}\inf_{y\in I}u_{i}\left(t,y\right)\geq\overline{\sigma_{\varepsilon}}\alpha_{\nicefrac{1}{2}}
\]
with $\overline{\sigma_{\varepsilon}}\alpha_{\nicefrac{1}{2}}$ dependent
only on $\varepsilon$. The conclusion follows of course by setting
\[
\nu=\sup_{\varepsilon>0}\left(\overline{\sigma_{\varepsilon}}\right)\alpha_{\nicefrac{1}{2}}.
\]
\end{proof}
\begin{rem*}
We point out that $\max\limits _{x\in\overline{I}}\hat{t}_{_{x}}$
is finite because $I$ is bounded. Of course, in $I=\mathbb{R}$,
this problem becomes a spreading problem (see \propref{Spreading_speed_1}). 
\end{rem*}

\section{Existence of positive steady states}

This section is devoted to the proof of \thmref{Existence_of_steady_states}
. 
\begin{prop}
Assume $\lambda_{PF}\left(\mathbf{L}\right)<0$. Then there exists
no positive classical solution of $\left(S_{KPP}\right)$. 
\end{prop}

\begin{proof}
Recall that the Dirichlet principal eigenvalue is non-increasing with
respect to the zeroth order coefficient.

On one hand, by virtue of the nonnegativity of $\mathbf{c}$ on $\mathsf{K}$
$\left(H_{2}\right)$, we have for all $R>0$ and all $\mathbf{v}\in\mathscr{C}_{b}\left(\mathbb{R},\mathsf{K}^{++}\right)$,
\[
\lambda_{1,Dir}\left(-\mathbf{D}\frac{\text{d}^{2}}{\text{d}x^{2}}-\left(\mathbf{L}-\text{diag}\mathbf{c}\left[\mathbf{v}\right]\right),\left(-R,R\right)\right)\geq\lambda_{1,Dir}\left(-\mathbf{D}\frac{\text{d}^{2}}{\text{d}x^{2}}-\mathbf{L},\left(-R,R\right)\right),
\]
 whence, as $R\to+\infty$,
\[
\lambda_{1}\left(-\mathbf{D}\frac{\text{d}^{2}}{\text{d}x^{2}}-\left(\mathbf{L}-\text{diag}\mathbf{c}\left[\mathbf{v}\right]\right)\right)\geq-\lambda_{PF}\left(\mathbf{L}\right)>0.
\]

On the other hand, any positive steady state $\mathbf{v}$ is also
a generalized principal eigenfunction for the generalized principal
eigenvalue
\[
\lambda_{1}\left(-\mathbf{D}\frac{\text{d}^{2}}{\text{d}x^{2}}-\left(\mathbf{L}-\text{diag}\mathbf{c}\left[\mathbf{v}\right]\right)\right)=0.
\]
\end{proof}
\begin{prop}
Assume $\lambda_{PF}\left(\mathbf{L}\right)=0$ and
\[
\text{span}\left(\mathbf{n}_{PF}\left(\mathbf{L}\right)\right)\cap\mathsf{K}\cap\mathbf{c}^{-1}\left(\left\{ \mathbf{0}\right\} \right)=\left\{ \mathbf{0}\right\} .
\]

Then there exists no bounded positive classical solution of $\left(S_{KPP}\right)$. 
\end{prop}

\begin{rem*}
The forthcoming argument is quite standard in the scalar setting.
We detail it for the sake of completeness. 
\end{rem*}
\begin{proof}
Assume by contradiction that there exists a bounded positive classical
solution $\mathbf{v}$ of $\left(S_{KPP}\right)$. 

By boundedness of $\mathbf{v}$, there exists $\kappa\in\left(0,+\infty\right)$
such that $\kappa\mathbf{n}_{PF}\left(\mathbf{L}\right)-\mathbf{v}\geq\mathbf{0}$
in $\mathbb{R}$. Let
\[
\kappa^{\star}=\inf\left\{ \kappa\in\left(0,+\infty\right)\ |\ \kappa\mathbf{n}_{PF}\left(\mathbf{L}\right)-\mathbf{v}\geq\mathbf{0}\text{ in }\mathbb{R}\right\} .
\]
By positivity of $\mathbf{v}$, $\kappa^{\star}>0$. Let $\left(\kappa_{n}\right)_{n\in\mathbb{N}}\in\left(0,\kappa^{\star}\right)^{\mathbb{N}}$
which converges from below to $\kappa^{\star}$. For all $n\in\mathbb{N}$,
there exists $x_{n}\in\mathbb{R}$ such that
\[
\kappa_{n}\mathbf{n}_{PF}\left(\mathbf{L}\right)-\mathbf{v}\left(x_{n}\right)<\mathbf{0}.
\]
Let 
\[
\mathbf{v}_{n}:x\mapsto\mathbf{v}\left(x+x_{n}\right)\text{ for all }n\in\mathbb{N}.
\]
 By virtue of the global boundedness of $\mathbf{v}$, Arapostathis\textendash Gosh\textendash Marcus\textquoteright s
Harnack inequality \cite{Araposthathis_} applied to the linear weakly
and fully coupled operator with bounded coefficients 
\[
\mathbf{D}\frac{\text{d}^{2}}{\text{d}\xi^{2}}+c\frac{\text{d}}{\text{d}\xi}+\left(\mathbf{L}-\text{diag}\left(\mathbf{c}\left[\mathbf{v}_{n}\right]\right)\right)
\]
and classical elliptic estimates (Gilbarg\textendash Trudinger \cite{Gilbarg_Trudin}),
$\left(\mathbf{v}_{n}\right)_{n\in\mathbb{N}}$ converges up to a
diagonal extraction in $\mathscr{C}_{loc}^{2}$ as $n\to+\infty$
to a nonnegative solution $\mathbf{v}^{\star}$ of $\left(S_{KPP}\right)$.
Moreover, $\mathbf{v}^{\star}$ satisfies
\[
\mathbf{v}^{\star}\leq\kappa^{\star}\mathbf{n}_{PF}\left(\mathbf{L}\right)\text{ in }\mathbb{R},
\]
\[
\kappa^{\star}\mathbf{n}_{PF}\left(\mathbf{L}\right)-\mathbf{v}^{\star}\left(0\right)\in\partial\mathsf{K},
\]
\[
-\left(\mathbf{D}\frac{\text{d}^{2}}{\text{d}x^{2}}+\mathbf{L}\right)\left(\kappa^{\star}\mathbf{n}_{PF}\left(\mathbf{L}\right)-\mathbf{v}^{\star}\right)=\mathbf{c}\left[\mathbf{v}^{\star}\right]\circ\mathbf{v}^{\star}\geq\mathbf{0}\text{ in }\mathbb{R}.
\]
 Applying Arapostathis\textendash Gosh\textendash Marcus\textquoteright s
Harnack inequality \cite{Araposthathis_} to $\mathbf{D}\frac{\text{d}^{2}}{\text{d}x^{2}}+\mathbf{L}$,
we deduce 
\[
\kappa^{\star}\mathbf{n}_{PF}\left(\mathbf{L}\right)=\mathbf{v}^{\star}\text{ in }\mathbb{R}
\]
and subsequently 
\[
\mathbf{c}\left(\kappa^{\star}\mathbf{n}_{PF}\left(\mathbf{L}\right)\right)\circ\kappa^{\star}\mathbf{n}_{PF}\left(\mathbf{L}\right)=-\left(\mathbf{D}\frac{\text{d}^{2}}{\text{d}x^{2}}+\mathbf{L}\right)\mathbf{0}=\mathbf{0},
\]
whence $\mathbf{c}\left(\kappa^{\star}\mathbf{n}_{PF}\left(\mathbf{L}\right)\right)=\mathbf{0}$,
which contradicts directly $\kappa^{\star}>0$. 
\end{proof}
Finally, recall that if $\lambda_{PF}\left(\mathbf{L}\right)>0$,
then the following quantity is well-defined and positive:
\[
\alpha_{\nicefrac{1}{2}}=\max\left\{ \alpha>0\ |\ \forall\mathbf{v}\in\left[0,\alpha\right]^{N}\quad\mathbf{c}\left(\mathbf{v}\right)\leq\frac{\lambda_{PF}\left(\mathbf{L}\right)}{2}\mathbf{1}_{N,1}\right\} .
\]
\begin{prop}
Assume $\lambda_{PF}\left(\mathbf{L}\right)>0$. Then there exists
a solution $\mathbf{v}\in\mathsf{K}^{++}$ of
\[
\mathbf{L}\mathbf{v}=\mathbf{c}\left(\mathbf{v}\right)\circ\mathbf{v}.
\]
\end{prop}

\begin{proof}
By virtue of the Perron\textendash Frobenius theorem, $\mathbf{n}_{PF}\left(\mathbf{L}^{T}\right)\in\mathsf{K}^{++}$. 

There exists $\eta>0$ such that, for all $\mathbf{v}\in\mathsf{K}$,
if $\mathbf{n}_{PF}\left(\mathbf{L}^{T}\right)^{T}\mathbf{v}=\eta$,
then $\mathbf{v}\in\left[0,\alpha_{\nicefrac{1}{2}}\right]^{N}$.
Defining
\[
\mathsf{A}=\left\{ \mathbf{v}\in\mathsf{K}\ |\ \mathbf{n}_{PF}\left(\mathbf{L}^{T}\right)^{T}\mathbf{v}=\eta\right\} ,
\]
it follows that for all $\mathbf{v}\in\mathsf{A}$, 
\[
\mathbf{n}_{PF}\left(\mathbf{L}^{T}\right)^{T}\left(\mathbf{c}\left(\mathbf{v}\right)\circ\mathbf{v}\right)\leq\frac{\lambda_{PF}\left(\mathbf{L}\right)}{2}\eta,
\]
whence
\begin{align*}
\mathbf{n}_{PF}\left(\mathbf{L}^{T}\right)^{T}\left(\mathbf{L}\mathbf{v}-\mathbf{c}\left(\mathbf{v}\right)\circ\mathbf{v}\right) & =\lambda_{PF}\left(\mathbf{L}^{T}\right)\eta-\mathbf{n}_{PF}\left(\mathbf{L}^{T}\right)^{T}\left(\mathbf{c}\left(\mathbf{v}\right)\circ\mathbf{v}\right)\\
 & \geq\frac{\lambda_{PF}\left(\mathbf{L}\right)}{2}\eta,
\end{align*}
which is positive if $\lambda_{PF}\left(\mathbf{L}\right)>0$ is assumed
indeed.

Then, defining the convex compact set
\[
\mathsf{C}=\left\{ \mathbf{v}\in\mathsf{K}\ |\ \mathbf{n}_{PF}\left(\mathbf{L}^{T}\right)^{T}\mathbf{v}\geq\eta\text{ and }\mathbf{v}\leq\mathbf{k}+\mathbf{1}_{N,1}\right\} ,
\]
 it can easily be verified that, for all $\mathbf{v}\in\partial\mathsf{C}$,
\[
\mathbf{n}_{\mathbf{v}}^{T}\left(\mathbf{L}\mathbf{v}-\mathbf{c}\left(\mathbf{v}\right)\circ\mathbf{v}\right)<0
\]
where $\mathbf{n}_{\mathbf{v}}$ is the outward pointing normal. In
particular, there is no solution of $\mathbf{L}\mathbf{v}=\mathbf{c}\left(\mathbf{v}\right)\circ\mathbf{v}$
in $\partial\mathsf{C}$. Also, by convexity, for all $\mathbf{v}\in\partial\mathsf{C}$,
there exists a unique $\delta_{\mathbf{v}}>0$ such that
\[
\mathbf{v}+\delta_{\mathbf{v}}\left(\mathbf{L}\mathbf{v}-\mathbf{c}\left(\mathbf{v}\right)\circ\mathbf{v}\right)\in\partial\mathsf{C}.
\]

Assume by contradiction that there is no solution of $\mathbf{L}\mathbf{v}=\mathbf{c}\left(\mathbf{v}\right)\circ\mathbf{v}$
in $\text{int}\mathsf{C}$. Consequently and by convexity again, for
all $\mathbf{v}\in\text{int}\mathsf{C}$, there exists a unique $\delta_{\mathbf{v}}>0$
such that
\[
\mathbf{v}+\delta_{\mathbf{v}}\left(\mathbf{L}\mathbf{v}-\mathbf{c}\left(\mathbf{v}\right)\circ\mathbf{v}\right)\in\partial\mathsf{C}.
\]
The function
\[
\begin{matrix}\mathsf{C} & \to & \left(0,+\infty\right)\\
\mathbf{v} & \mapsto & \delta_{\mathbf{v}}
\end{matrix}
\]
 is continuous and so is the function
\[
\begin{matrix}\mathsf{C} & \to & \partial\mathsf{C}\\
\mathbf{v} & \mapsto & \mathbf{v}+\delta_{\mathbf{v}}\left(\mathbf{L}\mathbf{v}-\mathbf{c}\left(\mathbf{v}\right)\circ\mathbf{v}\right).
\end{matrix}
\]
According to the Brouwer fixed point theorem, this function has a
fixed point, which of course contradicts the assumption.

Hence there exists indeed a solution in $\text{int}\mathsf{C}\subset\mathsf{K}^{++}$
of
\[
\mathbf{L}\mathbf{v}=\mathbf{c}\left(\mathbf{v}\right)\circ\mathbf{v}.
\]
\end{proof}

\section{Traveling waves}

In this section, we assume $\lambda_{PF}\left(\mathbf{L}\right)>0$
and prove \thmref{Traveling_waves}.

Notice as a preliminary that, for any $\left(\mathbf{p},c\right)\in\mathscr{C}^{2}\left(\mathbb{R},\mathbb{R}^{N}\right)\times[0,+\infty)$,
\[
\mathbf{u}:\left(t,x\right)\mapsto\mathbf{p}\left(x-ct\right)
\]
is a classical solution of $\left(E_{KPP}\right)$ if and only if
$\mathbf{p}$ is a classical solution of
\[
-\mathbf{D}\mathbf{p}''-c\mathbf{p}'=\mathbf{L}\mathbf{p}-\mathbf{c}\left[\mathbf{p}\right]\circ\mathbf{p}\text{ in }\mathbb{R}.\quad\left(TW\left[c\right]\right)
\]

\subsection{The linearized equation}

As usual in KPP-type problems, the linearized equation near $\mathbf{0}$:
\[
-\mathbf{D}\mathbf{p}''-c\mathbf{p}'=\mathbf{L}\mathbf{p}\text{ in }\mathbb{R}\quad\left(TW_{0}\left[c\right]\right)
\]
will bring forth the main informations we need in order to construct
and study the traveling wave solutions. Hence we devote this first
subsection to its detailed study. 
\begin{lem}
\label{lem:Existence_exponential_eigenfunctions} Let $\left(c,\lambda\right)\in\mathbb{R}^{2}$. 

If there exists a classical positive solution $\mathbf{p}$ of 
\[
-\mathbf{D}\mathbf{p}''-c\mathbf{p}'-\left(\mathbf{L}+\lambda\mathbf{I}\right)\mathbf{p}=\mathbf{0}\text{ in }\mathbb{R},\quad\left(TW_{0}\left[c,\lambda\right]\right)
\]
 then there exists $\left(\mu,\mathbf{n}\right)\in\mathbb{R}\times\mathsf{K}^{++}$
such that $\mathbf{q}:\xi\mapsto\text{e}^{-\mu\xi}\mathbf{n}$ is
a classical solution of $\left(TW_{0}\left[c,\lambda\right]\right)$.
\end{lem}

\begin{rem*}
This is of course to be related with the notions of generalized principal
eigenvalue and generalized principal eigenfunction (see \thmref{Generalized_principal_eigenvalue}).
The mere existence of $\mathbf{p}$ enforces
\[
\lambda_{1}\left(-\mathbf{D}\frac{\text{d}^{2}}{\text{d}\xi^{2}}-c\frac{\text{d}}{\text{d}\xi}-\left(\mathbf{L}+\lambda\mathbf{I}\right)\right)\geq0.
\]

The following proof is inspired by Berestycki\textendash Hamel\textendash Roques
\cite[Lemma 3.1]{Berestycki_Ham_2}. 
\end{rem*}
\begin{proof}
Let $\mathbf{p}$ be a classical positive solution of $\left(TW_{0}\left[c,\lambda\right]\right)$. 

Let $\mathbf{v}=\left(\frac{p_{i}'}{p_{i}}\right)_{i\in\left[N\right]}$.
By virtue of Arapostathis\textendash Gosh\textendash Marcus\textquoteright s
Harnack inequality \cite{Araposthathis_} applied to the operator
$\mathbf{D}\frac{\text{d}^{2}}{\text{d}\xi^{2}}+c\frac{\text{d}}{\text{d}\xi}+\left(\mathbf{L}+\lambda\mathbf{I}\right)$,
classical elliptic estimates (Gilbarg\textendash Trudinger \cite{Gilbarg_Trudin})
and invariance by translation of $\left(TW_{0}\left[c,\lambda\right]\right)$,
$\mathbf{v}$ is globally bounded. Let
\[
\Lambda_{i}=\limsup_{\xi\to+\infty}v_{i}\left(\xi\right)\text{ for all }i\in\left[N\right],
\]
\[
\overline{\Lambda}=\max\limits _{i\in\left[N\right]}\Lambda_{i},
\]
 so that
\[
\left(\limsup_{\xi\to+\infty}v_{i}\left(\xi\right)\right)_{i\in\left[N\right]}\leq\overline{\Lambda}\mathbf{1}_{N,1}.
\]

Let $\left(\xi_{n}\right)_{n\in\mathbb{N}}\in\mathbb{R}^{\mathbb{N}}$
such that $\xi_{n}\to+\infty$ and such that there exists $\overline{i}\in\left[N\right]$
such that
\[
v_{\overline{i}}\left(\xi_{n}\right)\to\overline{\Lambda}.
\]

On one hand, let
\[
\hat{\mathbf{p}}_{n}:\xi\mapsto\frac{\mathbf{p}\left(\xi+\xi_{n}\right)}{p_{\overline{i}}\left(\xi_{n}\right)}\text{ for all }n\in\mathbb{N}.
\]
 Once more by virtue of Arapostathis\textendash Gosh\textendash Marcus\textquoteright s
Harnack inequality, the sequence $\left(\hat{\mathbf{p}}_{n}\right)_{n\in\mathbb{N}}$
is locally uniformly bounded. Since all $\hat{\mathbf{p}}_{n}$ solve
$\left(TW_{0}\left[c,\lambda\right]\right)$, by classical elliptic
estimates, $\left(\hat{\mathbf{p}}_{n}\right)_{n\in\mathbb{N}}$ converges
up to a diagonal extraction as $n\to+\infty$ in $\mathscr{C}_{loc}^{2}$.
Let $\hat{\mathbf{p}}_{\infty}$ be its limit. Notice by linearity
of $\left(TW_{0}\left[c,\lambda\right]\right)$ that $\hat{\mathbf{p}}_{\infty}$
is in fact smooth and all its derivatives satisfy $\left(TW_{0}\left[c,\lambda\right]\right)$
as well.

On the other hand, let
\[
\mathbf{w}_{n}=\overline{\Lambda}\hat{\mathbf{p}}_{n}-\hat{\mathbf{p}}_{n}'\text{ for all }n\in\mathbb{N}\cup\left\{ +\infty\right\} .
\]

Notice the following equality:
\[
\mathbf{w}_{n}\left(\xi\right)=\hat{\mathbf{p}}_{n}\left(\xi\right)\circ\left(\overline{\Lambda}\mathbf{1}_{N,1}-\mathbf{v}\left(\xi+\xi_{n}\right)\right)\text{ for all }n\in\mathbb{N}\text{ and }\xi\in\mathbb{R}.
\]

Fix $\xi\in\mathbb{R}$. Recalling
\[
\left(\limsup_{n\to+\infty}v_{i}\left(\xi+\xi_{n}\right)\right)_{i\in\left[N\right]}\leq\left(\limsup_{\zeta\to+\infty}v_{i}\left(\zeta\right)\right)_{i\in\left[N\right]}\leq\overline{\Lambda}\mathbf{1}_{N,1},
\]
 it follows that for all $\varepsilon>0$ there exists $n_{\xi,\varepsilon}\in\mathbb{N}$
such that for all $n\geq n_{\xi,\varepsilon}$, 
\[
\left(\overline{\Lambda}+\varepsilon\right)\mathbf{1}_{N,1}\geq\mathbf{v}\left(\xi+\xi_{n}\right),
\]
whence, for all $n\geq n_{\xi,\varepsilon}$,
\begin{align*}
\mathbf{w}_{n}\left(\xi\right) & \geq-\varepsilon\left(\sup_{m\geq n_{\xi,\varepsilon}}\hat{p}_{m,i}\left(\xi\right)\right)_{i\in\left[N\right]}\\
 & \geq-\varepsilon\left(\sup_{m\in\mathbb{N}}\hat{p}_{m,i}\left(\xi\right)\right)_{i\in\left[N\right]},
\end{align*}
and consequently, passing to the limit $n\to+\infty$ and then $\varepsilon\to0$,
we obtain the non-negativity of $\mathbf{w}_{\infty}\left(\xi\right)$. 

Hence $\mathbf{w}_{\infty}$ is a nonnegative solution of $\left(TW_{0}\left[c,\lambda\right]\right)$
satisfying in addition
\[
w_{\infty,\overline{i}}\left(0\right)=\hat{p}_{\infty,\overline{i}}\left(0\right)\left(\overline{\Lambda}-\lim_{n\to+\infty}v_{\overline{i}}\left(\xi_{n}\right)\right)=0,
\]
whence, again by Arapostathis\textendash Gosh\textendash Marcus\textquoteright s
Harnack inequality, $\mathbf{w}_{\infty}$ is in fact the null function. 

Consequently, $\overline{\Lambda}\hat{\mathbf{p}}_{\infty}=\hat{\mathbf{p}}_{\infty}'$,
that is $\hat{\mathbf{p}}_{\infty}$ has exactly the form
\[
\xi\mapsto\text{e}^{\overline{\Lambda}\xi}\mathbf{n}\text{ with }\mathbf{n}\in\mathbb{R}^{N}.
\]

Since $\hat{\mathbf{p}}_{\infty}$ is nonnegative with $\hat{p}_{\infty,\overline{i}}\left(0\right)=1$
by construction, $\mathbf{n}\in\mathsf{K}^{+}$, and since any nonnegative
nonzero solution of $\left(TW_{0}\left[c,\lambda\right]\right)$ is
positive (\corref{Positivity_for_the_elliptic_problem_with_drift}),
$\mathbf{n}\in\mathsf{K}^{++}$. The proof is ended with $\mu=-\overline{\Lambda}$.
\end{proof}
For all $\mu\in\mathbb{R}$, the matrix $\mu^{2}\mathbf{D}+\mathbf{L}$
is essentially nonnegative irreducible. Define $\kappa_{\mu}=-\lambda_{PF}\left(\mu^{2}\mathbf{D}+\mathbf{L}\right)$
and $\mathbf{n}_{\mu}=\mathbf{n}_{PF}\left(\mu^{2}\mathbf{D}+\mathbf{L}\right)$.

Of course, the interest of the pair $\left(\kappa_{\mu},\mathbf{n}_{\mu}\right)$
lies in the preceding lemma: for all $\left(\mu,\mathbf{n}\right)\in\mathbb{R}\times\mathsf{K}^{++}$,
$\xi\mapsto\text{e}^{-\mu\xi}\mathbf{n}$ is a solution of $\left(TW_{0}\left[c\right]\right)$
if and only if
\[
-\mu^{2}\mathbf{D}\mathbf{n}+\mu c\mathbf{n}-\mathbf{L}\mathbf{n}=\mathbf{0},
\]
 that is, thanks to the Perron\textendash Frobenius theorem, if and
only if $\mu c=-\kappa_{\mu}$ and $\frac{\mathbf{n}}{\left|\mathbf{n}\right|}=\mathbf{n}_{\mu}$.
This most important observation leads naturally to the following study
of the equation $c=-\frac{\kappa_{\mu}}{\mu}$.
\begin{lem}
\label{lem:Definition_c_star} The quantity
\[
c^{\star}=\min_{\mu>0}\left(-\frac{\kappa_{\mu}}{\mu}\right)
\]
 is well-defined and positive. 

Let $c\in[0,+\infty)$. In $\left(-\infty,0\right)$, the equation
$-\frac{\kappa_{\mu}}{\mu}=c$ admits no solution. In $\left(0,+\infty\right)$,
it admits exactly:
\begin{enumerate}
\item no solution if $c<c^{\star}$;
\item one solution $\mu_{c^{\star}}>0$ if $c=c^{\star}$;
\item two solutions $\left(\mu_{1,c},\mu_{2,c}\right)$ if $c>c^{\star}$,
which satisfy moreover
\[
0<\mu_{1,c}<\mu_{c^{\star}}<\mu_{2,c}.
\]
\end{enumerate}
\end{lem}

\begin{rem*}
$c^{\star}$ does not depend on $\mathbf{c}$ and is entirely determined
by $\mathbf{D}$ and $\mathbf{L}$. It will be the minimal speed of
traveling waves and this kind of dependency is strongly reminiscent
of the scalar Fisher\textendash KPP case, where $c^{\star}=2\sqrt{rd}$.
In fact the following proof is mostly a generalization of scalar arguments.
\end{rem*}
\begin{proof}
Of course, $\mu\mapsto-\frac{\kappa_{\mu}}{\mu}$ is odd in $\mathbb{R}\backslash\left\{ 0\right\} $.
It is also positive in $\left(0,+\infty\right)$:
\[
-\frac{\kappa_{\mu}}{\mu}=\frac{1}{\mu}\lambda_{PF}\left(\mu^{2}\mathbf{D}+\mathbf{L}\right)>\frac{1}{\mu}\lambda_{PF}\left(\mathbf{L}\right)>0.
\]
 Therefore it is negative in $\left(-\infty,0\right)$ and in particular
there is no solution of $-\frac{\kappa_{\mu}}{\mu}=c\geq0$ in $\left(-\infty,0\right)$. 

We recall Nussbaum\textquoteright s theorem \cite{Nussbaum_1986}
which proves the convexity of the function $\mu\mapsto\rho\left(\mathbf{A}_{\mu}\right)$
provided:
\begin{itemize}
\item the matrix $\mathbf{A}_{\mu}$ is irreducible,
\item its diagonal entries are convex functions of $\mu$,
\item its off-diagonal entries are nonnegative log-convex functions of $\mu$.
\end{itemize}
These conditions are easily verified for $\mu^{2}\mathbf{D}+\mathbf{L}$
and $\mu\mathbf{D}+\frac{1}{\mu}\mathbf{L}$ (actually, for all $\mu^{-\gamma}\left(\mu^{2}\mathbf{D}+\mathbf{L}\right)$
provided $\gamma\in\left[0,2\right]$). Their spectral radii being
respectively $-\kappa_{\mu}$ and $-\frac{\kappa_{\mu}}{\mu}$, these
are therefore convex functions of $\mu$. Moreover, Nussbaum\textquoteright s
result also proves that these convexities are actually strict. Therefore
$\mu\mapsto-\kappa_{\mu}$ and $\mu\mapsto-\frac{\kappa_{\mu}}{\mu}$
are strictly convex functions in $\left(0,+\infty\right)$.

Now, we investigate the behavior of $-\frac{\kappa_{\mu}}{\mu}$ as
$\mu\to0$ and $\mu\to+\infty$.

By continuity, 
\[
\kappa_{\mu}\to\kappa_{0}\text{ as }\mu\to0,
\]
whence $-\frac{\kappa_{\mu}}{\mu}\to+\infty$ as $\mu\to0$. 

Since $\mu\mapsto-\frac{\kappa_{\mu}}{\mu}$ is convex and positive,
either it is bounded in a neighborhood of $+\infty$ and then it converges
to some nonnegative constant, either it is unbounded in a neighborhood
of $+\infty$ and then it converges to $+\infty$. Assume that it
converges to a finite constant. Notice
\[
\lim_{\mu\to+\infty}\frac{1}{\mu^{2}}\left(\mu^{2}\mathbf{D}+\mathbf{L}\right)=\mathbf{D}.
\]

There exists a family of Perron\textendash Frobenius eigenvectors
of $\mu\mathbf{D}+\frac{1}{\mu}\mathbf{L}$, $\left(\mathbf{m}_{\mu}\right)_{\mu>0}$,
normalized so that $\max\limits _{i\in\left[N\right]}m_{\mu,i}=1$
for all $\mu>0$. Thanks to classical compactness arguments in $\mathbb{R}$
and $\mathbb{R}^{N}$, we can extract a sequence $\left(\mu_{n}\right)_{n\in\mathbb{N}}$
such that $\mu_{n}\to+\infty$, $-\frac{\kappa_{\mu_{n}}}{\mu_{n}^{2}}$
converges to $0$ and $\mathbf{m}_{\mu_{n}}$ converges to some $\mathbf{m}\in\mathsf{K}^{+}$.
We point out that we do not know if $\mathbf{m}\in\mathsf{K}^{++}$,
but from the normalizations, we do know that $\mathbf{m}\in\mathsf{K}^{+}$.
Since $\mathbf{m}$ satisfies $\mathbf{D}\mathbf{m}=\mathbf{0}$ and
since $\mathbf{D}$ is invertible, we get a contradiction. Thus
\[
\lim_{\mu\to+\infty}-\frac{\kappa_{\mu}}{\mu}=+\infty.
\]

Hence $\mu\mapsto-\frac{\kappa_{\mu}}{\mu}$ is a strictly convex
positive function which goes to $+\infty$ as $\mu\to0$ or $\mu\to+\infty$:
it admits necessarily a unique global minimum in $\left(0,+\infty\right)$.
The quantity $c^{\star}$ is well-defined.

Define $\mu_{c^{\star}}>0$ such that
\[
c^{\star}=-\frac{\kappa_{\mu_{c^{\star}}}}{\mu_{c^{\star}}}.
\]
The quantity $\mu_{c^{\star}}$ is uniquely defined by strict convexity.
The function $\mu\mapsto-\frac{\kappa_{\mu}}{\mu}$ is bijective from
$\left(0,\mu_{c^{\star}}\right)$ to $\left(c^{\star},+\infty\right)$
and from $\left(\mu_{c^{\star}},+\infty\right)$ to $\left(c^{\star},+\infty\right)$
as well. This ends the proof.
\end{proof}
Putting together \lemref{Existence_exponential_eigenfunctions} and
\lemref{Definition_c_star}, we get the following important result.
\begin{cor}
\label{cor:Existence_of_solution_of_the_linearized_eq_for_the_waves}
For all $c\in[0,+\infty)$, the set of nonnegative nonzero classical
solutions of $\left(TW_{0}\left[c\right]\right)$ is empty if and
only if $c\in[0,c^{\star})$.
\end{cor}

We can also get the exact values of $c$ for which $\mathbf{0}$ is
an unstable steady state of $\left(TW_{0}\left[c\right]\right)$,
in the sense of \lemref{Negativity_lambda_Dir_for_large_R}. 
\begin{lem}
\label{lem:Value_of_the_generalized_principal_eigenvalue_with_drift}
Let $c\in[0,+\infty)$. Then
\[
\lambda_{1}\left(-\mathbf{D}\frac{\text{d}^{2}}{\text{d}x^{2}}-c\frac{\text{d}}{\text{d}x}-\mathbf{L}\right)=\sup\limits _{\mu\in\mathbb{R}}\left(\kappa_{\mu}+\mu c\right).
\]

Furthermore:
\begin{enumerate}
\item $\sup\limits _{\mu\in\mathbb{R}}\left(\kappa_{\mu}+\mu c\right)=\max\limits _{\mu\geq0}\left(\kappa_{\mu}+\mu c\right)$;
\item $\max\limits _{\mu\geq0}\left(\kappa_{\mu}+\mu c\right)<0$ if and
only if $c<c^{\star}$. 
\end{enumerate}
\end{lem}

\begin{rem*}
Just as in the case $c=0$, it can be shown that, for all $c\in[0,+\infty)$,
\[
R\mapsto\lambda_{1,Dir}\left(-\mathbf{D}\frac{\text{d}^{2}}{\text{d}\xi^{2}}-c\frac{\text{d}}{\text{d}\xi}-\mathbf{L},\left(-R,R\right)\right)
\]
is a decreasing homeomorphism from $\left(0,+\infty\right)$ onto
$\left(\lambda_{1}\left(-\mathbf{D}\frac{\text{d}^{2}}{\text{d}x^{2}}-c\frac{\text{d}}{\text{d}x}-\mathbf{L}\right),+\infty\right)$. 
\end{rem*}
\begin{proof}
The fact that $\sup\limits _{\mu\in\mathbb{R}}\left(\kappa_{\mu}+\mu c\right)$
is finite and actually a maximum attained in $[0,+\infty)$ is a direct
consequence of:
\begin{itemize}
\item the evenness of $\mu\mapsto\kappa_{\mu}$ (whence, for all $\mu>0$,
$\kappa_{-\mu}+\left(-\mu\right)c<\kappa_{\mu}+\mu c$);
\item $\kappa_{0}<0$;
\item $\frac{\kappa_{\mu}}{\mu}+c\to-\infty$ as $\mu\to+\infty$ (see the
proof of \lemref{Definition_c_star}).
\end{itemize}
In addition, the sign of this maximum depending on the sign $c-c^{\star}$
is given by \lemref{Definition_c_star}.

Hence it only remains to prove
\[
\lambda_{1}\left(-\mathbf{D}\frac{\text{d}^{2}}{\text{d}x^{2}}-c\frac{\text{d}}{\text{d}x}-\mathbf{L}\right)=\max\limits _{\mu\geq0}\left(\kappa_{\mu}+\mu c\right).
\]
To do so, we use and adapt a well-known strategy of proof (see for
instance Nadin \cite{Nadin_2007}).

We recall from \thmref{Generalized_principal_eigenvalue} the definition
of the generalized principal eigenvalue:
\[
\lambda_{1}\left(-\mathbf{D}\frac{\text{d}^{2}}{\text{d}x^{2}}-c\frac{\text{d}}{\text{d}x}-\mathbf{L}\right)=\sup\left\{ \lambda\in\mathbb{R}\ |\ \exists\mathbf{n}\in\mathscr{C}^{2}\left(\mathbb{R},\mathsf{K}^{++}\right)\quad-\mathbf{D}\mathbf{n}''-c\mathbf{n}'-\mathbf{L}\mathbf{n}\geq\lambda\mathbf{n}\right\} .
\]
Also, there exists a generalized principal eigenfunction. We recall
from \lemref{Existence_exponential_eigenfunctions} that if there
exists a generalized principal eigenfunction, then there exists a
generalized principal eigenfunction of the form $\xi\mapsto\text{e}^{-\mu^{\star}\xi}\mathbf{m}$
with some constant $\mu^{\star}\geq0$ and $\mathbf{m}\in\mathsf{K}^{++}$.

Now, $\left(\mu^{\star},\mathbf{m}\right)\in[0,+\infty)\times\mathsf{K}^{++}$
satisfies
\[
-\left(\mu^{\star}\right)^{2}\mathbf{D}\mathbf{m}+c\mu^{\star}\mathbf{m}-\mathbf{L}\mathbf{m}=\lambda_{1}\left(-\mathbf{D}\frac{\text{d}^{2}}{\text{d}\xi^{2}}-c\frac{\text{d}}{\text{d}\xi}-\mathbf{L}\right)\mathbf{m},
\]
 that is
\[
-\left(\left(\mu^{\star}\right)^{2}\mathbf{D}+\mathbf{L}\right)\mathbf{m}=\left(\lambda_{1}\left(-\mathbf{D}\frac{\text{d}^{2}}{\text{d}\xi^{2}}-c\frac{\text{d}}{\text{d}\xi}-\mathbf{L}\right)-c\mu^{\star}\right)\mathbf{m},
\]
 or in other words
\[
\lambda_{1}\left(-\mathbf{D}\frac{\text{d}^{2}}{\text{d}\xi^{2}}-c\frac{\text{d}}{\text{d}\xi}-\mathbf{L}\right)=\kappa_{\mu^{\star}}+c\mu^{\star}\text{ and }\frac{\mathbf{m}}{\left|\mathbf{m}\right|}=\mathbf{n}_{\mu^{\star}}.
\]

Finally, the suitable test function to verify
\[
\lambda_{1}\left(-\mathbf{D}\frac{\text{d}^{2}}{\text{d}\xi^{2}}-c\frac{\text{d}}{\text{d}\xi}-\mathbf{L}\right)\geq\kappa_{\mu}+\mu c\text{ for all }\mu\geq0
\]
 is of course $\mathbf{v}_{\mu}:\xi\mapsto\text{e}^{-\mu\xi}\mathbf{n}_{\mu}$
itself, which satisfies precisely
\[
-\mathbf{D}\mathbf{v}_{\mu}''-c\mathbf{v}_{\mu}'-\mathbf{L}\mathbf{v}_{\mu}=\left(\kappa_{\mu}+\mu c\right)\mathbf{v}_{\mu}.
\]
\end{proof}
\begin{cor}
The quantity $c^{\star}$ is characterized by
\begin{align*}
c^{\star} & =\sup\left\{ c\geq0\ |\ \lambda_{1}\left(-\mathbf{D}\frac{\text{d}^{2}}{\text{d}\xi^{2}}-c\frac{\text{d}}{\text{d}\xi}-\mathbf{L}\right)<0\right\} \\
 & =\inf\left\{ c\geq0\ |\ \lambda_{1}\left(-\mathbf{D}\frac{\text{d}^{2}}{\text{d}\xi^{2}}-c\frac{\text{d}}{\text{d}\xi}-\mathbf{L}\right)>0\right\} .
\end{align*}
\end{cor}

\subsection{Qualitative properties of the traveling solutions}

Thanks to \lemref{Existence_exponential_eigenfunctions} and \corref{Existence_of_solution_of_the_linearized_eq_for_the_waves},
we are now in position to establish a few interesting properties that
have direct consequences but will also be used at the end of the construction
of the traveling waves.
\begin{lem}
\label{lem:Linearized_equation_at_the_edge_of_the_fronts} Let $c\in[0,+\infty)$
and $\mathbf{p}$ be a bounded nonnegative nonzero classical solution
of $\left(TW\left[c\right]\right)$. 

If $\left(\liminf\limits _{\xi\to+\infty}p_{i}\left(\xi\right)\right)_{i\in\left[N\right]}\in\partial\mathsf{K}$,
then $c\geq c^{\star}$.
\end{lem}

\begin{rem*}
The following proof is analogous to that of Berestycki\textendash Nadin\textendash Perthame\textendash Ryzhik
\cite[Lemma 3.8]{Berestycki_Nadin_Perthame_Ryzhik} for the non-local
KPP equation.
\end{rem*}
\begin{proof}
Let $\left(\zeta_{n}\right)_{n\in\mathbb{N}}\in\mathbb{R}^{\mathbb{N}}$
such that, as $n\to+\infty$, $\zeta_{n}\to+\infty$ and at least
one component of $\left(\mathbf{p}\left(\zeta_{n}\right)\right)_{n\in\mathbb{N}}$
converges to $0$. Define
\[
\mathbf{p}_{n}:\xi\mapsto\mathbf{p}\left(\xi+\zeta_{n}\right)
\]
 and observe that $\mathbf{p}_{n}$ satisfies $\left(TW\left[c\right]\right)$
as well. By virtue of Arapostathis\textendash Gosh\textendash Marcus\textquoteright s
Harnack inequality \cite{Araposthathis_} applied to the linear operator
\[
\mathbf{D}\frac{\text{d}^{2}}{\text{d}\xi^{2}}+c\frac{\text{d}}{\text{d}\xi}+\left(\mathbf{L}-\text{diag}\left(\mathbf{c}\left[\mathbf{p}_{n}\right]\right)\right),
\]
 classical elliptic estimates (Gilbarg\textendash Trudinger \cite{Gilbarg_Trudin}),
$\left(\mathbf{p}_{n}\right)_{n\in\mathbb{N}}$ converges up to a
diagonal extraction in $\mathscr{C}_{loc}^{2}$ to $\mathbf{0}$.
This proves that there is no limit point of $\mathbf{p}$ at $+\infty$
in $\partial\mathsf{K}\backslash\left\{ \mathbf{0}\right\} $.

Next, define
\[
\tilde{\mathbf{p}}_{n}:\xi\mapsto\frac{\mathbf{p}\left(\xi+\zeta_{n}\right)}{\left|\mathbf{p}\left(\zeta_{n}\right)\right|}
\]
 and notice, again by Arapostathis\textendash Gosh\textendash Marcus\textquoteright s
Harnack inequality, that $\left(\tilde{\mathbf{p}}_{n}\right)_{n\in\mathbb{N}}$
is locally uniformly bounded. Since, for all $n\in\mathbb{N}$, $\tilde{\mathbf{p}}_{n}$
solves
\[
-\mathbf{D}\tilde{\mathbf{p}}_{n}''-c\tilde{\mathbf{p}}_{n}'=\mathbf{L}\tilde{\mathbf{p}}_{n}-\mathbf{c}\left[\mathbf{p}_{n}\right]\circ\tilde{\mathbf{p}}_{n},
\]
 with, thanks to the fact that $\mathbf{c}$ vanishes at $\mathbf{0}$
$\left(H_{3}\right)$, $\mathbf{c}\left[\mathbf{p}_{n}\right]\to\mathbf{0}$
locally uniformly, up to extraction $\left(\tilde{\mathbf{p}}_{n}\right)_{n\in\mathbb{N}}$
converges in $\mathscr{C}_{loc}^{2}$ to a nonnegative solution $\tilde{\mathbf{p}}$
of $\left(TW_{0}\left[c\right]\right)$. Since $\tilde{\mathbf{p}}_{n}\left(0\right)\in\mathsf{S}^{++}\left(\mathbf{0},1\right)$
for all $n\in\mathbb{N}$, $\tilde{\mathbf{p}}$ is nonnegative nonzero,
whence positive (\corref{Positivity_for_the_elliptic_problem_with_drift}). 

Now, from \corref{Existence_of_solution_of_the_linearized_eq_for_the_waves},
we deduce indeed that $c\geq c^{\star}$.
\end{proof}
This result implies the nonexistence half of \thmref{Traveling_waves}
\ref{enu:Existence_minimal_wave_speed}.
\begin{cor}
\label{cor:Nonexistence_for_c_smaller_than_c_star_and_positivity_at_the_back}
For all $c\in[0,c^{\star})$, there is no traveling wave solution
of $\left(E_{KPP}\right)$ with speed $c$. 
\end{cor}

Now, with \propref{Absorbing_set}, $c\geq c^{\star}>0$ and the fact
that $\left(t,x\right)\mapsto\mathbf{p}\left(x-ct\right)$ solves
$\left(E_{KPP}\right)$, we can straightforwardly derive the uniform
upper bound \thmref{Traveling_waves} \enuref{Upper_bound_for_the_profiles},
which is interestingly independent of $c$.
\begin{cor}
\label{cor:Uniform_upper_bound_for_the_traveling_waves} All profiles
$\mathbf{p}$ satisfy
\[
\mathbf{p}\leq\mathbf{g}\left(0\right)\text{ in }\mathbb{R}.
\]
\end{cor}

Subsequently, using \propref{Persistence} and again $c\geq c^{\star}>0$
and the fact that $\left(t,x\right)\mapsto\mathbf{p}\left(x-ct\right)$
solves $\left(E_{KPP}\right)$, we get \thmref{Traveling_waves} \enuref{Persistence_at_the_back_of_the_front},
independent of $c$ as well.
\begin{cor}
\label{cor:Uniform_lower_bound_at_the_back_of_the_front} All profiles
$\mathbf{p}$ satisfy
\[
\left(\liminf_{\xi\to-\infty}p_{i}\left(\xi\right)\right)_{i\in\left[N\right]}\geq\nu\mathbf{1}_{N,1}.
\]
\end{cor}

Now, we establish \thmref{Traveling_waves} \ref{enu:Monotonicity_at_the_edge_of_the_front}.
Its proof is actually mostly a repetition of that of \lemref{Existence_exponential_eigenfunctions}.
\begin{prop}
\label{prop:Monotonicity_at_the_edge_of_the_fronts} Let $\left(\mathbf{p},c\right)$
be a traveling wave solution of $\left(E_{KPP}\right)$.

Then there exists $\overline{\xi}\in\mathbb{R}$ such that $\mathbf{p}$
is component-wise decreasing in $[\overline{\xi},+\infty)$.
\end{prop}

\begin{proof}
Let $\mathbf{v}=\left(\frac{p_{i}'}{p_{i}}\right)_{i\in\left[N\right]}$.
By virtue of Arapostathis\textendash Gosh\textendash Marcus\textquoteright s
Harnack inequality \cite{Araposthathis_}, classical elliptic estimates
(Gilbarg\textendash Trudinger \cite{Gilbarg_Trudin}) and invariance
by translation of $\left(TW\left[c\right]\right)$, $\mathbf{v}$
is globally bounded. Define for all $i\in\left[N\right]$
\[
\Lambda_{i}=\limsup_{\xi\to+\infty}v_{i}\left(\xi\right).
\]

Let $\overline{\Lambda}=\max\limits _{i\in\left[N\right]}\Lambda_{i}$,
so that
\[
\left(\limsup_{\xi\to+\infty}v_{i}\left(\xi\right)\right)_{i\in\left[N\right]}\leq\overline{\Lambda}\mathbf{1}_{N,1}.
\]

Let $\left(\xi_{n}\right)_{n\in\mathbb{N}}\in\mathbb{R}^{\mathbb{N}}$
such that $\xi_{n}\to+\infty$ and such that there exists $\overline{i}\in\left[N\right]$
such that
\[
v_{\overline{i}}\left(\xi_{n}\right)\to\overline{\Lambda}\text{ as }n\to+\infty.
\]

Let
\[
\hat{\mathbf{p}}_{n}:\xi\mapsto\frac{\mathbf{p}\left(\xi+\xi_{n}\right)}{p_{\overline{i}}\left(\xi_{n}\right)}\text{ for all }n\in\mathbb{N}.
\]
 and notice, again by Arapostathis\textendash Gosh\textendash Marcus\textquoteright s
Harnack inequality, that $\left(\hat{\mathbf{p}}_{n}\right)_{n\in\mathbb{N}}$
is locally uniformly bounded. Since, for all $n\in\mathbb{N}$, $\hat{\mathbf{p}}_{n}$
solves
\[
-\mathbf{D}\hat{\mathbf{p}}_{n}''-c\hat{\mathbf{p}}_{n}'=\mathbf{L}\hat{\mathbf{p}}_{n}-\mathbf{c}\left[p_{\overline{i}}\left(\xi_{n}\right)\hat{\mathbf{p}}_{n}\right]\circ\hat{\mathbf{p}}_{n},
\]
 and, thanks to the fact that $\mathbf{c}$ vanishes at $\mathbf{0}$
$\left(H_{3}\right)$ and the asymptotic behavior of $\mathbf{p}$
at $+\infty$, $\mathbf{c}\left[p_{\overline{i}}\left(\xi_{n}\right)\hat{\mathbf{p}}_{n}\right]$
converges locally uniformly to $\mathbf{0}$ as $n\to+\infty$, up
to a diagonal extraction process, $\left(\hat{\mathbf{p}}_{n}\right)_{n\in\mathbb{N}}$
converges in $\mathscr{C}_{loc}^{2}$ to a nonnegative solution $\hat{\mathbf{p}}_{\infty}$
of $\left(TW_{0}\left[c\right]\right)$. 

Now we repeat the second part of the proof of \lemref{Existence_exponential_eigenfunctions}
and we deduce in the end from \lemref{Definition_c_star} that $\hat{\mathbf{p}}_{\infty}$
has exactly the form
\[
\xi\mapsto A\text{e}^{-\mu_{c}\xi}\mathbf{n}_{\mu_{c}},
\]
with $\mu_{c}\in\left\{ \mu_{1,c},\mu_{2,c}\right\} $ if $c>c^{\star}$,
$\mu_{c}=\mu_{c^{\star}}$ if $c=c^{\star}$, $A>0$ and, most importantly,
with $\mu_{c}=-\overline{\Lambda}$. 

Thus $\overline{\Lambda}<0$. This implies that there exists $\overline{\xi}\in\mathbb{R}$
such that, for all $\xi\geq\overline{\xi}$, 
\[
\mathbf{v}\left(\xi\right)\leq-\frac{\left|\overline{\Lambda}\right|}{2}\mathbf{1}_{N,1},
\]
whence, by positivity of $\mathbf{p}$,
\[
\mathbf{p}'\left(\xi\right)\leq-\frac{\left|\overline{\Lambda}\right|}{2}\mathbf{p}\left(\xi\right).
\]
The right-hand side being negative, $\mathbf{p}$ is component-wise
decreasing indeed.
\end{proof}
\begin{lem}
\label{lem:Null_liminf_enforces_null_lim} Let $c\in[0,+\infty)$
and $\mathbf{p}$ be a bounded nonnegative nonzero classical solution
of $\left(TW\left[c\right]\right)$. 

If $\left(\liminf\limits _{\xi\to+\infty}p_{i}\left(\xi\right)\right)_{i\in\left[N\right]}\in\partial\mathsf{K}$,
then $\lim\limits _{\xi\to+\infty}\mathbf{p}\left(\xi\right)=\mathbf{0}$.
\end{lem}

\begin{proof}
Let $\left(\zeta_{n}\right)_{n\in\mathbb{N}}\in\mathbb{R}^{\mathbb{N}}$
such that, as $n\to+\infty$, $\zeta_{n}\to+\infty$ and at least
one component of $\left(\mathbf{p}\left(\zeta_{n}\right)\right)_{n\in\mathbb{N}}$
converges to $0$. The proof of \lemref{Linearized_equation_at_the_edge_of_the_fronts}
shows that $\left(\mathbf{p}_{n}\right)_{n\in\mathbb{N}}$, defined
by $\mathbf{p}_{n}:\xi\mapsto\mathbf{p}\left(\xi+\zeta_{n}\right)$,
converges up to extraction in $\mathscr{C}_{loc}^{2}$ to $\mathbf{0}$.

Now, defining
\[
\mathbf{v}_{n}:\xi\mapsto\left(\frac{p_{n,i}'\left(\xi\right)}{p_{n,i}\left(\xi\right)}\right)_{i\in\left[N\right]},
\]
\[
\Lambda_{i}=\limsup_{n\to+\infty}\max_{\left[-1,1\right]}v_{n,i},
\]
\[
\overline{\Lambda}=\max_{i\in\left[N\right]}\Lambda_{i},
\]
\[
\overline{i}\in\left[N\right]\text{ such that }\Lambda_{\overline{i}}=\overline{\Lambda},
\]
and $\left(n_{m}\right)_{m\in\mathbb{N}}\in\mathbb{N}^{\mathbb{N}}$
an increasing sequence such that $v_{n_{m},\overline{i}}\left(0\right)\to\overline{\Lambda}$
as $m\to+\infty$, we can repeat once more the argument of the proof
of \lemref{Existence_exponential_eigenfunctions} and obtain
\[
\overline{\Lambda}\hat{\mathbf{p}}_{\infty}=\hat{\mathbf{p}}_{\infty}'\text{ in }\left(-1,1\right)
\]
 (notice that, contrarily to the proof of \lemref{Existence_exponential_eigenfunctions}
where this equality was proved in $\mathbb{R}$, here it only holds
locally). This brings forth $\overline{\Lambda}=-\mu_{c}<0$, as in
the proof of \propref{Monotonicity_at_the_edge_of_the_fronts}, whence
$\mathbf{p}_{n}$ is component-wise decreasing in $\left[-1,1\right]$
provided $n$ is large enough.

Now, assuming by contradiction
\[
\left(\limsup_{\xi\to+\infty}p_{i}\left(\xi\right)\right)_{i\in\left[N\right]}\in\mathsf{K}^{+},
\]
that is
\[
\left(\limsup_{\xi\to+\infty}p_{i}\left(\xi\right)\right)_{i\in\left[N\right]}\in\mathsf{K}^{++},
\]
we deduce from the $\mathscr{C}^{1}$ regularity of $\mathbf{p}$
that, for any $i\in\left[N\right]$, there exists a sequence $\left(\zeta_{n}'\right)_{n\in\mathbb{N}}\in\mathbb{R}^{\mathbb{N}}$
such that:
\begin{itemize}
\item $\zeta_{n}'\to+\infty$ as $n\to+\infty$, 
\item $p_{i}\left(\zeta_{n}'\right)$ is a local minimum of $p_{i}$,
\item $p_{i}\left(\zeta_{n}'\right)\to0$ as $n\to+\infty$. 
\end{itemize}
Since this directly contradicts the preceding argument, we get indeed
\[
\left(\limsup_{\xi\to+\infty}p_{i}\left(\xi\right)\right)_{i\in\left[N\right]}=\mathbf{0}=\left(\liminf\limits _{\xi\to+\infty}p_{i}\left(\xi\right)\right)_{i\in\left[N\right]}.
\]
\end{proof}
\begin{lem}
\label{lem:Uniform_lower_bound_for_positive_infima} Let $c\in[0,+\infty)$.
There exists $\eta_{c}>0$ such that, for all bounded nonnegative
classical solutions $\mathbf{p}$ of $\left(TW\left[c\right]\right)$,
exactly one of the following properties holds:
\begin{enumerate}
\item $\lim\limits _{\xi\to+\infty}\mathbf{p}\left(\xi\right)=\mathbf{0}$;
\item $\left(\inf\limits _{\left(0,+\infty\right)}p_{i}\right)_{i\in\left[N\right]}\geq\eta_{c}\mathbf{1}_{N,1}$.
\end{enumerate}
\end{lem}

\begin{rem*}
The following proof is again analogous to that of Berestycki\textendash Nadin\textendash Perthame\textendash Ryzhik
\cite[Lemma 3.4]{Berestycki_Nadin_Perthame_Ryzhik} for the non-local
KPP equation.
\end{rem*}
\begin{proof}
Recall from \corref{Positivity_for_the_elliptic_problem_with_drift}
and \lemref{Null_liminf_enforces_null_lim} that $\left(\inf\limits _{\left(0,+\infty\right)}p_{i}\right)_{i\in\left[N\right]}\in\partial\mathsf{K}$
if and only if $\lim\limits _{\xi\to+\infty}\mathbf{p}\left(\xi\right)=\mathbf{0}$.
Hence, defining $\Sigma$ as the set of all bounded nonnegative classical
solutions $\mathbf{p}$ of $\left(TW\left[c\right]\right)$ such that
\[
\min_{i\in\left[N\right]}\inf_{\left(0,+\infty\right)}p_{i}>0,
\]
this set containing at least one positive constant vector by virtue
of \thmref{Existence_of_steady_states}, it only remains to show the
positivity of
\[
\eta_{c}=\inf\left\{ \min_{i\in\left[N\right]}\inf_{\left(0,+\infty\right)}p_{i}\ |\ \mathbf{p}\in\Sigma\right\} .
\]

We assume by contradiction the existence of a sequence $\left(\mathbf{p}_{n}\right)_{n\in\mathbb{N}}\in\Sigma^{\mathbb{N}}$
such that
\[
\lim_{n\to+\infty}\min_{i\in\left[N\right]}\inf_{\left(0,+\infty\right)}p_{n,i}=0.
\]

For all $n\in\mathbb{N}$, define
\[
\beta_{n}=\min_{i\in\left[N\right]}\inf_{\left(0,+\infty\right)}p_{n,i}>0,
\]
fix $\xi_{n}\in\left(0,+\infty\right)$ such that
\[
\min_{i\in\left[N\right]}p_{n,i}\left(\xi_{n}\right)\in\left[\beta_{n},\beta_{n}+\frac{1}{n}\right],
\]
 and define finally
\[
\mathbf{v}_{n}:\xi\mapsto\frac{1}{\beta_{n}}\mathbf{p}_{n}\left(\xi+\xi_{n}\right).
\]
By virtue of Arapostathis\textendash Gosh\textendash Marcus\textquoteright s
Harnack inequality \cite{Araposthathis_}, classical elliptic estimates
(Gilbarg\textendash Trudinger \cite{Gilbarg_Trudin}) and invariance
by translation of $\left(TW\left[c\right]\right)$, $\left(\mathbf{v}_{n}\right)_{n\in\mathbb{N}}$
is locally uniformly bounded and, up to a diagonal extraction process,
converges in $\mathscr{C}_{loc}^{2}$ to some bounded limit $\mathbf{v}_{\infty}$.
As in the proof of \lemref{Existence_exponential_eigenfunctions},
it is easily verified that $\mathbf{v}_{\infty}$ is a bounded positive
classical solution of $\left(TW_{0}\left[c\right]\right)$. Furthermore,
by definition of $\left(\mathbf{v}_{n}\right)_{n\in\mathbb{N}}$,
\[
\mathbf{v}_{\infty}\geq\mathbf{1}_{N,1}\text{ in }\left(0,+\infty\right).
\]

Repeating once more the argument of the proof of \lemref{Existence_exponential_eigenfunctions},
we deduce that $\mathbf{v}_{\infty}$ is component-wise decreasing
in a neighborhood of $+\infty$. Thus its limit at $+\infty$, say
$\mathbf{m}\geq\mathbf{1}_{N,1}$, is well-defined. By classical elliptic
estimates, $\mathbf{m}$ satisfies $\mathbf{L}\mathbf{m}=\mathbf{0}$,
which obviously contradicts $\lambda_{PF}\left(\mathbf{L}\right)>0$.
\end{proof}

\subsection{Existence of traveling waves}

This whole subsection is devoted to the adaptation of a proof of existence
due to Berestycki, Nadin, Perthame and Ryzhik \cite{Berestycki_Nadin_Perthame_Ryzhik}
and originally applied to the non-local KPP equation.
\begin{rem*}
There is a couple of slight mistakes in the aforementioned proof. 
\begin{enumerate}
\item Using the notations of \cite{Berestycki_Nadin_Perthame_Ryzhik}, the
sub-solution is defined as $\overline{r}_{c}=\max\left(0,r_{c}\right)$,
with $r_{c}$ chosen so that
\[
-cr_{c}'\leq r_{c}''+\mu r_{c}-\mu\overline{q}_{c}\left(\phi\star\overline{q}_{c}\right)
\]
 and it is claimed that $\overline{r}_{c}$ satisfies as well this
inequality, in the distributional sense. This is false: in an interval
where $\overline{r}_{c}=0$, we have
\[
-c\overline{r}_{c}'-\overline{r}_{c}''-\mu\overline{r}_{c}=0>-\mu\overline{q}_{c}\left(\phi\star\overline{q}_{c}\right).
\]
 As we will show, the correct sub-solution is $\overline{r}_{c}=\max\left(0,r_{c}\right)$
with $r_{c}$ chosen so that
\[
-cr_{c}'\leq r_{c}''+\mu r_{c}-\mu r_{c}\left(\phi\star\overline{q}_{c}\right).
\]
 Fortunately, the function $r_{c}$ constructed by the authors satisfies
this inequality as well.
\item Later on, $\Phi_{a}$ is defined as the mapping which maps $u_{0}$
to the solution of
\[
-cu'=u''+\mu u_{0}\left(1-\phi\star u_{0}\right).
\]
This mapping does not leave invariant the set of functions $R_{a}$
defined with the correct sub-solution. It is necessary to change $\Phi_{a}$
and to define it as the mapping which maps $u_{0}$ to the solution
of
\[
-cu'=u''+\mu u\left(1-\phi\star u_{0}\right).
\]
 Consequently, in order to establish that the set of functions $R_{a}$
is invariant by $\Phi_{a}$, the elliptic maximum principle is applied
not to $u\mapsto-cu'-u''$ but to
\[
u\mapsto-u''-cu'-\mu u
\]
 on one hand and to
\[
u\mapsto-u''-cu'-\mu\left(1-\phi\star\overline{q}_{c}\right)u
\]
 on the other hand.
\end{enumerate}
\end{rem*}
During the first three subsubsections, we fix $c>c^{\star}$.

\subsubsection{Super-solution}

We will use $\overline{\mathbf{p}}:\xi\mapsto\text{e}^{-\mu_{1,c}\xi}\mathbf{n}_{\mu_{1,c}}$
as a super-solution (recall from \lemref{Definition_c_star} that
it is a solution of $\left(TW_{0}\left[c\right]\right)$).

\subsubsection{Sub-solution}
\begin{prop}
\label{prop:Sub-solution} There exist $\overline{\varepsilon}>0$
such that, for any $\varepsilon\in\left(0,\overline{\varepsilon}\right)$,
there exists $A_{\varepsilon}\in\left(0,+\infty\right)$ such that
the function
\[
\underline{\mathbf{p}}:\xi\mapsto\left(\max\left(\text{e}^{-\mu_{1,c}\xi}n_{\mu_{1,c},i}-A_{\varepsilon}\text{e}^{-\left(\mu_{1,c}+\varepsilon\right)\xi}n_{\mu_{1,c}+\varepsilon,i},0\right)\right)_{i\in\left[N\right]},
\]
satisfies
\[
-\mathbf{D}\underline{\mathbf{p}}''-c\underline{\mathbf{p}}'-\mathbf{L}\underline{\mathbf{p}}\leq-\mathbf{c}\left[\overline{\mathbf{p}}\right]\circ\underline{\mathbf{p}}\text{ in }\mathscr{H}^{-1}\left(\mathbb{R},\mathbb{R}^{N}\right).
\]
\end{prop}

\begin{rem*}
Notice that, in the right-hand side of the inequality above, we find
$\mathbf{c}\left[\overline{\mathbf{p}}\right]$ and not $\mathbf{c}\left[\underline{\mathbf{p}}\right]$.
This is of course related to the lack of comparison principle for
$\left(E_{KPP}\right)$. 

During the forthcoming quite technical proof, in order to ease the
reading, we denote $\left\langle \bullet,\bullet\right\rangle _{1}$
and $\left\langle \bullet,\bullet\right\rangle _{N}$ the duality
pairings of $\mathscr{H}^{1}\left(\mathbb{R},\mathbb{R}\right)$ and
$\mathscr{H}^{1}\left(\mathbb{R},\mathbb{R}^{N}\right)$ respectively,
the latter being of course defined by:
\[
\left\langle \mathbf{f},\mathbf{g}\right\rangle _{\mathscr{H}^{-1}\left(\mathbb{R},\mathbb{R}^{N}\right)\times\mathscr{H}^{1}\left(\mathbb{R},\mathbb{R}^{N}\right)}=\sum_{i=1}^{N}\left\langle f_{i},g_{i}\right\rangle _{\mathscr{H}^{-1}\left(\mathbb{R}\right)\times\mathscr{H}^{1}\left(\mathbb{R}\right)}.
\]
 The speed $c$ being fixed, we also omit the subscript $c$ in the
notations $\mu_{1,c}$ and $\mu_{2,c}$.
\end{rem*}
\begin{proof}
For the moment, let $A,\varepsilon>0$ (they will be made precise
during the course of the proof) and define
\[
\mathbf{v}:\xi\mapsto\text{e}^{-\mu_{1}\xi}\mathbf{n}_{\mu_{1}}-A\text{e}^{-\left(\mu_{1}+\varepsilon\right)\xi}\mathbf{n}_{\mu_{1}+\varepsilon},
\]
\[
\underline{\mathbf{p}}:\xi\mapsto\left(\max\left(\text{e}^{-\mu_{1}\xi}n_{\mu_{1},i}-A_{\varepsilon}\text{e}^{-\left(\mu_{1}+\varepsilon\right)\xi}n_{\mu_{1}+\varepsilon,i},0\right)\right)_{i\in\left[N\right]},
\]
\[
\Xi_{+}=\underline{\mathbf{p}}^{-1}\left(\mathsf{K}^{++}\right),
\]
\[
\Xi_{0}=\underline{\mathbf{p}}^{-1}\left(\mathbf{0}\right),
\]
\[
\Xi_{\#}=\mathbb{R}\backslash\left(\Xi_{+}\cup\Xi_{0}\right).
\]

Notice that $\Xi_{\#}$ is a connected compact set. 

Fix a positive test function $\varphi\in\mathscr{H}^{1}\left(\mathbb{R},\mathsf{K}^{++}\right)$.
We have to verify that
\[
\left\langle -\mathbf{D}\underline{\mathbf{p}}''-c\underline{\mathbf{p}}'-\mathbf{L}\underline{\mathbf{p}},\varphi\right\rangle _{N}\leq\left\langle -\mathbf{c}\left[\overline{\mathbf{p}}\right]\circ\underline{\mathbf{p}},\varphi\right\rangle _{N}.
\]

To this end, we distinguish three cases: $\text{supp}\varphi\subset\Xi_{+}$,
$\text{supp}\varphi\subset\Xi_{0}$ and $\text{supp}\varphi\cap\Xi_{\#}\neq\emptyset$.
The case $\text{supp}\varphi\subset\Xi_{0}$ is trivial, with the
inequality above satisfied in the classical sense. 

Regarding the case $\text{supp}\varphi\subset\Xi_{+}$, we only have
to verify the inequality in the classical sense in $\Xi_{+}$ for
the regular function $\mathbf{v}$. 

Fix temporarily $\xi\in\Xi_{+}$. We have
\[
-\mathbf{D}\mathbf{v}''\left(\xi\right)-c\mathbf{v}'\left(\xi\right)-\mathbf{L}\mathbf{v}\left(\xi\right)=A\text{e}^{-\left(\mu_{1}+\varepsilon\right)\xi}\left(\left(\mu_{1}+\varepsilon\right)^{2}\mathbf{D}-c\left(\mu_{1}+\varepsilon\right)\mathbf{I}+\mathbf{L}\right)\mathbf{n}_{\mu_{1}+\varepsilon},
\]
\[
\left(-\mathbf{c}\left[\overline{\mathbf{p}}\right]\circ\mathbf{v}\right)\left(\xi\right)=-\text{e}^{-\mu_{1}\xi}\mathbf{c}\left(\text{e}^{-\mu_{1}\xi}\mathbf{n}_{\mu_{1}}\right)\circ\left(\mathbf{n}_{\mu_{1}}-A\text{e}^{-\varepsilon\xi}\mathbf{n}_{\mu_{1}+\varepsilon}\right).
\]

From
\[
\left(\left(\mu_{1}+\varepsilon\right)^{2}\mathbf{D}+\mathbf{L}\right)\mathbf{n}_{\mu_{1}+\varepsilon}=-\kappa_{\mu_{1}+\varepsilon}\mathbf{n}_{\mu_{1}+\varepsilon},
\]
\[
-c\left(\mu_{1}+\varepsilon\right)\mathbf{n}_{\mu_{1}+\varepsilon}=\frac{\kappa_{\mu_{1}}}{\mu_{1}}\left(\mu_{1}+\varepsilon\right)\mathbf{n}_{\mu_{1}+\varepsilon},
\]
and the following direct consequence of the nonnegativity of $\mathbf{c}$
on $\mathsf{K}$ $\left(H_{2}\right)$,
\[
-\mathbf{c}\left(\text{e}^{-\mu_{1}\xi}\mathbf{n}_{\mu_{1}}\right)\circ\left(\mathbf{n}_{\mu_{1}}-A\text{e}^{-\varepsilon\xi}\mathbf{n}_{\mu_{1}+\varepsilon}\right)\geq-\mathbf{c}\left(\text{e}^{-\mu_{1}\xi}\mathbf{n}_{\mu_{1}}\right)\circ\mathbf{n}_{\mu_{1}},
\]
it follows that it suffices to find $A$ and $\varepsilon$ such that
\[
A\text{e}^{-\varepsilon\xi}\left(\mu_{1}+\varepsilon\right)\left(-\frac{\kappa_{\mu_{1}+\varepsilon}}{\mu_{1}+\varepsilon}+\frac{\kappa_{\mu_{1}}}{\mu_{1}}\right)\mathbf{n}_{\mu_{1}+\varepsilon}\leq-\mathbf{c}\left(\text{e}^{-\mu_{1}\xi}\mathbf{n}_{\mu_{1}}\right)\circ\mathbf{n}_{\mu_{1}}.
\]

The right-hand side above being nonnegative ($\mu\mapsto\frac{\kappa_{\mu}}{\mu}$
is positive and convex in $\left(0,+\infty\right)$, as detailed in
the proof of \lemref{Definition_c_star}), it follows clearly that
such an inequality is never satisfied if $\mu_{1}+\varepsilon>\mu_{2}$,
whence a first necessary condition on $\varepsilon$ is $\varepsilon\leq\mu_{2}-\mu_{1}$
(notice that if $\varepsilon=\mu_{2}-\mu_{1}$, then the inequality
above holds if and only if $\mathbf{c}\left(\text{e}^{-\mu_{1}\xi}\mathbf{n}_{\mu_{1}}\right)=\mathbf{0}$,
which is in general not true). Thus from now on we assume $\varepsilon<\mu_{2}-\mu_{1}$.
This ensures that $\frac{\kappa_{\mu_{1}+\varepsilon}}{\mu_{1}+\varepsilon}-\frac{\kappa_{\mu_{1}}}{\mu_{1}}>0$,
whence we now search for $A$ and $\varepsilon$ such that
\[
A\mathbf{n}_{\mu_{1}+\varepsilon}>\frac{\text{e}^{\varepsilon\xi}}{\left(\mu_{1}+\varepsilon\right)\left(\frac{\kappa_{\mu_{1}+\varepsilon}}{\mu_{1}+\varepsilon}-\frac{\kappa_{\mu_{1}}}{\mu_{1}}\right)}\mathbf{c}\left(\text{e}^{-\mu_{1}\xi}\mathbf{n}_{\mu_{1}}\right)\circ\mathbf{n}_{\mu_{1}}.
\]

Define $\overline{\xi}=\min\Xi_{\#}$, so that any $\xi\in\Xi_{+}$
satisfies necessarily $\xi>\overline{\xi}$. Remark that there exists
$\overline{i}\in\left[N\right]$ such that
\[
\overline{\xi}=\frac{1}{\varepsilon}\left(\ln A+\ln\left(\frac{n_{\mu_{1}+\varepsilon,\overline{i}}}{n_{\mu_{1},\overline{i}}}\right)\right).
\]

Now, defining $\alpha:\xi\mapsto\text{e}^{-\mu_{1}\xi}$, if
\[
A\geq\max_{i\in\left[N\right]}\left(\frac{n_{\mu_{1}+\varepsilon,i}}{n_{\mu_{1},i}}\right),
\]
 then $\overline{\xi}\geq0$ and $\alpha\left(\xi\right)\leq1$ in
$\left(\overline{\xi},+\infty\right)$. Moreover, we have
\[
\text{e}^{\varepsilon\xi}=\left(\alpha\left(\xi\right)\right)^{-\frac{\varepsilon}{\mu_{1}}},
\]
whence, for all $i\in\left[N\right]$,
\[
\text{e}^{\varepsilon\xi}c_{i}\left(\text{e}^{-\mu_{1}\xi}\mathbf{n}_{\mu_{1}}\right)=\frac{c_{i}\left(\alpha\left(\xi\right)\mathbf{n}_{\mu_{1}}\right)}{\left(\alpha\left(\xi\right)\right)^{\frac{\varepsilon}{\mu_{1}}}},
\]
and from the $\mathscr{C}^{1}$ regularity of $\mathbf{c}$ as well
as the fact that it vanishes at $\mathbf{0}$ $\left(H_{3}\right)$,
the above function of $\xi$ is globally bounded in $\left(\overline{\xi},+\infty\right)$,
provided $\frac{\varepsilon}{\mu_{1}}\leq1$, by the positive constant
\begin{align*}
M_{i} & =\sup_{\xi\in\left(\overline{\xi},+\infty\right)}\frac{c_{i}\left(\alpha\left(\xi\right)\mathbf{n}_{\mu_{1}}\right)}{\alpha\left(\xi\right)}\\
 & =\sup_{\alpha\in\left(0,1\right)}\frac{c_{i}\left(\alpha\mathbf{n}_{\mu_{1}}\right)}{\alpha}.
\end{align*}

Subsequently, if $A$ and $\varepsilon$ satisfy also
\[
\varepsilon\leq\mu_{1},
\]
\[
A\geq\max_{i\in\left[N\right]}\left(\frac{M_{i}n_{\mu_{1},i}}{\left(\mu_{1}+\varepsilon\right)\left(\frac{\kappa_{\mu_{1}+\varepsilon}}{\mu_{1}+\varepsilon}-\frac{\kappa_{\mu_{1}}}{\mu_{1}}\right)n_{\mu_{1}+\varepsilon,i}}\right),
\]
 then the inequality is established indeed in $\Xi_{+}$. Hence we
define
\[
\overline{\varepsilon}=\min\left(\mu_{2}-\mu_{1},\mu_{1}\right)
\]
 and, for any $\varepsilon\in\left(0,\overline{\varepsilon}\right)$,
\[
A_{\varepsilon}=\max_{i\in\left[N\right]}\max\left(\frac{n_{\mu_{1}+\varepsilon,i}}{n_{\mu_{1},i}},\frac{M_{i}n_{\mu_{1},i}}{\left(\mu_{1}+\varepsilon\right)\left(\frac{\kappa_{\mu_{1}+\varepsilon}}{\mu_{1}+\varepsilon}-\frac{\kappa_{\mu_{1}}}{\mu_{1}}\right)n_{\mu_{1}+\varepsilon,i}}\right)
\]
and we assume from now on $\varepsilon\in\left(0,\overline{\varepsilon}\right)$
and $A=A_{\varepsilon}$.

Let us point out here a fact which is crucial for the next step: choosing
$\overline{\xi}=\min\Xi_{\#}$ instead of $\overline{\xi}=\max\Xi_{\#}$
(which might seem more natural at first view) implies that the differential
inequality
\[
-\mathbf{D}\mathbf{v}''-c\mathbf{v}'-\mathbf{L}\mathbf{v}\leq-\mathbf{c}\left[\overline{\mathbf{p}}\right]\circ\mathbf{v}
\]
 holds classically in $\Xi_{\#}\cup\Xi_{+}$. 

To conclude, let us verify the case $\text{supp}\varphi\cap\Xi_{\#}\neq\emptyset$.
In order to ease the following computations, we actually assume $\varphi\in\mathscr{D}\left(\mathbb{R},\mathbb{R}^{N}\right)$
(the result with $\varphi\in\mathscr{H}^{1}\left(\mathbb{R},\mathbb{R}^{N}\right)$
can be recovered as usual by density). By definition,
\[
\left\langle -\mathbf{D}\underline{\mathbf{p}}''-c\underline{\mathbf{p}}'-\mathbf{L}\underline{\mathbf{p}}+\mathbf{c}\left[\overline{\mathbf{p}}\right]\circ\underline{\mathbf{p}},\varphi\right\rangle _{N}=\sum_{i=1}^{N}\left\langle -d_{i}\underline{p}_{i}''-c\underline{p}_{i}'-\sum_{j=1}^{N}l_{i,j}\underline{p}_{j}+c_{i}\left[\overline{\mathbf{p}}\right]\underline{p}_{i},\varphi_{i}\right\rangle _{1}.
\]

Fix $i\in\left[N\right]$ and define $\xi_{0,i}$ as the unique element
of $v_{i}^{-1}\left(\left\{ 0\right\} \right)$ and
\[
\Psi_{i}=\left\langle -d_{i}\underline{p}_{i}''-c\underline{p}_{i}'-\sum_{j=1}^{N}l_{i,j}\underline{p}_{j}+c_{i}\left[\overline{\mathbf{p}}\right]\underline{p}_{i},\varphi_{i}\right\rangle _{1}.
\]

Classical integrations by parts yield
\[
\int_{\mathbb{R}}\underline{p}_{i}''\varphi_{i}=\int_{\xi_{0,i}}^{+\infty}v_{i}''\varphi_{i}+v_{i}'\left(\xi_{0,i}\right)\varphi_{i}\left(\xi_{0,i}\right)\geq\int_{\xi_{0,i}}^{+\infty}v_{i}''\varphi_{i},
\]
\[
\int_{\mathbb{R}}\underline{p}_{i}'\varphi_{i}=\int_{\xi_{0,i}}^{+\infty}v_{i}'\varphi_{i},
\]
 whence
\[
\Psi_{i}\leq\int_{\xi_{0,i}}^{+\infty}\left(-d_{i}v_{i}''-cv_{i}'+c_{i}\left[\overline{\mathbf{p}}\right]v_{i}\right)\varphi_{i}-\sum_{j=1}^{N}l_{i,j}\int_{\xi_{0,j}}^{+\infty}v_{j}\varphi_{i}.
\]

As was pointed out previously, from the construction of $\varepsilon$
and $A$, we know that
\[
-\mathbf{D}\mathbf{v}''-c\mathbf{v}'+\mathbf{c}\left[\overline{\mathbf{p}}\right]\circ\mathbf{v}\leq\mathbf{L}\mathbf{v}\text{ in }\Xi_{\#},
\]
whence, with $J_{i}=\left\{ j\in\left[N\right]\ |\ \xi_{0,j}<\xi_{0,i}\right\} $,
\[
\Psi_{i}\leq-\sum_{j\in J_{i}}\int_{\xi_{0,j}}^{\xi_{0,i}}l_{i,j}v_{j}\varphi_{i}+\sum_{j\in\left[N\right]\backslash J_{i}}\int_{\xi_{0,i}}^{\xi_{0,j}}l_{i,j}v_{j}\varphi_{i}.
\]

Finally, recalling that $v_{j}\left(\xi\right)>0$ if $\xi>\xi_{0,j}$
and $v_{j}\left(\xi\right)<0$ if $\xi<\xi_{0,j}$, the inequality
above yields $\Psi_{i}\leq0$, which ends the proof.
\end{proof}

\subsubsection{The finite domain problem}

Let $R>0$ and define the following truncated problem:
\[
\left\{ \begin{matrix}-\mathbf{D}\mathbf{p}''-c\mathbf{p}'=\mathbf{L}\mathbf{p}-\mathbf{c}\left[\mathbf{p}\right]\circ\mathbf{p} & \text{ in }\left(-R,R\right),\\
\mathbf{p}\left(\pm R\right)=\underline{\mathbf{p}}\left(\pm R\right).
\end{matrix}\right.\quad\left(TW\left[R,c\right]\right)
\]
\begin{lem}
\label{lem:Finite_domain_problem} Assume
\[
D\mathbf{c}\left(\mathbf{v}\right)\geq\mathbf{0}\text{ for all }\mathbf{v}\in\mathsf{K}.
\]

Then there exists a nonnegative nonzero classical solution $\mathbf{p}_{R}$
of $\left(TW\left[R,c\right]\right)$. 
\end{lem}

\begin{rem*}
The new assumption made here ensures that the vector field $\mathbf{c}$
is non-decreasing in $\mathsf{K}$, in the following natural sense:
if $\mathbf{0}\leq\mathbf{v}\leq\mathbf{w}$, then $\mathbf{0}\leq\mathbf{c}\left(\mathbf{v}\right)\leq\mathbf{c}\left(\mathbf{w}\right)$.
\end{rem*}
\begin{proof}
Fix arbitrarily $\varepsilon\in\left(0,\overline{\varepsilon}\right)$,
define consequently $\underline{\mathbf{p}}$ and then define the
following convex set of functions:
\[
\mathscr{F}=\left\{ \mathbf{v}\in\mathscr{C}\left(\left[-R,R\right],\mathbb{R}^{N}\right)\ |\ \underline{\mathbf{p}}\leq\mathbf{v}\leq\overline{\mathbf{p}}\right\} .
\]

Recall that Figueiredo\textendash Mitidieri \cite{Figueiredo_Mit}
establishes that the elliptic weak maximum principle holds for a weakly
and fully coupled elliptic operator with null Dirichlet boundary conditions
if this operator admits a positive strict super-solution. Since, for
all $\mathbf{v}\in\mathscr{C}\left(\left[-R,R\right],\mathbb{R}^{N}\right)$
such that $\mathbf{0}\leq\mathbf{v}\leq\overline{\mathbf{p}}$, we
have by the nonnegativity of $\mathbf{c}$ on $\mathsf{K}$ $\left(H_{2}\right)$
\[
-\mathbf{D}\overline{\mathbf{p}}''-c\overline{\mathbf{p}}'-\mathbf{L}\overline{\mathbf{p}}+\mathbf{c}\left[\mathbf{v}\right]\circ\overline{\mathbf{p}}\geq-\mathbf{D}\overline{\mathbf{p}}''-c\overline{\mathbf{p}}'-\mathbf{L}\overline{\mathbf{p}}\geq\mathbf{0},
\]
\[
\overline{\mathbf{p}}\left(\pm R\right)\gg\mathbf{0},
\]
 it follows that every operator of the family
\[
\left(\mathbf{D}\frac{\text{d}^{2}}{\text{d}\xi^{2}}+c\frac{\text{d}}{\text{d}\xi}+\left(\mathbf{L}-\text{diag}\mathbf{c}\left[\mathbf{v}\right]\right)\right)_{\mathbf{0}\leq\mathbf{v}\leq\overline{\mathbf{p}}}
\]
supplemented with null Dirichlet boundary conditions at $\pm R$ satisfies
the weak maximum principle in $\left(-R,R\right)$.

Define the map $\mathbf{f}$ which associates with some $\mathbf{v}\in\mathscr{F}$
the unique classical solution $\mathbf{f}\left[\mathbf{v}\right]$
of:
\[
\left\{ \begin{matrix}-\mathbf{D}\mathbf{p}''-c\mathbf{p}'=\mathbf{L}\mathbf{p}-\mathbf{c}\left[\mathbf{v}\right]\circ\mathbf{p} & \text{ in }\left(-R,R\right)\\
\mathbf{p}\left(\pm R\right)=\underline{\mathbf{p}}\left(\pm R\right).
\end{matrix}\right.
\]

The map $\mathbf{f}$ is compact by classical elliptic estimates (Gilbarg\textendash Trudinger
\cite{Gilbarg_Trudin}). 

Let $\mathbf{v}\in\mathscr{F}$. By monotonicity of $\mathbf{c}$,
the function $\mathbf{w}=\mathbf{f}\left[\mathbf{v}\right]-\underline{\mathbf{p}}$
satisfies
\begin{align*}
-\mathbf{D}\mathbf{w}''-c\mathbf{w}'-\mathbf{L}\mathbf{w} & \geq-\mathbf{c}\left[\mathbf{v}\right]\circ\mathbf{f}\left[\mathbf{v}\right]+\mathbf{c}\left[\overline{\mathbf{p}}\right]\circ\underline{\mathbf{p}}\\
 & \geq-\mathbf{c}\left[\mathbf{v}\right]\circ\mathbf{f}\left[\mathbf{v}\right]+\mathbf{c}\left[\mathbf{v}\right]\circ\underline{\mathbf{p}}\\
 & \geq-\mathbf{c}\left[\mathbf{v}\right]\circ\mathbf{w}
\end{align*}
 with null Dirichlet boundary conditions at $\pm R$. Therefore, by
virtue of the weak maximum principle applied to $\mathbf{D}\frac{\text{d}^{2}}{\text{d}\xi^{2}}+c\frac{\text{d}}{\text{d}\xi}+\left(\mathbf{L}-\text{diag}\mathbf{c}\left[\mathbf{v}\right]\right)$,
$\mathbf{f}\left[\mathbf{v}\right]\geq\underline{\mathbf{p}}$ in
$\left(-R,R\right)$. Next, since it is now established that $\mathbf{f}\left[\mathbf{v}\right]\geq\mathbf{0}$,
we also have by $\left(H_{2}\right)$
\begin{align*}
-\mathbf{D}\overline{\mathbf{p}}''-c\overline{\mathbf{p}}'-\mathbf{L}\overline{\mathbf{p}} & =\mathbf{0}\\
 & \geq-\mathbf{c}\left[\mathbf{v}\right]\circ\mathbf{f}\left[\mathbf{v}\right]\\
 & =-\mathbf{D}\mathbf{f}\left[\mathbf{v}\right]''-c\mathbf{f}\left[\mathbf{v}\right]'-\mathbf{L}\mathbf{f}\left[\mathbf{v}\right],
\end{align*}
\[
\overline{\mathbf{p}}\left(\pm R\right)\geq\underline{\mathbf{p}}\left(\pm R\right)=\mathbf{f}\left[\mathbf{v}\right]\left(\pm R\right),
\]
whence $\overline{\mathbf{p}}\geq\mathbf{f}\left[\mathbf{v}\right]$
follows from the weak maximum principle applied this time to $\mathbf{D}\frac{\text{d}^{2}}{\text{d}\xi^{2}}+c\frac{\text{d}}{\text{d}\xi}+\mathbf{L}$.

Thus $\underline{\mathbf{p}}\leq\mathbf{f}\left[\mathbf{v}\right]\leq\overline{\mathbf{p}}$
and consequently $\mathbf{f}\left(\mathscr{F}\right)\subset\mathscr{F}$.

Finally, by virtue of the Schauder fixed point theorem, $\mathbf{f}$
admits a fixed point $\mathbf{p}_{R}\in\mathscr{F}$, which is indeed
a classical solution of $\left(TW\left[R,c\right]\right)$ by elliptic
regularity.
\end{proof}

\subsubsection{The infinite domain limit and the minimal wave speed}

The speed $c$ is not fixed anymore. 

The following uniform upper estimate is a direct consequence of \propref{Global_bounds_for_the_elliptic_problem_with_drift}.
\begin{cor}
\label{cor:Uniform_upper_bound_for_the_finite_domain_problem} There
exists $R^{\star}>0$ such that, for any $c>c^{\star}$, any $R\geq R^{\star}$
and any nonnegative classical solution $\mathbf{p}$ of $\left(TW\left[R,c\right]\right)$,
\[
\left(\max_{\left[-R,R\right]}p_{i}\right)_{i\in\left[N\right]}\leq\mathbf{g}\left(0\right).
\]
\end{cor}

We are now in position to prove the second half of \thmref{Traveling_waves}
\ref{enu:Existence_minimal_wave_speed}. 
\begin{prop}
\label{prop:Existence_for_c_larger_than_or_equal_to_c_star} Assume
\[
D\mathbf{c}\left(\mathbf{v}\right)\geq\mathbf{0}\text{ for all }\mathbf{v}\in\mathsf{K}.
\]

Then for all $c\geq c^{\star}$, there exists a traveling wave solution
of $\left(E_{KPP}\right)$ with speed $c$. 
\end{prop}

\begin{rem*}
Of course, it would be interesting to exhibit other additional assumptions
on $\mathbf{c}$ sufficient to ensure existence of traveling waves
for all $c\geq c^{\star}$. In view of known results about scalar
multistable reaction\textendash diffusion equations (we refer for
instance to Fife\textendash McLeod \cite{Fife_McLeod_19}), some additional
assumption should in any case be necessary. 
\end{rem*}
\begin{proof}
Hereafter, for all $c>c^{\star}$ and all $R>0$, the triplet $\left(\overline{\mathbf{p}},\underline{\mathbf{p}},\mathbf{p}_{R}\right)$
constructed in the preceding subsections is denoted $\left(\overline{\mathbf{p}}_{c},\underline{\mathbf{p}}_{c},\mathbf{p}_{R,c}\right)$.

For all $c>c^{\star}$, thanks to \corref{Uniform_upper_bound_for_the_finite_domain_problem},
the family $\left(\mathbf{p}_{R,c}\right)_{R>0}$ is uniformly globally
bounded. By classical elliptic estimates (Gilbarg\textendash Trudinger
\cite{Gilbarg_Trudin}) and a diagonal extraction process, we can
extract a sequence $\left(R_{n},\mathbf{p}_{R_{n},c}\right)_{n\in\mathbb{N}}$
such that, as $n\to+\infty$, $R_{n}\to+\infty$ and $\mathbf{p}_{R_{n},c}$
converges to some limit $\mathbf{p}_{c}$ in $\mathscr{C}_{loc}^{2}$.
As expected, $\mathbf{p}_{c}$ is a bounded nonnegative classical
solution of $\left(TW\left[c\right]\right)$. The fact that its limit
as $\xi\to+\infty$ is $\mathbf{0}$, as well as the fact that $\mathbf{p}_{c}$
is nonzero whence positive (\corref{Positivity_for_the_elliptic_problem_with_drift}),
are obvious thanks to the inequality $\underline{\mathbf{p}}_{c}\leq\mathbf{p}_{c}\leq\overline{\mathbf{p}}_{c}$.
At the other end of the real line, \corref{Nonexistence_for_c_smaller_than_c_star_and_positivity_at_the_back}
clearly enforces
\[
\left(\liminf\limits _{\xi\to-\infty}p_{c,i}\left(\xi\right)\right)_{i\in\left[N\right]}\in\mathsf{K}^{++}\subset\mathsf{K}^{+}.
\]

Thus $\left(\mathbf{p}_{c},c\right)$ is a traveling wave solution.

In order to construct a critical traveling wave $\left(\mathbf{p}_{c^{\star}},c^{\star}\right)$,
we consider a decreasing sequence $\left(c_{n}\right)_{n\in\mathbb{N}}\in\left(c^{\star},+\infty\right)^{\mathbb{N}}$
such that $c_{n}\to c^{\star}$ as $n\to+\infty$ and intend to apply
a compactness argument to a normalized version of the sequence $\left(\mathbf{p}_{c_{n}}\right)_{n\in\mathbb{N}}$. 

By \corref{Uniform_lower_bound_at_the_back_of_the_front}, 
\[
\liminf_{\xi\to-\infty}\min_{i\in\left[N\right]}p_{c_{n},i}\left(\xi\right)\geq\nu\text{ for all }n\in\mathbb{N}.
\]

Recall from \lemref{Uniform_lower_bound_for_positive_infima} the
definition of $\eta_{c}>0$. For all $n\in\mathbb{N}$ the following
quantity is well-defined and finite:
\[
\xi_{n}=\inf\left\{ \xi\in\mathbb{R}\ |\ \min_{i\in\left[N\right]}p_{c_{n},i}\left(\xi\right)<\min\left(\frac{\nu}{2},\frac{\eta_{c^{\star}}}{2}\right)\right\} .
\]

We define then the sequence of normalized profiles 
\[
\tilde{\mathbf{p}}_{c_{n}}:\xi\mapsto\mathbf{p}_{c_{n}}\left(\xi+\xi_{n}\right)\text{ for all }n\in\mathbb{N}.
\]
A translation of a profile of traveling wave being again a profile
of traveling wave, $\left(\tilde{\mathbf{p}}_{c_{n}},c_{n}\right)_{n\in\mathbb{N}}$
is again a sequence of traveling wave solutions. Notice the following
two immediate consequences of the normalization:
\[
\min_{i\in\left[N\right]}\tilde{p}_{c_{n},i}\left(0\right)=\min\left(\frac{\nu}{2},\frac{\eta_{c^{\star}}}{2}\right)\text{ for all }n\in\mathbb{N},
\]
\[
\inf_{\xi\in\left(-\infty,0\right)}\min_{i\in\left[N\right]}\tilde{p}_{c_{n},i}\left(\xi\right)\geq\min\left(\frac{\nu}{2},\frac{\eta_{c^{\star}}}{2}\right)\text{ for all }n\in\mathbb{N}.
\]

We are now in position to pass to the limit $n\to+\infty$. The sequence
$\left(\tilde{\mathbf{p}}_{c_{n}}\right)_{n\in\mathbb{N}}$ being
globally uniformly bounded, it admits, up to a diagonal extraction
process, a bounded nonnegative limit $\mathbf{p}_{c^{\star}}$ in
$\mathscr{C}_{loc}^{2}$. Since $c_{n}\to c^{\star}$, $\mathbf{p}_{c^{\star}}$
satisfies $\left(TW\left[c^{\star}\right]\right)$. The normalization
yields
\[
\min_{i\in\left[N\right]}p_{c^{\star},i}\left(0\right)=\min\left(\frac{\nu}{2},\frac{\eta_{c^{\star}}}{2}\right),
\]
\[
\inf_{\xi\in\left(-\infty,0\right)}\min_{i\in\left[N\right]}p_{c^{\star},i}\left(\xi\right)\geq\min\left(\frac{\nu}{2},\frac{\eta_{c^{\star}}}{2}\right).
\]

Consequently,
\[
\left(\liminf_{\xi\to-\infty}p_{c^{\star},i}\left(\xi\right)\right)_{i\in\left[N\right]}\in\mathsf{K}^{++}
\]
 and, according to \lemref{Uniform_lower_bound_for_positive_infima},
\[
\lim_{\xi\to+\infty}\mathbf{p}_{c^{\star}}\left(\xi\right)=\mathbf{0}.
\]

The pair $\left(\mathbf{p}_{c^{\star}},c^{\star}\right)$ is a traveling
wave solution indeed and this ends the proof.
\end{proof}

\section{Spreading speed}

In this section, we assume $\lambda_{PF}\left(\mathbf{L}\right)>0$
and prove \thmref{Spreading_speed}. In order to do so, we fix $\mathbf{u}_{0}\in\mathscr{C}_{b}\left(\mathbb{R},\mathbb{R}^{N}\right)$
of the form $\mathbf{u}_{0}=\mathbf{v}\mathbf{1}_{\left(-\infty,x_{0}\right)}$
with $x_{0}\in\mathbb{R}$ and $\mathbf{v}$ nonnegative nonzero and
we define $\mathbf{u}$ as the unique classical solution of $\left(E_{KPP}\right)$
set in $\left(0,+\infty\right)\times\mathbb{R}$ with initial data
$\mathbf{u}_{0}$.
\begin{rem*}
This type of spreading result, as well as its proof by means of super-
and sub-solutions, is quite classical (we refer to Aronson\textendash Weinberger
\cite{Aronson_Weinbe} and Berestycki\textendash Hamel\textendash Nadin
\cite{Berestycki_Hamel_Nadin} among others). Still, we provide it
to make clear that the lack of comparison principle for $\left(E_{KPP}\right)$
is not really an issue. 

Of course, for the scalar KPP equation, much more precise spreading
results exist (for instance the celebrated articles by Bramson \cite{Bramson_1978,Bramson_1983}
using probabilistic methods). Here, our aim is not to give a complete
description of the spreading properties of $\left(E_{KPP}\right)$
but rather to illustrate that it is, once more, very similar to the
scalar situation and that further generalizations should be possible.
\end{rem*}

\subsection{Upper estimate}
\begin{prop}
\label{prop:Spreading_speed_1} Let $c>c^{\star}$ and $y\in\mathbb{R}$.
We have
\[
\left(\lim_{t\to+\infty}\sup_{x\in\left(y,+\infty\right)}u_{i}\left(t,x+ct\right)\right)_{i\in\left[N\right]}=\mathbf{0}.
\]
\end{prop}

\begin{proof}
By definition of $\mathbf{u}_{0}$, there exists $\xi_{1}\in\mathbb{R}$
such that
\[
\overline{\mathbf{p}}:\xi\mapsto\text{e}^{-\mu_{c^{\star}}\left(\xi-\xi_{1}\right)}\mathbf{n}_{\mu_{c^{\star}}}
\]
(which is a positive solution of $\left(TW_{0}\left[c^{\star}\right]\right)$
by \lemref{Definition_c_star}) satisfies $\overline{\mathbf{p}}\geq\mathbf{u}_{0}$.
Then, defining $\overline{\mathbf{u}}:\left(t,x\right)\mapsto\overline{\mathbf{p}}\left(x-c^{\star}t\right)$,
we obtain by the nonnegativity of $\mathbf{c}$ on $\mathsf{K}$ $\left(H_{2}\right)$
\begin{align*}
\partial_{t}\overline{\mathbf{u}}-\mathbf{D}\partial_{xx}\overline{\mathbf{u}}-\mathbf{L}\overline{\mathbf{u}} & =\mathbf{0}\\
 & \geq-\mathbf{c}\left[\mathbf{u}\right]\circ\mathbf{u}\\
 & =\partial_{t}\mathbf{u}-\mathbf{D}\partial_{xx}\mathbf{u}-\mathbf{L}\mathbf{u}
\end{align*}
 and then, applying the parabolic strong maximum principle to the
operator $\partial_{t}-\mathbf{D}\partial_{xx}-\mathbf{L}$, we deduce
that $\overline{\mathbf{u}}-\mathbf{u}$ is nonnegative in $[0,+\infty)\times\mathbb{R}$.
Consequently, for all $x\in\mathbb{R}$, $t>0$ and $c>c^{\star}$,
\[
\mathbf{0}\leq\mathbf{u}\left(t,x+ct\right)\leq\overline{\mathbf{p}}\left(x+\left(c-c^{\star}\right)t\right),
\]
 and by component-wise monotonicity of $\overline{\mathbf{p}}$, for
all $y\in\mathbb{R}$ and all $x\geq y$,
\[
\mathbf{0}\leq\mathbf{u}\left(t,x+ct\right)\leq\overline{\mathbf{p}}\left(y+\left(c-c^{\star}\right)t\right),
\]
which gives the result.
\end{proof}

\subsection{Lower estimate}
\begin{prop}
\label{prop:Spreading_speed_2} Let $c\in[0,c^{\star})$ and $I\subset\mathbb{R}$
be a bounded interval. We have
\[
\left(\liminf_{t\to+\infty}\inf_{x\in I}u_{i}\left(t,x+ct\right)\right)_{i\in\left[N\right]}\in\mathsf{K}^{++}.
\]
\end{prop}

\begin{proof}
Recall \lemref{Value_of_the_generalized_principal_eigenvalue_with_drift}
and define
\[
\lambda_{c}=-\max\limits _{\mu\geq0}\left(\kappa_{\mu}+\mu c\right)>0
\]
 ($-\lambda_{c}$ being the generalized principal eigenvalue of $-\mathbf{D}\frac{\text{d}^{2}}{\text{d}x^{2}}-c\frac{\text{d}}{\text{d}x}-\mathbf{L}$)
and, using the fact that $\mathbf{c}$ vanishes at $\mathbf{0}$ $\left(H_{3}\right)$,
\[
\alpha_{c}=\max\left\{ \alpha>0\ |\ \forall\mathbf{v}\in\left[0,\alpha\right]^{N}\quad\mathbf{c}\left(\mathbf{v}\right)\leq\frac{\lambda_{c}}{2}\mathbf{1}_{N,1}\right\} .
\]
 Let $R_{c}$ be a sufficiently large radius satisfying
\[
\lambda_{1,Dir}\left(-\mathbf{D}\frac{\text{d}^{2}}{\text{d}\xi^{2}}-c\frac{\text{d}}{\text{d}\xi}-\left(\mathbf{L}-\frac{\lambda_{c}}{2}\mathbf{I}\right),\left(-R_{c},R_{c}\right)\right)<0.
\]
 Let $\mathbf{u}_{c}:\left(t,y\right)\mapsto\mathbf{u}\left(t,y+ct\right)$.
It is a solution of
\[
\partial_{t}\mathbf{u}_{c}-\mathbf{D}\partial_{yy}\mathbf{u}_{c}-c\partial_{y}\mathbf{u}_{c}=\mathbf{L}\mathbf{u}_{c}-\mathbf{c}\left[\mathbf{u}_{c}\right]\circ\mathbf{u}_{c}\text{ in }\left(0,+\infty\right)\times\mathbb{R}
\]
 with initial data $\mathbf{u}_{0}$. Just as in the proof of \propref{Persistence},
we can use $R_{c}$, $\alpha_{c}$ and Földes\textendash Polá\v{c}ik\textquoteright s
Harnack inequality \cite{Foldes_Polacik} to deduce the existence
of $\nu_{c}>0$ such that
\[
\left(\liminf_{t\to+\infty}\inf_{x\in I}u_{i}\left(t,x+ct\right)\right)_{i\in\left[N\right]}\geq\nu_{c}\mathbf{1}_{N,1}.
\]
This ends the proof.
\end{proof}
\begin{rem*}
We point out that $R_{c}\to+\infty$ as $c\to c^{\star}$. Hence the
proof above cannot be used directly to obtain a lower bound uniform
with respect to $c$. Although we expect indeed the existence of such
a bound, we do not know how to obtain it. 
\end{rem*}

\section{Estimates for the minimal wave speed}

In this section, we assume $\lambda_{PF}\left(\mathbf{L}\right)>0$,
\[
d_{1}\leq d_{2}\leq\text{\dots}\leq d_{N},
\]
and prove the estimates provided by \thmref{Characterization_minimal_speed}.
Recall the equality
\[
c^{\star}=\min_{\mu>0}\left(-\frac{\kappa_{\mu}}{\mu}\right).
\]

Recall as a preliminary that for all $r>0$ and $d>0$, the following
equality holds:
\[
2\sqrt{rd}=\min_{\mu>0}\left(\mu d+\frac{r}{\mu}\right).
\]
\begin{prop}
We have
\[
2\sqrt{d_{1}\lambda_{PF}\left(\mathbf{L}\right)}\leq c^{\star}\leq2\sqrt{d_{N}\lambda_{PF}\left(\mathbf{L}\right)}.
\]

If $d_{1}<d_{N}$, both inequalities are strict. If $d_{1}=d_{N}$,
both inequalities are equalities.
\end{prop}

\begin{proof}
Since $d_{1}\mathbf{1}_{N,1}\leq\mathbf{d}\leq d_{N}\mathbf{1}_{N,1}$,
we have, for all $\mu>0$,
\[
\mu d_{1}+\frac{1}{\mu}\lambda_{PF}\left(\mathbf{L}\right)\leq\lambda_{PF}\left(\mu\mathbf{D}+\frac{1}{\mu}\mathbf{L}\right)\leq\mu d_{N}+\frac{1}{\mu}\lambda_{PF}\left(\mathbf{L}\right),
\]
 whence we deduce
\[
2\sqrt{d_{1}\lambda_{PF}\left(\mathbf{L}\right)}\leq c^{\star}\leq2\sqrt{d_{N}\lambda_{PF}\left(\mathbf{L}\right)}.
\]
On one hand, it is well-known that if $d_{1}<d_{N}$, then the above
inequalities are strict. On the other hand, if $d_{1}=d_{N}$, we
have 
\[
\lambda_{PF}\left(\mu\mathbf{D}+\frac{1}{\mu}\mathbf{L}\right)=\mu d_{1}+\frac{1}{\mu}\lambda_{PF}\left(\mathbf{L}\right),
\]
 whence the equality. 
\end{proof}
Recall from \lemref{Definition_c_star} that $\mathbf{n}_{\mu_{c^{\star}}}=\mathbf{n}_{PF}\left(\mu_{c^{\star}}^{2}\mathbf{D}+\mathbf{L}\right)$. 
\begin{prop}
For all $i\in\left[N\right]$ such that $l_{i,i}>0$, we have
\[
c^{\star}>2\sqrt{d_{i}l_{i,i}}.
\]
\end{prop}

\begin{proof}
Let $i\in\left[N\right]$. The characterization of $c^{\star}$ (see
\lemref{Definition_c_star}) yields
\[
\mu_{c^{\star}}d_{i}+\frac{l_{i,i}}{\mu_{c^{\star}}}=c^{\star}-\frac{1}{\mu_{c^{\star}}}\sum_{j\in\left[N\right]\backslash\left\{ i\right\} }l_{i,j}\frac{n_{\mu_{c^{\star}},j}}{n_{\mu_{c^{\star}},i}},
\]
 whence, if $l_{i,i}>0$,
\[
c^{\star}\geq2\sqrt{d_{i}l_{i,i}}+\frac{1}{\mu_{c^{\star}}}\sum_{j\in\left[N\right]\backslash\left\{ i\right\} }l_{i,j}\frac{n_{\mu_{c^{\star}},j}}{n_{\mu_{c^{\star}},i}}.
\]
 From the irreducibility and essential nonnegativity of $\mathbf{L}$
$\left(H_{1}\right)$, there exists $j\in\left[N\right]\backslash\left\{ i\right\} $
such that $l_{i,j}>0$, whence $c^{\star}>2\sqrt{d_{i}l_{i,i}}$.
\end{proof}
Recall the existence of a unique decomposition of $\mathbf{L}$ of
the form
\[
\mathbf{L}=\text{diag}\mathbf{r}+\mathbf{M}\text{ with }\mathbf{r}\in\mathbb{R}^{N}\text{ and }\mathbf{M}^{T}\mathbf{1}_{N,1}=\mathbf{0}.
\]
\begin{rem*}
Regarding the Lotka\textendash Volterra mutation\textendash competition\textendash diffusion
ecological model, the decomposition $\mathbf{L}=\text{diag}\mathbf{r}+\mathbf{M}$
is ecological meaningful: $\mathbf{r}$ is the vector of the growth
rates of the phenotypes whereas $\mathbf{M}$ describes the mutations
between the phenotypes. 
\end{rem*}
\begin{prop}
Let $\left(\left\langle d\right\rangle ,\left\langle r\right\rangle \right)\in\left(0,+\infty\right)\times\mathbb{R}$
be defined as
\[
\left\{ \begin{matrix}\left\langle d\right\rangle =\frac{\mathbf{d}^{T}\mathbf{n}_{PF}\left(\mu_{c^{\star}}^{2}\mathbf{D}+\mathbf{L}\right)}{\mathbf{1}_{1,N}\mathbf{n}_{PF}\left(\mu_{c^{\star}}^{2}\mathbf{D}+\mathbf{L}\right)},\\
\left\langle r\right\rangle =\frac{\mathbf{r}^{T}\mathbf{n}_{PF}\left(\mu_{c^{\star}}^{2}\mathbf{D}+\mathbf{L}\right)}{\mathbf{1}_{1,N}\mathbf{n}_{PF}\left(\mu_{c^{\star}}^{2}\mathbf{D}+\mathbf{L}\right)}.
\end{matrix}\right.
\]
If $\left\langle r\right\rangle \geq0$, then
\[
c^{\star}\geq2\sqrt{\left\langle d\right\rangle \left\langle r\right\rangle }.
\]
\end{prop}

\begin{proof}
Using $\left(\mathbf{r},\mathbf{M}\right)$, the characterization
of $c^{\star}$ (see \lemref{Definition_c_star}) is rewritten as
\[
\left(\mu_{c^{\star}}^{2}\mathbf{D}+\text{diag}\mathbf{r}\right)\mathbf{n}_{\mu_{c^{\star}}}+\mathbf{M}\mathbf{n}_{\mu_{c^{\star}}}=\mu_{c^{\star}}c^{\star}\mathbf{n}_{\mu_{c^{\star}}}.
\]
 Summing the lines of this system, dividing by $\sum\limits _{i=1}^{N}n_{\mu_{c^{\star}},i}$
and defining $\left\langle d\right\rangle $ and $\left\langle r\right\rangle $
as in the statement, we find 
\[
\mu_{c^{\star}}^{2}\left\langle d\right\rangle +\left\langle r\right\rangle =\mu_{c^{\star}}c^{\star}.
\]
 The equation $\left\langle d\right\rangle \mu^{2}-c^{\star}\mu+\left\langle r\right\rangle =0$
admits a real positive solution $\mu$ if and only if $\left(c^{\star}\right)^{2}-4\left\langle d\right\rangle \left\langle r\right\rangle \geq0$.
\end{proof}

\section*{Acknowledgments}

The author thanks Grégoire Nadin for the attention he paid to this
work, Vincent Calvez for fruitful discussions on the cane toads equation,
Cécile Taing for pointing out the related work by Wang and an anonymous
reviewer for a detailed report thanks to which the original manuscript
was largely improved. 

\bibliographystyle{plain}
\bibliography{ref}

\end{document}